\documentclass[10pt]{amsart}
\usepackage{amsmath, amsthm, amssymb}
\usepackage{graphicx}
\usepackage{mathbbol}
\usepackage{txfonts}
\usepackage[usenames]{color}

\newtheorem{thm}{Theorem}[section]
\newcommand{\bt}{\begin{thm}}
\newcommand{\et}{\end{thm}}

\newtheorem{ex}[thm]{Example}

\newtheorem{cor}[thm]{Corollary}   
\newcommand{\bc}{\begin{cor}}
\newcommand{\ec}{\end{cor}}

\newtheorem{lem}[thm]{Lemma}   
\newcommand{\bl}{\begin{lem}}
\newcommand{\el}{\end{lem}}

\newtheorem{prop}[thm]{Proposition}
\newcommand{\bp}{\begin{prop}}
\newcommand{\ep}{\end{prop}}

\newtheorem{defn}[thm]{Definition}
\newcommand{\bd}{\begin{defn}}    
\newcommand{\ed}{\end{defn}}

\newtheorem{rmrk}[thm]{Remark}   

\newcommand{\br}{\begin{rmrk}}
\newcommand{\er}{\end{rmrk}}

\newcommand{\weaklyto}{\to}

\newtheorem{example}[thm]{Example}

\newcommand{\GHto}{\stackrel { \textrm{GH}}{\longrightarrow} }
\newcommand{\Fto}{\stackrel {{\tiny{\mathcal{F}}}}{\longrightarrow} }

\newcommand{\sgn}{\operatorname{sgn}}

\newcommand{\be}{\begin{equation}}

 \newcommand{\ee}{\end{equation}}

\newcommand{\injrad}{{\textrm{injrad}}}

\newcommand{\N}{\mathbb{N}}

\newcommand{\R}{\mathbb{R}}
\newcommand{\One}{{\bf \rm{1}}}

\newcommand{\E}{\mathbb{E}}

\newcommand{\Z}{\mathbb{Z}}

\newcommand{\diam}{\operatorname{Diam}}
\newcommand{\Sphere}{\operatorname{Sphere}}
\newcommand{\Slice}{\operatorname{Slice}}

\newcommand{\Fm}{{\mathcal F}}

\newcommand{\lm}{{\mathcal L}}

\newcommand{\set}{{\rm{set}}}

\newcommand{\disjointunion}{\sqcup}

\newcommand{\Lip}{\operatorname{Lip}}

\newcommand{\mass}{{\mathbf M}}
\newcommand{\SF}{{\mathbf {SF}}}
\newcommand{\IFV}{{\mathbf{IFV}_\epsilon}}
\newcommand{\SIF}{{\mathbf{SIF}_\epsilon}}

\newcommand{\curr}{{\mathbf M}}         
\newcommand{\rectcurr}{{\mathcal R}}    
\newcommand{\intrectcurr}{{\mathcal I}} 
\newcommand{\intcurr}{{\mathbf I}}      
\newcommand{\ess}{\operatorname{ess}}

\newcommand{\vol}{\operatorname{Vol}}
\newcommand{\nmass}{\mathbf N}

\newcommand{\fillvol}{{\operatorname{FillVol}}}

\newcommand{\rstr}{\:\mbox{\rule{0.1ex}{1.2ex}\rule{1.1ex}{0.1ex}}\:}
\newcommand{\bdry}{\partial}
\newcommand{\spt}{\operatorname{spt}}

\begin{document}

\title{Properties of the Intrinsic Flat Distance}

\author{J. Portegies}
\thanks{Portegies partially supported by Max Planck Institute for Mathematics in the Sciences
\\ \indent and by Sormani's NSF grant: DMS 1309360.}
\address{Max Planck Institute for Math. in the Sciences,
Inselstra{\ss}e 22, 
04103 Leipzig,
Germany}
\email{jacobus.portegies@mis.mpg.de}

\author{C. Sormani}
\thanks{Sormani partially supported by NSF DMS 1006059 and a PSC CUNY Research Grant.}
\address{CUNY Graduate Center and Lehman College,
365 Fifth Avenue, NY NY 10016, USA}
\email{sormanic@gmail.com}

\date{June 2015}




\begin{abstract}
Here we explore a variety of properties of intrinsic flat convergence.
We introduce the sliced filling volume and interval sliced filling volume
 and explore the relationship between these notions, the tetrahedral property 
 and the disappearance of points under intrinsic flat convergence. 
We prove two new Gromov-Hausdorff
and intrinsic flat compactness theorems including the
Tetrahedral Compactness Theorem.   
Much of the work in this paper builds upon Ambrosio-Kirchheim's
 Slicing Theorem combined with an adapted version Gromov's Filling Volume.
\end{abstract}

\maketitle

\newpage

\tableofcontents

\newpage

\section{Introduction}

The intrinsic flat convergence of Riemannian manifolds has
been applied to study stability of the Positive Mass Theorem,
rectifiability of Gromov-Hausdorff limits of Riemannian manifolds,
and smooth convergence away from singular sets.     
Applications of intrinsic flat convergence
to Riemannian General Relativity appear in joint work of the second author with Lan-Hsuan Huang, Dan Lee, and Philippe LeFloch  \cite{LeeSormani1}\cite{LeFlochSormani1}\cite{HLS}. 
Sajjad Lakzian has applied intrinsic flat convergence to study
smooth convergence away from singular sets obtaining results about the 
limits of Kahler manifolds in \cite{Lakzian-diameter}
and Ricci flow through singularities in \cite{Lakzian-Ricci-flow}. 
Other potential applications of intrinsic flat convergence have 
been suggested by Misha Gromov in \cite{Gromov-Plateau}.

The initial notions of
the intrinsic flat distance and integral current spaces appeared in joint 
work of the second author with Wenger \cite{SorWen2}  building upon Ambrosio-Kirchheim's 
important work on currents in metric spaces \cite{AK}. 
Here we explore new properties and their relationship with
intrinsic flat convergence building upon 
Ambrosio-Kirchheim's Slicing Theorem \cite{AK} (see Theorem~\ref{theorem-slicing} ) combined with a slightly adaped version of Gromov's
Filling Volume \cite{Gromov-Filling} (see Definition~\ref{defn-filling-volume}).
These ideas were intuitively applied in prior work of the author with Wenger to
prove continuity of the filling volumes of spheres under intrinsic flat convergence
and prevent the disappearance of points under intrinsic flat convergence \cite{SorWen1}.   Recall that under intrinsic flat convergence, points may
disappear in the limit.  In fact the limit space could simply be the $\bf{0}$ space
and one must try to avoid this in most applications.

In this paper, we use the full iterative strength of Ambrosio-Kirchheim's
Slicing Theorem to introduce and study the {\em Sliced Filling Volume} [Definitions~\ref{sliced-filling-vol} and~\ref{defn-SF_k}],
the {\em Interval Filling Volume} [Definition~\ref{IFV}], and the {\em Sliced Interval
Filling Volume} [Definition~\ref{SIF}].   We prove the sliced filling volume is bounded
below by constants in the {\em Tetrahedral Property} and the {\em Integral Tetrahedral Property} (see Definitions~\ref{defn-tetra} and~\ref{defn-int-tetra} and Theorem~\ref{tetra-ball}).   The three
dimensional version of the tetrahedral property appears in 
(\ref{3D-1})-(\ref{3D-2}) and is depicted in Figure 1.
Note that some of these notions were first announced by the second author in \cite{Sormani-tetra}.

We prove continuity of the Sliced Filling Volumes with respect to
intrinsic flat convergence in Theorem~\ref{SF-cont}.   We prove continuity
of the Interval Filling Volumes and Sliced Interval Filling Volumes in Theorems~\ref{IFV-cont} and~\ref{SIF-cont}.   
The first author proved semicontinuity of eigenvalues under
volume preserving intrinsic flat convergence in \cite{Portegies-evalues}.   Here we do not make any assumptions on the preservation of volume in the limit.

We then use the
notion of the sliced filling volume to explore when a point does not disappear
under intrinsic flat convergence [Theorems~\ref{SF_k-in-set}, ~\ref{B-W-fillvol}
and~\ref{B-W}].   Note that the disappearance of points was also studied in
prior work of the second author \cite{Sormani-AA}. However, in that paper, one
could not determine if a sequence of points converged to a limit point that was only in the metric completion of the limit space.   Here we are able to determine if the limit of the
points lies in the intrinsic flat limit itself.   Theorems~\ref{B-W-fillvol} and~\ref{B-W}
are Bolzano-Weierstrass type theorems, producing converging subsequences of points.

This paper culminates with two compactness theorems: the Sliced Filling Compactness Theorem
[Theorem~\ref{SF_k-compactness-2}] and the
Tetrahedral Compactness Theorem [Theorem~\ref{tetra-compactness-2}].    We state the three dimensional version of the Tetrahedral Compactness 
Theorem here (including the three dimensional Tetrahedral
Property within the statement):

\newpage

\begin{thm} \label{tetra-compactness-3D}
Given $r_0>0, \beta\in (0,1), C>0, V_0>0$.  Suppose a sequence of 
Riemannian manifolds, $M_i^3$,  satisfies the 
{\bf \em $C, \beta$ tetrahedral  property}
for all balls, $B_p(r) \subset M_i^3$, of radius $r\le r_0$
as in Figure~\ref{fig-tetra-prop}.  That is,
\be\label{3D-1}
\exists p_1, p_2\in \partial B_p(r) \textrm{ such that }
\forall t_1, t_2 \in [(1-\beta)r, (1+\beta)r] \textrm{ we have }\\
\ee
\be\label{3D-2}
\inf\{d(x,y): \, x\neq y, \, x,y \in \partial B_p(r)\cap \partial B_{p_1}(t_1)
\cap \partial B_{p_2}(t_2) \} \in [Cr, \infty).
\ee
Assume in addition each $M_i$ has
\be
\vol(M_i^3)\le V_0 \textrm{ and } \diam(M_i^3)\le D_0.
\ee
Then a subsequence of the $M_i$ converges in
the Gromov-Hausdorff and the Intrinsic Flat sense
to the same space.   In particular, the limit is countably
$\mathcal{H}^3$ rectifiable.   
\end{thm} 

\begin{figure}[htbp] 
   \centering
   \includegraphics[width=3in]{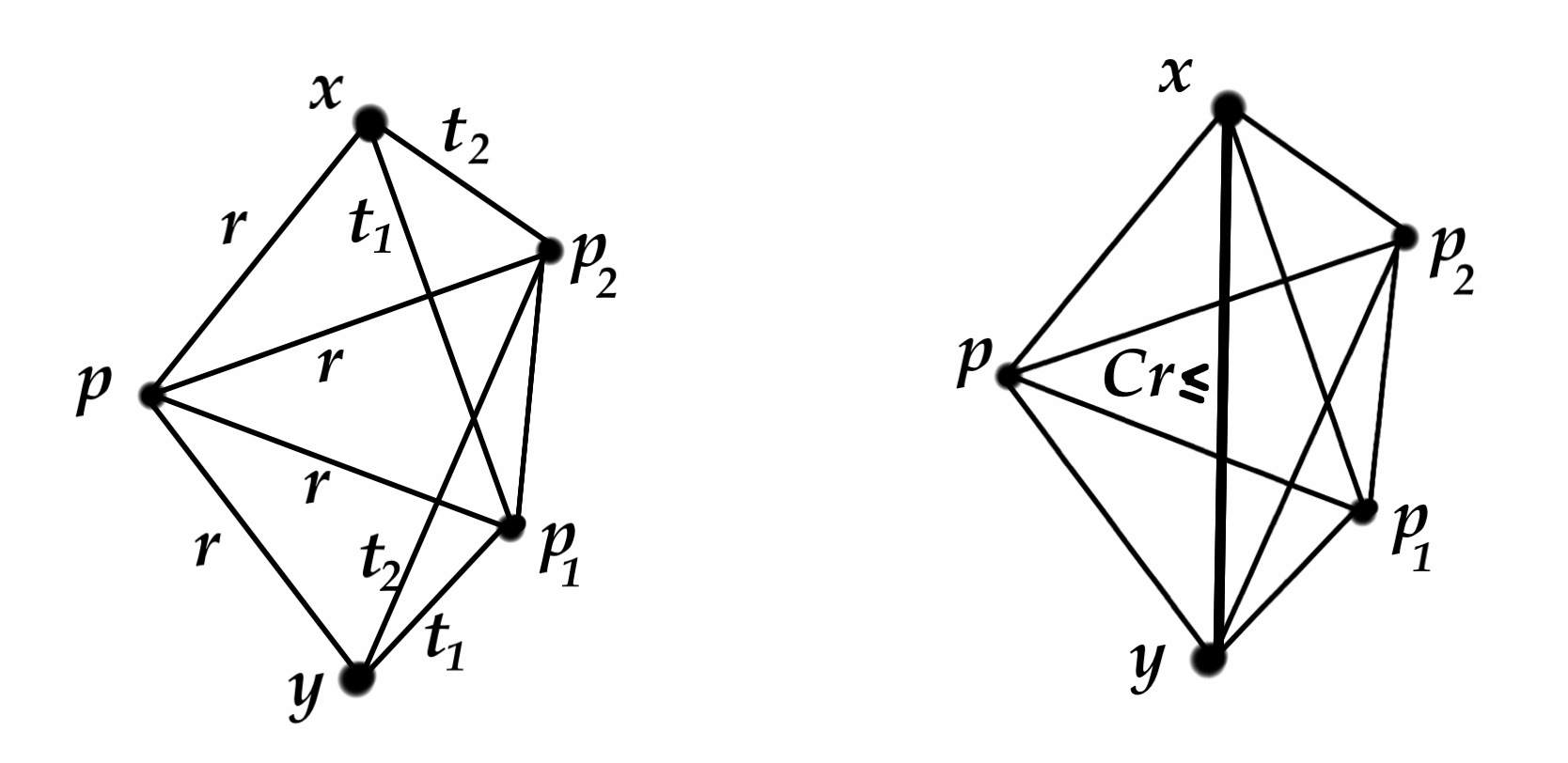} 
   \caption{Tetrahedral Property as depicted in \cite{Sormani-tetra}. }
   \label{fig-tetra-prop}
\end{figure}

One might view this compactness theorem as a higher dimensional analogue of the compactness of Alexandrov spaces.
The Sliced Filling Compactness Theorem is applied to prove this Tetrahedral Compactness Theorem.    It assumes a uniform lower bound on the sliced filling volumes of balls and draws the same conclusion.   To prove this theorem, we first prove 
Gromov-Hausdorff convergence of a subsequence [Theorem~\ref{SF_k-compactness}].   We only obtain the fact that the
intrinsic flat limit agrees with the Gromov-Hausdorff limit in the
final section of the paper by applying our theorems which avoid the disappearance of points.      Once the two notions of convergence agree, then we can
conclude the limits are noncollapsed   
countably $\mathcal{H}^m$ rectifiable metric spaces.

These theorems
were announced by the second author in \cite{Sormani-tetra}
but the rigorous proof has required the development of the full theory of sliced filling volumes developed herein.  
Prior results relating 
intrinsic flat limits to Gromov-Hausdorff limits
appear in joint work of the second author with Wenger 
concerning sequences of spaces with contractibility functions and
noncollapsing manifolds with nonnegative Ricci curvature \cite{SorWen1},
in work of Li-Perales concerning Alexandrov spaces \cite{Li-Perales},
in work of Munn concerning noncollapsing manifolds with pinched
Ricci curvature \cite{Munn-GH=IF}, and in work of Perales concerning
noncollapsing Riemannian manifolds with boundary \cite{Perales-GH=IF}.    These prior results apply powerful theorems from
Cheeger-Colding Theory and Alexandrov Geometry.   The results contained herein are built only upon the theorems in 
Ambrosio-Kirchheim's
{\em ``Currents on Metric Spaces"} \cite{AK}
and the ideas in Gromov's {\em ``Filling Riemannian Manifolds"}
 \cite{Gromov-Filling}.

Applications of the results in this paper will appear in future work of the authors.

\vspace{.4cm}
\newpage
\noindent{\bf{Recommended Reading:}}

\vspace{.2cm}

This paper attempts to be completely self contained, providing all necessary
background material within the paper.   However, students reading this 
paper are encouraged to read Burago-Burago-Ivanov's award winning textbook \cite{BBI} 
which provides a thorough background in Gromov-Hausdorff convergence 
and also to read the second author's joint paper with Wenger \cite{SorWen2} and 
the second author's recent paper \cite{Sormani-AA}.   Those who
would like to understand the Geometric Measure Theory more
deeply should read Morgan's textbook \cite{Morgan-text} or
Fanghua Lin's textbook \cite{Lin-Yang-GMT-text} and then the work of
Ambrosio-Kirchheim \cite{AK}.   

\vspace{.4cm}

\noindent{\bf{Acknowledgements:}}

\vspace{.2cm}
We are grateful to Luigi Ambrosio, Toby Colding, Jozef Dodziuk,
Carolyn Gordon, Misha Gromov, Gerhard Huisken, Tom Ilmanen, Jim Isenberg, J\"urgen Jost, Blaine Lawson, Fanghua Lin, Bill Minicozzi, Paul Yang and Shing-Tung Yau for their interest in the Intrinsic Flat distance and invitations to visit their universities and to 
attend Oberwolfach.  
We discussed many interesting applications and we hope this paper
provides the properties required to explore these possibilities. 
We appreciate Maria Gordina for early assistance with some
of the details regarding metric measure spaces.  We would especially like to thank the
Jorge Basilio, Sajjad Lakzian and Raquel Perales for 
actively participating in the CUNY Geometric Measure Theory Reading Seminar 
with us while this paper was first being written.   Their 
presentations of Ambrosio-Kirchheim's work and deep questions
lead to many interesting ideas.  
The first author would like to thank the Max Planck Institute for Mathematics in the Sciences for its hospitality.
The second author would like to thank Cavalletti, Ketterer
and Munn for the invitation to present this paper in a series of talks at the
{\em Winter School for Optimal Transport} at the Hausdorff Institute in Bonn.  This
lead to many exciting discussions with Shouhei Honda, Nicola Gigli, Yashar Memarian
and Tapio Rajala concerning potential future applications combining this work with their
results. 

\newpage

\section{Background}
In this section we review the Gromov-Hausdorff distance
introduced by Gromov in \cite{Gromov-metric}, then various
topics from Ambrosio-Kirchheim's work in \cite{AK},
then intrinsic flat convergence and integral current
spaces from prior work of the second author with Wenger
in \cite{SorWen2} and end
with a review of filling volumes which are related to
Gromov's notion from \cite{Gromov-Filling} but defined using
the work of Ambrosio-Kirchheim.

\vspace{.4cm}

\subsection{Review of the Gromov-Hausdorff Distance}

First recall that $\varphi: X \to Y$ is an isometric embedding iff
\be
d_Y(\varphi(x_1), \varphi(x_2)) = d_X(x_1, x_2) \qquad \forall x_1, x_2 \in X.
\ee
This is referred to as a metric isometric embedding in \cite{LeeSormani1}
and it should be distinguished from a Riemannian isometric embedding.

\begin{defn}[Gromov]\label{defn-GH} 
The Gromov-Hausdorff distance between two 
compact metric spaces $\left(X, d_X\right)$ and $\left(Y, d_Y\right)$
is defined as
\be \label{eqn-GH-def}
d_{GH}\left(X,Y\right) := \inf  \, d^Z_H\left(\varphi\left(X\right), \psi\left(Y\right)\right)
\ee
where $Z$ is a complete metric space, and $\varphi: X \to Z$ and $\psi:Y\to Z$ are
isometric embeddings and where the Hausdorff distance in $Z$ is defined as
\be
d_{H}^Z\left(A,B\right) = \inf\{ \epsilon>0: A \subset T_\epsilon\left(B\right) \textrm{ and } B \subset T_\epsilon\left(A\right)\}.
\ee
\end{defn}

Gromov proved that this is indeed a distance on compact metric spaces: $d_{GH}\left(X,Y\right)=0$
iff there is an isometry between $X$ and $Y$.   When studying metric
spaces which are only precompact, one may take their metric
completions before studying the Gromov-Hausdorff distance
between them.

\begin{defn}\label{defn-equibounded}
A collection of metric spaces is said to be equibounded or
uniformly bounded if there is a uniform upper bound
on the diameter of the spaces.  
\end{defn}

\begin{rmrk}\label{rmrk-NXr}
 We will write $N\left(X,r\right)$ to denote the
number of disjoint balls of radius $r$ in a space $X$.   Note that $X$ can always be covered
by $N\left(X,r\right)$ balls of radius $2r$.
\end{rmrk}

Note that Ilmanen's Example of \cite{SorWen2}
of a sequence of spheres with increasingly many splines is not
equicompact, as the number of balls centered on the tips approaches infinity.

\begin{defn} \label{defn-equicompact}
A collection of spaces is said to
be equicompact or uniformly compact if they 
are have a common upper bound $N\left(r\right)$
such that $N\left(X,r\right) \le N\left(r\right)$ for all spaces $X$ in the collection.
\end{defn}

Gromov's Compactness Theorem states that sequences of equibounded and equicompact
metric spaces have a Gromov-Hausdorff converging subsequence \cite{Gromov-metric}.  In fact,
Gromov proves a stronger version of this statement in \cite{Gromov-groups}:

\begin{thm}[Gromov's Compactness Theorem] \label{Thm-Gromov} 
If a sequence of compact metric spaces, $X_j$, is equibounded and equicompact, then a subsequence of the $X_j$ converges to a compact metric space $X_\infty$.
\end{thm}

Gromov also proved the following useful theorem:

\begin{thm}\label{Gromov-Z}
If a sequence of compact metric
spaces $X_j$ converges to a compact metric space $X_\infty$ then
$X_j$ are equibounded and equicompact.  Furthermore,
there is a compact metric space, $Z$, and isometric embeddings
$\varphi_j: X_j \to Z$ such that
\be
d_H^Z\left(\varphi_{j}(X_{j}), \varphi_\infty(X_\infty)\right)  
\le 2d_{GH}\left(X_j, X_\infty\right) \to 0.
\ee
\end{thm}

This theorem allows one to define converging sequences of
points: 

\begin{defn}\label{Gromov-points}
We say that $x_j \in X_j$ converges to $x_\infty \in X_\infty$,
if there is a common space $Z$ as in Theorem~\ref{Gromov-Z}
such that $\varphi_j(x_j) \to \varphi_\infty(x)$ as points in
$Z$.  If one discusses the limits of multiple sequences of points
then one uses a common $Z$ to determine the convergence
to avoid difficulties arising from isometries in the limit space.
Then one immediately has
\be
\lim_{j\to\infty} d_{X_j}(x_j, x_j')=d_{X_\infty}(x_\infty, x_\infty')
\ee
whenever $x_j \to x_\infty$ and $x_j'\to x_\infty'$ via
a common $Z$.
\end{defn}

Theorem~\ref{Gromov-Z} also allows one to extend the
Arzela-Ascoli Theorem:

\begin{defn}\label{equicont-def}
A collection of functions, $f_j : X_j \to X_j'$ is said to be equicontinuous
if for all $\epsilon>0$ there exists $\delta_\epsilon>0$ independent of
$j$ such that 
\be
f_j \left(B_x(\delta_\epsilon)\right)\subset B_{f_j(x)}(\epsilon) \qquad \forall x \in X_j.
\ee
\end{defn}

\begin{thm}\label{Gromov-Arz-Asc}
Suppose $X_j$ and $X_j'$ are compact metric spaces converging
in the Gromov-Hausdorff sense to compact metric spaces
$X_\infty$ and $X_\infty'$, and suppose $f_j : X_j \to X_j'$
are equicontinuous, then a subsequence $f_{j_i}$
converge to a continuous function $f_\infty: X_\infty\to X_\infty'$
such that for any sequence $x_j \to x_\infty$ via a common $Z$
we have $f_{j_i}(x_{j_i}) \to f_\infty(x_\infty)$.
\end{thm}

In particular, one can define limits of curves $C_i: [0,1] \to X_i$
(parametrized proportional to arclength with a uniform upper
bound on length) to obtain curves $C_\infty:[0,1]\to X_\infty$.
So that when $X_i$ are compact length spaces whose distances
are achieved by minimizing geodesics, so are the limit spaces $X_\infty$.

One only needs uniform lower bounds on 
Ricci curvature and upper bounds on diameter to prove equicompactness
on a sequence of Riemannian manifolds.  This is a consequence of
the Bishop-Gromov Volume Comparison Theorem \cite{Gromov-metric}.
Colding and Cheeger-Colding have studied the limits of such sequences 
of spaces proving volume convergence and eigenvalue convergence
and many other interesting properties \cite{Colding-volume}
\cite{ChCo-PartI}-\cite{ChCo-PartIII}.   
 One property of particular interest here, is that 
when the sequence of manifolds is noncollapsing (i.e. is assumed to have
a uniform lower bound on volume), Cheeger-Colding prove that the limit space
is countably $\mathcal{H}^m$ rectifiable with the same dimension
as the sequence \cite{ChCo-PartI}.

Greene-Petersen have shown that conditions on contractibility
and uniform upper bounds on diameter also suffice to achieve
compactness without any assumption on Ricci curvature or
volume \cite{Greene-Petersen}.   Sormani-Wenger have shown
that if one assumes a uniform linear contractibility function on the
sequence of manifolds then the limit spaces achieved in their setting are
also countably $\mathcal{H}^m$ rectifiable with the same dimension
as the sequence.   Without the assumption of linearity, Schul-Wenger
have provided an example where the Gromov-Hausdorff limit
is not countably $\mathcal{H}^m$ rectifiable. \cite{SorWen1}.   The
proofs here involve the Intrinsic Flat Convergence.

\vspace{.4cm}
\subsection{Review of Ambrosio-Kirchheim Currents on Metric Spaces}

In \cite{AK}, Ambrosio-Kirchheim extend Federer-Fleming's notion of integral
currents using DiGeorgi's notion of tuples of functions.   Here we review
their ideas.   Here $Z$ denotes a complete metric space.

In Federer-Fleming currents were defined as linear functionals on differential
forms \cite{FF}. This approach extends naturally
to smooth manifolds but not to complete metric spaces which do not have differential
forms.  In the place of differential forms, Ambrosio-Kirchheim use
DiGeorgi's
$m+1$ tuples, $\omega\in \mathcal{D}^m(Z)$,
\be
\omega=\left(f,\pi_1 ...\pi_m\right) \in \mathcal{D}^m(Z)
\ee 
where
$f: X \to \R$ is a bounded Lipschitz function and
$\pi_i: X \to \R$ are Lipschitz.  

In \cite{AK} Definitions 2.1, 2.2, 2.6 and 3.1, an $m$ dimensional
current $T\in \curr_m(Z)$ is defined.  Here we combine them into a single
definition:

\begin{defn}
On a complete metric space, $Z$,  an
$m$ dimensional {\bf \em current}, denoted
$T\in \curr_m(Z)$, is a real valued
{\em multilinear functional} on $\mathcal{D}^m(Z)$, with the
following three required properties:
\vspace{.2cm}

{\bf i) { Locality}:}
$$
T(f, \pi_1,...\pi_m)=0 \textrm{ if }\exists i\in \{1,...m\} \textrm{ s.t. }\pi_i
\textrm{ is constant on a nbd of } \{f\neq0\}.
$$

{\bf ii) { Continuity}:}
$$
\textcolor{white}{\int}\textrm{Continuity of $T$ with respect to the ptwise convergence
of $\pi_i$ such that $\Lip(\pi_i)\le 1$.} 
$$

{\bf iii) {Finite mass}: }
$$
\textcolor{white}{\int}\exists \textrm{ finite Borel } \mu \,\,\,s.t. \,\,\,
|T(f,\pi_1,...\pi_m)| \le \prod_{i=1}^m \Lip(\pi_i)  \int_Z |f| \,d\mu \,\,\, \forall (f,\pi_1,...\pi_m)\in \mathcal{D}^m(Z).
$$
\end{defn}

In \cite{AK} Definition 2.6 Ambrosio-Kirchheim introduce
their mass measure which is distinct from the masses used
in work of Gromov \cite{Gromov-Filling} and 
Burago-Ivanov \cite{Burago-Ivanov-Vol-Tori}.   This definition
is later used to define the notion of filling volume used in this 
paper.

\begin{defn}  \label{defn-mass}  
The mass measure, $\|T\| $,
of a current, $T\in \curr_m(Z)$, is the smallest Borel measure $\mu$ such that
\be \label{def-measure-T}
  \Big|T\left(f,\pi\right)\Big|    \le     \int_X |f| d\mu
  \qquad  \forall  \, \left(f,\pi\right) \textrm{ where } \Lip\left(\pi_i\right)\le 1.
\ee  
The mass of $T$ is defined
\be \label{def-mass-from-current}
M\left(T\right) = || T || \left(Z\right) = \int_Z \, d\| T\|.
\ee
\end{defn}

In particular
\be \label{eqn-mass}
\Big| T(f,\pi_1,...\pi_m) \Big| \le \mass(T)\, |f|_\infty \Lip(\pi_1) \cdots \Lip(\pi_m).
\ee

Stronger versions of locality and continuity, as well as product and
chain rules are  proven in \cite{AK}[Theorem 3.5].  In particular they
define $T(f, \pi_1,..., \pi_m)$ for $f$ which are only Borel functions
as limits of $T(f_j, \pi_1,...,\pi_m)$ where $f_j$ are bounded Lipschitz
functions converging to $f$ in $L^1(E, ||T||)$.  They also prove
 \be
T(f, \pi_\sigma(1),...\pi_\sigma(m))= \sgn(\sigma) \,T(f, \pi_1,...\pi_m)
\ee
for any permutation, $\sigma$, of $\{1,2,...m\}$.

Ambrosio-Kirchheim then define restriction \cite{AK}[Defn 2.5]:

\begin{defn} 
The {\em restriction} $T\rstr \omega\in \curr_m(Z)$
of a current $T\in M_{m+k}(Z)$ by a $k+1$ tuple
 $\omega=(g,\tau_1,...\tau_k)\in \mathcal{D}^k(Z)$:
\be
(T\rstr\omega)(f,\pi_1,...\pi_m):=T(f\cdot g, \tau_1,...\tau_k, \pi_1,...\pi_m).
\ee
Given a Borel set, $A$, 
\be
T\rstr A := T\rstr \omega
\ee
where $\omega= \One_A \in \mathcal{D}^0(Z)$ is the indicator
function of the set.  In this case,
\be
\mass(T\rstr \omega) = ||T||(A).
\ee
\end{defn}

Ambrosio-Kirchheim then define the push forward map:

\begin{defn}
Given a Lipschitz map $\varphi:Z\to Z'$, the  push
forward of a current $T\in \curr_m(Z)$
to a current $\varphi_\# T \in \curr_m(Z')$ is given in \cite{AK}[Defn 2.4] by
\be \label{def-push-forward}
\varphi_\#T(f,\pi_1,...\pi_m):=T(f\circ \varphi, \pi_1\circ\varphi,...\pi_m\circ\varphi)
\ee
exactly as in the smooth setting. 
\end{defn}

\begin{rmrk} \label{rstr-push} 
One immediately sees that 
\be
(\varphi_\#T) \rstr (f, \pi_1,...\pi_k))= \varphi_\#(T \rstr (f\circ \varphi, \pi_1\circ\varphi,...\pi_k\circ\varphi) )
\ee
and
\be
(\varphi_\#T )\rstr A = (\varphi_\#T) \rstr (\One_A)
=\varphi_\# (T \rstr (\One_A\circ \varphi)) = \varphi_\#(T \rstr \varphi^{-1}(A)).
\ee
\end{rmrk}

In (2.4) \cite{AK}, Ambrosio-Kirchheim show that
\be  \label{mass-push}
||\varphi_\#T|| \le [\Lip(\varphi)]^m \varphi_\# ||T||,
\ee
so that when $\varphi$ is an isometric
embedding 
\be \label{lem-push-mass}
||\varphi_\#T||=\varphi_\#||T|| \textrm{ and }
\mass(T)=\mass(\varphi_\#T).
\ee

The simplest example of a current is:

\begin{example}\label{basic-current-pushed}
If one has a bi-Lipschitz map, $\varphi:\R^m \to Z$, and 
a Lebesgue function $h\in L^1(A,\Z)$ where $A\in \R^m$ is Borel,
then $\varphi_\# \Lbrack h \Rbrack \in \curr_m(Z)$ an 
$m$ dimensional current in $Z$.  Note that
\be
\varphi_\# \Lbrack h \Rbrack (f,\pi_1,...\pi_m)=
\int_{A \subset \R^m} (h\circ \varphi )(f\circ\varphi) \, 
d(\pi_1\circ \varphi) \wedge \dots \wedge d(\pi_m\circ\varphi)
\ee
where $d(\pi_i\circ\varphi)$ is well defined almost everywhere
by Rademacher's Theorem.  Here the mass measure is
\be
||\Lbrack h \Rbrack ||= h \,d\mathcal{L}_m
\ee
and the mass is
\be
\mass(\Lbrack h \Rbrack ) =\int_A h d\mathcal{L}_m.
\ee
\end{example}

In \cite{AK}[Theorem 4.6] Ambrosio-Kirchheim define a canonical set associated with any
integer rectifiable current:   

\begin{defn} \label{defn-set}
The (canonical) set of a current, $T$,
 is the collection of points in $Z$ with positive lower density:
\be \label{def-set-current}
\set\left(T\right)= \{p \in Z: \Theta_{*m}\left( \|T\|, p\right) >0\},
\ee
where the definition of lower density is
\be \label{eqn-lower-density}
\Theta_{*m}\left( \mu, p\right) =\liminf_{r\to 0} \frac{\mu(B_p(r))}{\omega_m r^m}.
\ee
\end{defn}

In \cite{AK} Definition 4.2 and Theorems 4.5-4.6, an integer rectifiable current is defined using the
Hausdorff measure, $\mathcal{H}^m$:

\begin{defn}\label{int-rect-curr}
Let $m\ge 1$.   A current, $T\in \mathcal{D}_m(Z)$, is rectifiable if $\set(T)$ is countably
$\mathcal{H}^m$ rectifiable and if $||T||(A)=0$ for any set
$A$ such that $\mathcal{H}^k(A)=0$.   We write $T\in \rectcurr_m(Z)$.

We say $T\in \rectcurr_m(Z)$ is integer rectifiable,
denoted $T\in \intrectcurr_m(Z)$, if for any $\varphi\in \Lip(Z, \R^m)$ and any open set $A\in Z$, we have 
\be
\exists\, \theta \in \mathcal{L}^1(\R^k,Z) \,\,\,s.t.\,\,\,
\varphi_{\#}(T\rstr A)=\Lbrack \theta \Rbrack.
\ee   
In fact, $T \in \intcurr_m(Z)$
iff it has a parametrization.  A parametrization $\left(\{\varphi_i\}, \{\theta_i\}\right)$ of an integer rectifiable current $T\in \intrectcurr^m\left(Z\right)$ is a collection of
bi-Lipschitz maps $\varphi_i:A_i \to Z$ with $A_i\subset\R^m$ precompact
Borel measurable and with pairwise disjoint images and
weight functions $\theta_i\in L^1\left(A_i,\N\right)$ such that
\be\label{param-representation}
T = \sum_{i=1}^\infty \varphi_{i\#} \Lbrack \theta_i \Rbrack \quad\text{and}\quad \mass\left(T\right) = \sum_{i=1}^\infty \mass\left(\varphi_{i\#}\Lbrack \theta_i \Rbrack\right).
\ee
A $0$ dimensional rectifiable current is defined by the existence of
countably many distinct points, $\{x_i\}\in Z$,  weights $\theta_i \in \R^+$
and orientation, $\sigma_i \in \{-1,+1\}$
such that 
\be \label{0-param-representation}
T(f)=\sum_h \sigma_i \theta_i f(x_i) \qquad \forall f \in \mathcal{B}^\infty(Z).
\ee
where $\mathcal{B}^\infty(Z)$ is the class of bounded Borel
functions on $Z$ and where
\be
\mass(T)=\sum_h \theta_i< \infty
\ee
If $T$ is integer rectifiable $\theta_i \in \Z^+$, so the sum must be finite.
\end{defn}

In particular, the mass measure of $T \in \intcurr_m(Z)$ satisfies
\be
||T|| = \sum_{i=1}^\infty ||\varphi_{i\#}\Lbrack \theta_i \Rbrack ||.
\ee
Theorems 4.3 and 8.8 of \cite{AK} provide necessary and sufficient
criteria for determining when a current is integer rectifiable.

Note that the current in Example~\ref{basic-current-pushed} is an integer
rectifiable current.  

\begin{ex} \label{basic-mani}
If one has a Riemannian manifold, $M^m$,
and a biLipschitz map $\varphi:M^m\to Z$, then 
$T=\varphi_\#\Lbrack\One_M\Rbrack$
is an integer rectifiable current of dimension $m$ in $Z$.  If $\varphi$
is an isometric embedding, and $Z=M$
then $\mass(T)=\vol(M^m)$.   Note further that
$\set(T)=\varphi(M)$.   

If $M$ has a conical singularity then $\set(T)=\varphi(M)$.
However, if $M$ has a cusp singularity at a point $p\in M$
then $\set(T)=\varphi(M\setminus\{p\})$.
\end{ex}

\begin{defn}\label{rmrk-def-boundary} \cite{AK}[Defn 2.3] 
The {\em boundary} of $T\in \curr_m(Z)$ is defined
\be \label{def-boundary}
\partial T(f, \pi_1, ... \pi_{m-1}):= T(1, f, \pi_1,...\pi_{m-1}) \in M_{m-1}(Z)
\ee
When $m=0$, we set $\partial T=0$.
\end{defn} 

Note that $\varphi_\#(\partial T)=\partial(\varphi_\#T)$. 
  
\begin{defn}  \cite{AK}[Defn 3.4 and 4.2]
An integer rectifiable current  $ T\in\intrectcurr_m(Z)$  is called an 
integral current, denoted $T\in \intcurr_m(Z)$,  if $\partial T$ 
defined as
\be
\partial T \left(f, \pi_1,...\pi_{m-1}\right) := T \left(1, f, \pi_1,...\pi_{m-1}\right)
\ee
has finite mass.   The total mass of an integral current is
\be
\nmass(T)=\mass(T) +\mass(\partial T).
\ee  
\end{defn}

Observe that $\partial \partial T=0$.
In \cite{AK} Theorem 8.6, Ambrosio-Kirchheim prove that 
\be
\partial: \intcurr_m(Z) \to \intcurr_{m-1}(Z)
\ee whenever $m\ge 1$.
By (\ref{mass-push}) one can see that if 
$\varphi: Z_1 \to Z_2$ is Lipschitz, then 
\be
\varphi_{\#}: \intcurr_m(Z_1) \to \intcurr_{m}(Z_2).
\ee

However, the restriction of an integral current need not be 
an integral current except in special circumstances.
For example, $T$ might be integration over $[0,1]^2$ with
the Euclidean metric and $A\subset [0,1]^2$ could have
an infinitely long boundary, so that $T\rstr A\notin \intcurr_2([0,1]^2)$
because $\partial(T\rstr A)$ has infinite mass.

\begin{rmrk}\label{bndry-1-current}
If $T $ is an $\mathcal{H}^1$ integral current then
$\partial T$ is an  $\mathcal{H}^0$ integer rectifiable current 
so
$
H=\set{\partial T} 
$
must be finite and
$\theta_{p_h}=||\partial T||(p_h)\in \Z^+$ for all $p\in H$  and
\be \label{0-param-representation-1}
\partial T(f)=\sum_{h\in H} \sigma_h \theta_h f(p_h) \qquad \forall f \in \mathcal{B}^\infty(Z).
\ee
as described above.  In addition, we have
\be 
0=T(1,1)=\partial T(1) =\sum_{h\in H} \sigma_h \theta_h.
\ee
\end{rmrk}

\begin{example} \label{sums-bad}
If $T $
is an $\mathcal{H}^1$ rectifiable current 
then 
\be
T=\sum_{i=1}^\infty \sigma_i \theta_i \varphi_{i\#}\Lbrack \chi_{A_i}\Rbrack
\ee
where $\theta_i\in \Z^+$, $\sigma_i\in \{+1, -1\}$
 and $A_i$ is an interval with $\bar{A}_i=[a_i,b_i]$
because all Borel sets are unions of intervals and all
integer valued Borel functions can be written up
to Lebesgue measure $0$ as a countable sum of 
characteristic functions of intervals.  One might like to
write:
\be 
\partial T(f)
=\sum_i \sigma_i \theta_i \left(f(\varphi_i(b_i))-f(\varphi_i(a_i))\right)
\qquad \forall f \in \mathcal{B}^\infty(Z).
\ee
This works when the sum happens to be a finite sum.
Yet if $T$ is a infinite collection
of circles based at a common point, $(0,0)\in \R^2$, 
defined with $\sigma_i=1$
$\theta_i=1$, $A_i=[0,\pi]$ and 
\begin{eqnarray}
\varphi_i(s)&=&(r_i\cos(s)+r_i, r_i\sin(s)) \textrm{ for } i \textrm{ odd and }\\
\varphi_i(s)&=&(r_i\cos(s+\pi)+r_i, r_i\sin(s+\pi)) \textrm{ for } i \textrm{ even}
\end{eqnarray}
where $r_{2i}=r_{2i-1}=1/i$ then 
\begin{eqnarray}
\varphi_i(a_i)=(2r_i, 0)
&\textrm{ and }& \varphi_i(b_i)=(0,0)
 \textrm{ for } i \textrm{ odd and }\\
 \varphi_i(a_i)=(0, 0)
&\textrm{ and }& \varphi_i(b_i)=(2r_i,0)
 \textrm{ for } i \textrm{ even. }
\end{eqnarray}
So when $f(0,0)=1$, we end up with an infinite sum
whose terms are all $+1$ and $-1$.
\end{example}

\vspace{.4cm}
\subsection{Review of Ambrosio-Kirchheim Slicing Theorems}

As in Federer-Fleming, Ambrosio-Kirchheim consider the slices of
currents:

\begin{thm} {\bf [Ambrosio-Kirchheim] }\cite{AK}[Theorems 5.6-5.7]
  \label{theorem-slicing}
Let $Z$ be a complete metric space, $T\in \intcurr_m Z$ and $f: Z \to \R$ a Lipschitz function.
Let
\be
<T,f,s> := \left(\partial T\right) \rstr f^{-1}\left(s,\infty\right)
-\partial\left( T \rstr f^{-1}\left(s,\infty\right) \right).  
\ee
Observe that
$$\set(<T,f,s>)\subset \left(\set(T)\cup \set(\partial T) \right)\cap f^{-1}(s),$$ and
\be \label{bndry-slice-1}
\partial<T,f,s> = <-\partial T, f,s>.
\ee
Furthermore $<T_1+T_2, f,s>=<T_1, f,s> + <T_2, f,s>$.
For almost every slice $s \in \R$, $<T,f,s>$ is
an integral current and we can
integrate the masses to obtain:
\be
\int_{s\in\R} \mass(<T,f,s>) \, ds = \mass(T \rstr df)
\le \Lip(f)\, \mass(T)
\ee
where
\be
(T \rstr df)(h, \pi_1,...,\pi_{m-1})=T(h, f,\pi_1,...\pi_{m-1}).
\ee
In particular, for almost every $s>0$ one has
\be
T \rstr f^{-1}(-\infty,s] \in \intcurr_{m-1}\left(Z\right).
\ee
Furthermore for all Borel sets $A$ we have
\be
<T\rstr A, f,s>=<T,f,S> \rstr A
\ee
and
\be
\int_{s\in\R} ||<T,f,s>||(A) \, ds = ||T \rstr df ||(A).
\ee
\end{thm}

\begin{rmrk} \label{push-slice-1}
Observe that for any $T\in \intcurr_m(Z')$, and
any Lipschitz functions, $\varphi: Z\to Z'$ and 
$f: Z' \to \R$ and any $s>0$, we have
\begin{eqnarray}
<\varphi_{\#}T,f,s> 
&=& \partial \left( (\varphi_{\#} T) \rstr f^{-1}\left(-\infty,s]\right) \right) - 
\left(\partial \varphi_\#T\right) \rstr f^{-1}\left(-\infty, s]\right)\\
&=& \partial \left(\varphi_{\#} (T \rstr \varphi^{-1}(f^{-1}\left(-\infty,s]\right) ) \right) - 
\left(\varphi_\#\partial T\right) \rstr f^{-1}\left(-\infty, s]\right)\\
&=& \partial \left(\varphi_{\#} (T \rstr (f \circ\varphi)^{-1}\left(-\infty,s]\right) \right) - 
\varphi_\# \left(\partial T \rstr \varphi^{-1}(f^{-1}\left(-\infty, s]\right)) \right)\\
&=& \left(\varphi_{\#} \partial (T \rstr (f \circ\varphi)^{-1}\left(-\infty,s]\right) \right) - 
\varphi_\# \left(\partial T \rstr (f\circ\varphi)^{-1}\left(-\infty, s]\right) \right)\\
&=& \varphi_{\#} <T , (f \circ\varphi), s >
 \end{eqnarray}
 \end{rmrk}
 
 \begin{rmrk}\label{rmrk-slicing}
 Ambrosio-Kirchheim then iterate this definition, $f_i: Z \to \R$,
$s_i \in \R$, to define iterated slices:
 \be \label{AK-iterated-slices}
 <T, f_1,...,f_k, s_1,...,s_k>= < <T, f_1,...,f_{k-1}, s_1,...,s_{k-1}>, f_k, s_k>,
 \ee
 so that 
 \be\label{slice-addition}
 <T_1+T_2, f_1,...,f_k, s_1,...,s_k> 
 =<T_1, f_1,...,f_k, s_1,...,s_k>+ <T_2, f_1,...,f_k, s_1,...,s_k>.
 \ee
  In \cite{AK} Lemma 5.9 they prove,
 \be \label{AK-Lem5.9}
 <T, f_1,...,f_k, s_1,...,s_k>= 
 < <T, f_1,...,f_{i}, s_1,...,s_{i}>, f_{i+1},...,f_k, s_{i+1},...,s_k>.
 \ee
 In \cite{AK} (5.9) they prove,
 \be \label{AK-5.9}
 \int_{\R^k} ||<T, f_1,...,f_k, s_1,...,s_k>||\, ds_1...ds_k=||T\rstr (1, f_1,...,f_k)||,
 \ee
where
\be
(T \rstr df)(h, \pi_1,...,\pi_{m-k})=T(h, f_1,...,f_k,\pi_1,...\pi_{m-k}),
\ee 
  so
 \be\label{Integral-Lip}
\int_{\R^k} \mass(<T,f_1,...f_k,s_1,...,s_k>) \, \mathcal{L}^k = \mass(T \rstr df)
\le \prod_{j=1}^k \Lip(f_j)\, \mass(T).
\ee
 In \cite{AK} (5.15) they prove for any Borel set $A \subset Z$
 and $\mathcal{L}^m$ almost every $(s_1,...s_k) \in \R^k$,
 \be \label{AK-5.15}
 <T\rstr A, f_1,...,f_k, s_1,...,s_k> = <T, f_1,...,f_k, s_1,...,s_k>\rstr A.
 \ee
 and
 \be
\int_{s\in\R^k} ||<T,f_1,...,f_k,s_1,...,s_k>||(A) \, ds = ||T \rstr df ||(A).
\ee
By (\ref{bndry-slice}) one can easily prove by induction that
\be \label{bndry-slice}
\partial<T, f_1,..., f_k, s_1,..., s_k>= (-1)^k <\partial T, f_1,..., f_k, s_1,..., s_k>.
\ee
In \cite{AK} Theorem 5.7 they prove 
 \be
 <T, f_1,..., f_k, s_1,..., s_k> \in \intcurr_{m-k}(Z).
\ee
 for $\mathcal{L}^k$ almost every $(s_1,...,s_k) \in \R^k$.
 By Remark~\ref{push-slice-1} one can prove inductively that
 \be \label{push-slice}
<\varphi_{\#}T,f_1,...f_k,s_1,...s_k> =
\varphi_{\#} <T , f_1 \circ\varphi,..., f_k\circ\varphi, s_1,..., s_k >.
 \ee
\end{rmrk}
 
 \vspace{.4cm}
\subsection{Review of Convergence of Currents}
 
Ambrosio Kirchheim's Compactness Theorem, which extends Federer-Fleming's Flat Norm 
Compactness Theorem, is stated in terms of weak convergence of
currents.  See Definition 3.6 in \cite{AK} which extends Federer-Fleming's notion of weak convergence except that they do not require compact support.

\begin{defn} \label{def-weak}
A sequence of integral currents $T_j \in \intcurr_m\left(Z\right)$ is said to converge weakly to
a current $T$ iff the pointwise limits satisfy
\be
\lim_{j\to \infty}  T_j\left(f, \pi_1,...\pi_m\right) = T\left(f, \pi_1,...\pi_m\right) 
\ee
for all bounded Lipschitz $f: Z \to \R$ and Lipschitz $\pi_i: Z \to \R$.
We write
\be
T_j \weaklyto T
\ee
\end{defn}

One sees immediately that $T_j \weaklyto T$ implies
\be
\partial T_j \weaklyto \partial T,   
\ee
\be
\varphi_\# T_j \weaklyto \varphi_\# T
\ee
and
\be
T_j \rstr (f, \pi_1,..., \pi_k) \weaklyto T\rstr (f, \pi_1,..., \pi_k).   
\ee
However $T_j \rstr A$ need not converge weakly to $T_j \rstr A$
as seen in the following example:

\begin{example}
Let $Z= \R^2$ with the Euclidean metric.  Let $\varphi_j: [0,1]\to Z$
be $\varphi_j(t) = (1/j, t)$ and $\varphi_\infty(t) = (0, t)$.   Let 
$S\in \intcurr_1([0,1])$ be
\be
S(f, \pi_1)= \int_0^1 f \, d\pi_.1
\ee
Let $T_j \in \intcurr_1(Z)$
be defined $T_j =\varphi_{j\#}(S)$.   Then $T_j \weaklyto T_\infty$.
Taking $A=[0,1]\times (0,1)$, we see that
$T_j \rstr A=T_j$ but $T_\infty \rstr A = 0$.
\end{example}

Immediately below the definition of weak convergence \cite{AK} Defn 3.6,
Ambrosio-Kirchheim prove 
the lower semicontinuity of mass:

\begin{rmrk} \label{rmrk-lower-mass}
If $T_j$ converges weakly to $T$, then $\liminf_{j\to\infty} \mass(T_j) \ge \mass(T)$.  
\end{rmrk}

\begin{thm}[Ambrosio-Kirchheim Compactness]\label{AK-compact}
Given any complete metric space 
$Z$, a compact set $K \subset Z$ and $A_0, V_0>0$.
Given
any sequence of integral currents  $T_j \in \intcurr_m \left(Z\right)$ satisfying
\be
\mass(T_j) \le V_0 \textrm{, } \mass(\partial T_j) \le A_0
\textrm{ and }
\set\left(T_j\right) \subset K,
\ee there exists a subsequence, $T_{j_i}$, and a limit current $T \in \intcurr_m\left(Z\right)$
such that $T_{j_i}$ converges weakly to $T$.
\end{thm}

\vspace{.4cm}
\subsection{Review of Integral Current Spaces}

The notion of an integral current space was introduced by
the second author and Stefan Wenger in \cite{SorWen2}:

\begin{defn} \label{defn-int-curr-space}
An $m$ dimensional metric space 
$\left(X,d,T\right)$ is called an integral current space if
it has a integral current structure $T \in \intcurr_m\left(\bar{X}\right)$
where $\bar{X}$ is the metric completion of $X$
and $\set(T)=X$.   
Given an 
integral current 
space $M=\left(X,d,T\right)$ we will use
$\set\left(M\right)$ or $X_M$ to denote $X$,  $d_M=d$ and $\Lbrack M \Rbrack =T $. 

Note that $\set\left(\partial T\right) \subset \bar{X}$.   
The boundary of $\left(X,d,T\right)$ is then the integral current space:
\be
\partial \left(X,d_X,T\right) := \left(\set\left(\partial T\right), d_{\bar{X}}, \partial T\right).
\ee
If $\partial T=0$ then we say $\left(X,d,T\right)$ is an integral current without boundary.
\end{defn}

\begin{rmrk} \label{space-param}
Note that any $m$ dimensional integral current space is countably 
$\mathcal{H}^m$ rectifiable with orientated charts, $\varphi_i$
and weights $\theta_i$ provided
as in (\ref{param-representation}).   A $0$ dimensional integral current space
is a finite collection of points with orientations, $\sigma_i$ and
weights $\theta_i$ provided as in (\ref{0-param-representation}).
If this space is the boundary
of a $1$ dimensional integral current space, then
as in Remark~\ref{bndry-1-current}, the sum of the signed weights is 0.
\end{rmrk}

\begin{example}
A compact oriented Riemannian manifold with boundary, $M^m$,
is an integral current space, where $X=M^m$, $d$ is the standard
metric on $M$ and $T$ is integration over $M$.  In this
case $\mass(M)=\vol(M)$ and $\partial M$ is the boundary manifold.
When $M$ has no boundary, $\partial M=0$.
\end{example}

\begin{defn} \label{defn-integral-current-space}
The space of $m\ge 0$ dimensional integral current spaces, 
$\mathcal{M}^m$,
consists of all metric spaces which are integral current spaces 
with currents of dimension $m$ as
in Definition~\ref{defn-int-curr-space} as well as the $\bf{0}$ spaces.
Then $\partial: \mathcal{M}^{m+1}\to \mathcal{M}^{m}$.
\end{defn}

\begin{rmrk}\label{bndry-1-space}
A $0$ dimensional integral current space, $M=(X,d,T)$,
is a finite collection of points, $\{p_1,...,p_N\}$,
with a metric $d_{i,j}=d(p_i,p_j)$ and a current structure
defined by assigning a weight, $\theta_i\in\Z^+$,
and an orientation, $\sigma_i\in \{+1,-1\}$ to each $p_i\in X$
and
\be
\mass(M)=\sum_{i=1}^N \theta_i.
\ee
If $M$ is the boundary of a $1$ dimensional integral current
space then, as in Remark~\ref{bndry-1-current}, we have
\be
\sum_{i=1}^N \sigma_i \theta_i=0
\ee
In particular $N\ge 2$ if $M\neq \bf{0}$.
\end{rmrk}

Any compact Riemannian manifold with boundary is an integral current space.   Additional examples
appear in the work of Wenger and the second author \cite{SorWen2}.

We end this subsection with an example of an
integral current space 
that is applied in this paper to justify hypothesis
of many of our results.   

\begin{example}\label{concentric-spheres}
Consider the one dimensional integral
current space $(X,d,T)$, where 
\be
X=\{0\} \cup \bigcup_{j=1}^\infty \partial B(0,1/R_j) \subset \E^2
\ee
where $(\E^2, d_{\E^2})$ is the Euclidean plane,
with the restricted metric, $d=d_{\E^2}$, where 
\be
T(\omega)=\sum_{j=1}^\infty \Lbrack \partial B(0,1/R_j) \Rbrack
\ee
is the integral current in $\bar{X}$ and in $\E^2$
and where $R_j =1/2^j$.    Observe that
for 
\be
N_r=\inf \{j:\, 1/2^j<r\} \subset [\log_2(1/r), \log_2(r)+1]
\ee 
we have
\begin{eqnarray}
||T||(B(0,r))&=& 
\sum_{j \ge N_r}^\infty \mathcal{H}^1(\partial B(0,1/R_j))\\
&=& 
\sum_{j \ge N_r}^\infty \frac{2\pi}{2^j}= \frac{4\pi}{2^{N_r}}
\in \left[ \frac{8\pi}{r},\frac{4\pi}{r} \right].
\end{eqnarray}
In this way the
total mass is finite and $\{0\} \in X=\set(T)$.   Observe that
$\partial T=0$.
\end{example}

\vspace{.4cm}
\subsection{Review of the Intrinsic Flat distance}

The Intrinsic Flat distance was defined in work of the second author and
Stefan Wenger \cite{SorWen2} as a new distance between Riemannian
manifolds based upon the work of Ambrosio-Kirchheim reviewed above.

Recall that the flat distance between $m$ dimensional integral currents 
$S,T\in\intcurr_m\left(Z\right)$ is given by 
\begin{equation} \label{eqn-Federer-Flat}
d^Z_{F}\left(S,T\right):= 
\inf\{\mass\left(U\right)+\mass\left(V\right):
S-T=U+\bdry V \}
\end{equation}
where $U\in\intcurr_m\left(Z\right)$ and $V\in\intcurr_{m+1}\left(Z\right)$.
This notion of a flat distance was first introduced by Whitney
in \cite{Whitney} and later adapted to rectifiable currents by Federer-Fleming \cite{FF}.
The flat distance between Ambrosio-Kirchheim's integral currents was
studied by Wenger in \cite{Wenger-flat}.   In particular,
Wengr proved that if $T_j \in \intcurr_m(Z)$ has
$\mass(T_j) \le V_0$ and $\mass(\partial T_j) \le A_0$ then
\be
T_j \weaklyto T \textrm{ iff } d^Z_F(T_j, T) \to 0
\ee
exactly as in Federer-Fleming.

The intrinsic flat distance between integral current spaces
was first defined in \cite{SorWen2}[Defn 1.1]:

\begin{defn} \label{def-flat1} 
 For $M_1=\left(X_1,d_1,T_1\right)$ and $M_2=\left(X_2,d_2,T_2\right)\in\mathcal M^m$ let the 
intrinsic flat distance  be defined:
 \begin{equation}\label{equation:def-abstract-flat-distance}
  d_{\Fm}\left(M_1,M_2\right):=
 \inf d_F^Z
\left(\varphi_{1\#} T_1, \varphi_{2\#} T_2 \right),
 \end{equation}
where the infimum is taken over all complete metric spaces 
$\left(Z,d\right)$ and isometric embeddings 
$\varphi_1 : \left(\bar{X}_1,d_1\right)\to \left(Z,d\right)$ and $\varphi_2: \left(\bar{X}_2,d_2\right)\to \left(Z,d\right)$
and the flat norm $d_F^Z$ is taken in $Z$.
Here $\bar{X}_i$ denotes the metric completion of $X_i$ and $d_i$ is the extension
of $d_i$ on $\bar{X}_i$, while $\phi_\# T$ denotes the push forward of $T$.
\end{defn}

In \cite{SorWen2}, it is observed that
 \be
d_{\Fm}\left(M_1,M_2\right) \le d_{\Fm}\left(M_1,0\right)+d_{\Fm}\left(0,M_2\right) \le \mass\left(M_1\right)+\mass\left(M_2\right).
\ee 
There it is also proven that $d_{\mathcal{F}}$ satisfies the
triangle inequality \cite{SorWen2}[Thm 3.2] and is a distance:

\begin{thm} \label{zero-mani}  \cite{SorWen2}[Thm 3.27]
Let $M,N$ be precompact integral current spaces and suppose that $d_{\Fm}\left(M,N\right)=0$.
Then there is a current preserving isometry from $M$ to $N$ where
an isometry $f: X_M \to X_N$ is called
a current preserving isometry between $M$ and $N$, if its
extension $\bar{f}: \bar{X}_M \to \bar{X}_N$ pushes forward the
current structure on $M$ to the current structure on $N$:
$
\bar{f}_\# T_M= T_N
$
\end{thm}

In \cite{SorWen2} Theorem 3.23 it is also proven that

\begin{thm}\label{inf-dist-attained}\label{achieved} \cite{SorWen2}[Thm 4.23]
Given a pair of precompact integral current spaces, $M^m_1=(X_1,d_1,T_1)$
and 
$M^m_2=(X_2,d_2,T_2)$, 
there exists a compact metric space, $(Z, d_Z)$,
integral
currents $U\in\intcurr_m\left(Z\right)$ and  $V\in\intcurr_{m+1}\left(Z\right)$,
and isometric embeddings
$\varphi_1 : \bar{X}_1\to Z$ and $\varphi_2:\bar{X}_2 \to Z$
with
\begin{equation} \label{eqn-Federer-Flat-3}
\varphi_\# T_1- \varphi'_\# T_2=U+\bdry V
\end{equation}
such that
\begin{equation}\label{eqn-local-defn-2}
d_{\Fm}\left(M_1,M_2\right)=\mass\left(U\right)+\mass\left(V\right).
\end{equation}
\end{thm}

\begin{rmrk}\label{rmrk-inf-dist-attained}
The metric space $Z$ in Theorem~\ref{inf-dist-attained} has
\be
\diam(Z) \le 3 \diam(X_1) + 3\diam(X_2).
\ee
This is seen by consulting the proof of Theorem 3.23 in \cite{SorWen2},
where $Z$ is constructed as the injective envelope of the Gromov-Hausdorff
limit of a sequence of spaces $Z_n$ with this same diameter bound.
\end{rmrk}

The following theorem in \cite{SorWen2} is an immediate consequence
of Gromov and Ambrosio-Kirchheim's Compactness Theorems:

\begin{thm} \label{GH-to-flat}
Given a sequence of $m$ dimensional integral current spaces $M_j=\left(X_j, d_j, T_j\right)$ such that $X_j$ are equibounded and
equicompact and with uniform upper bounds on mass and boundary mass.
A subsequence converges in the
Gromov-Hausdorff sense $\left(X_{j_i}, d_{j_i}\right) \GHto \left(Y,d_Y\right)$ and in the 
intrinsic flat sense 
$\left(X_{j_i}, d_{j_i}, T_{j_i}\right) \Fto \left(X,d,T\right)$
where either $\left(X,d,T\right)$ is an $m$ dimensional integral current space
with $X \subset Y$
or it is the ${\bf 0}$ current space.
\end{thm}

Immediately one notes that if $Y$ has Hausdorff dimension less than $m$,
then $(X,d,T)=\bf{0}$.   In \cite{SorWen2} Example A.7, 
there is an example where
$M_j$ are compact three 
dimensional Riemannian manifolds with positive scalar curvature that
converge in the Gromov-Hausdorff sense to a standard three sphere
but in the Intrinsic Flat sense to $\bf{0}$.   It is proven in \cite{SorWen1},
that if $(X_j,d_j, T_j)$ are compact Riemannian manifolds with nonnegative
Ricci curvature or a uniform linear contractibility function, then the
intrinsic flat and Gromov-Hausdorff limits agree.

There are many examples of sequences of Riemannian manifolds
which have no Gromov-Hausdorff limit but have an intrinsic flat limit.
The first is Ilmanen's Example of an increasingly hairy three sphere
with positive scalar curvature described in \cite{SorWen2} Example A.7.   Other
examples appear in work of the second author with Dan Lee concerning the
stability of the Positive Mass Theorem \cite{LeeSormani1} \cite{LeeSormani2} and in
work of the second author with Sajjad Lakzian concerning smooth convergence
away from singular sets \cite{Lakzian-Sormani}.

The following three theorems are proven in work of the
second author with Wenger \cite{SorWen2}.  Combining these theorems with
the work of Ambrosio-Kirchheim reviewed earlier will lead to many
of the properties of Intrinsic Flat Convergence described in this paper:

\begin{thm}\label{converge}\label{converge} \cite{SorWen2}[Thm 4.2]
If a sequence of 
integral current spaces
$M_{j}=\left(X_j, d_j, T_j\right)$ converges in the intrinsic flat sense
to an 
integral current space,
 $M_0=\left(X_0,d_0,T_0\right)$, then
there is a separable
complete metric space, $Z$, and isometric embeddings  $\varphi_j: X_j \to Z$ such that
$\varphi_{j\#}T_j$ flat converges to $\varphi_{0\#} T_0$ in $Z$
and thus converge weakly as well.
\end{thm}

\begin{thm}\label{convergeto0} \cite{SorWen2}[Thm 4.3]
If a sequence of 
integral current spaces
$M_{j}=\left(X_j, d_j, T_j\right)$ converges in the
intrinsic flat sense to the 
zero integral current space, $\bf{0}$, then
we may choose points $x_j\in X_j$ and a
separable
complete metric space, $Z$, and isometric embeddings  $\varphi_j: X_j \to Z$ such that
$\varphi_j(x_j)=z_0\in Z$ and
$\varphi_{j\#}T_j$ flat converges to $0$ in $Z$ and thus converges weakly as well.
\end{thm}

\begin{thm}\label{semi-mass}
If a sequence of integral current spaces
$M_{j}$ converges in the intrinsic flat sense
to a 
integral current space,
 $M_\infty$, then
 \be
 \liminf_{i\to\infty} \mass(M_i) \ge \mass(M_\infty)
 \ee
 \end{thm}
 
 \begin{proof}
 This follows from Theorems~\ref{converge} and~\ref{convergeto0}
 combined with Ambrosio-Kirchheim's lower semicontinuity
 of mass [c.f. Remark~\ref{semi-mass}].
 \end{proof}

Finally there is Wenger's Compactness Theorem \cite{Wenger-compactness}:

\begin{thm}[Wenger]
Given $A_0, V_0, D_0>0$.   If $M_j=(X_j, d_j, T_j)$ are integral current
spaces such that
\be
\diam(M_j) \le D_0 \qquad \mass(M_j) \le V_0 \qquad \mass(\partial(M_j))\le A_0
\ee
then a subsequence converges in the Intrinsic Flat Sense to an
integral current space of the same dimension, possibly the $\bf{0}$ space.
\end{thm}

Recall that this theorem applies to oriented Riemannian manifolds of
the same dimension with a uniform upper bound on volume and a uniform
upper bound on the volumes of the boundaries.   One immediately sees
that the conditions required to apply Wenger's Compactness Theorem
are far weaker than the conditions required for Gromov's Compactness
Theorem.  The only difficulty lies in determining whether the limit space is
$\bf{0}$ or not.  Wenger's proof involves a thick thin decomposition, a
study of filling volumes and uses the notion of an ultralimit.

It should be noted that Theorems~\ref{converge}-\ref{semi-mass}
and all other theorems reviewed and
proven within this paper are proven without applying Wenger's Compactness 
Theorem.  Thus one may wish to attempt alternate proofs of Wenger's
Compactness Theorem using the results in this paper.  

We end this subsection with an example of a converging
sequence of
integral current spaces 
that is applied in this paper to justify many of
our hypothesis.   Many other examples
appear in work of Wenger and the second author \cite{SorWen2}.

\begin{example}\label{many-intervals}
We will construct a particular sequence of one-dimensional integral current spaces $M_\ell$ which converges in the intrinsic flat sense to the integral current space $M$ induced by the standard one-dimensional torus of length $1$ denoted by $\mathbb{T}$.

We define a sequence $T_k\in \intcurr_1(\mathbb{T})$ as follows.
Let $A_{i,n}$  ($i = 0, \dots, 2^n - 1$) denote the dyadic interval
\be
A_{i,n} = \left[ \frac{i}{2^n}, \frac{i+1}{2^n} \right] \subset \mathbb{T},
\ee
and let $T_{i,j,n} \in \intcurr_1(\mathbb{T})$ for $0 \leq i < j \leq 2^n-1$ be defined by
\be
T_{i,j,n} = \Lbrack \chi_{A_{i,n}} \Rbrack + \Lbrack \chi_{A_{j,n}} \Rbrack,
\ee
where $\chi_A$ denotes the characteristic function of a set $A \subset Z$. 
Reindex $T_k = T_{i,j,n}$ according to $k = k(i,j,n)$ such that $k$ is one-to-one, onto $\mathbb{N}$ and $k(i_1, j_1, n_1) \leq k(i_2, j_2, n_2)$ if and only if $n_1 \leq n_2$.

Let $T = \Lbrack 1 \Rbrack \in \intcurr_1(\mathbb{T})$, let for every $k \in \mathbb{N}$, $M_{2k}$ and $M_{2k+1}$ be the one-dimensional integral current spaces associated to the currents $T - T_k$ and $T + T_k$ respectively. Note moreover that $M$ is the integral current space associated to $T$.

Then 
\begin{eqnarray}
d_{\mathcal{F}}\left(M_{2k}, M\right)
&\le&
d_F^Z\left(T - T_k, T\right)\\
&\le& \mass(T_k) \to 0. 
\end{eqnarray}
Similarly, $M_{2k+1} \Fto M$, so that $M_\ell \to M$ as $\ell \to \infty$.
\end{example}

\vspace{.4cm}
\subsection{Filling Volumes}

The notion of a filling volume was first introduced by Gromov
in \cite{Gromov-Filling}.   Wenger studied the filling volumes 
of integral currents in metric spaces in \cite{Wenger-flat}.
This was applied in joint work of the second author with
Wenger in \cite{SorWen1}.

First we discuss the Plateau Problem on complete metric spaces.
Given a integral current $T\in \intcurr_mZ$, one may define the filling
volume of $\partial T$ within $Z$ as
\be
\fillvol_Z(\partial T) = \inf \{\mass(S): \,\, S\in \intcurr_m(Z)
\,s.t.\,\,\partial S=\partial T\}.
\ee
This immediately provides an upper bound on the flat distance:
\be
d_F^Z(\partial T, {\bf{0}}) \le \fillvol_Z(\partial T)\le \mass(T) 
\ee
Ambrosio-Kirchheim proved this infimum is achieved on Banach
spaces, $Z$ \cite{AK} [Theorem10.2].

Wenger defined the
absolute filling volume of $T\in \intcurr_m Y$  to be
\be
\fillvol_\infty(\partial T)=\inf\{\mass(S): \,\, S\in \intcurr_m(Z)
\,s.t. \,\, \partial S=\varphi_\#\partial T\}
\ee
where the infimum is taken over all isometric embeddings $\varphi:Y\to Z$,
all complete metric spaces, $Z$, and all $S\in \intcurr_m(Z)$
such that $\partial S=\varphi_\#T$.   Clearly
\be
\fillvol_\infty(\partial T) \le \fillvol_Y(\partial T).
\ee

Here we will use the following notion of a filling of an integral current space:

\begin{defn} \label{defn-filling-volume}
Given an integral current space 
$M=(X,d,T)\in \mathcal{M}^m$ with $m\ge 1$ we define
\be \label{eq-filling-volume}
\fillvol(\partial M):= \inf\{ \mass(N): \, N\in \mathcal{M}^{m+1} \textrm{ and } \partial N=\partial M\}.
\ee
That is we require that there exists a current preserving 
isometry from $\partial N$ onto $\partial M$, where as usual, we
have taken the metrics on the boundary spaces to be the restrictions
of the metrics on the metric completions of $N$ and $M$ respectively.
\end{defn}

We note that for $M = (X, d, T) \in \mathcal{M}^m$, it holds that
\be
\fillvol(\partial M) = \fillvol_\infty(\partial T).
\ee
It is also easy to see that
\be \label{fill-est}
\fillvol(\partial M) \le \mass(M).
\ee
and 
\be
d_\mathcal{F}(\partial M,0) \le \fillvol(\partial M)\le \mass(M)
\ee 
for any integral current
space $M$.   

\begin{rmrk}
The infimum in the definition of the filling volume is achieved
when the space is precompact.
This may be seen 
by imitating the proof that the infimum in the definition of the
intrinsic flat norm is attained in \cite{SorWen2}.
Since the $N$ achieving the infimum has $\partial N \neq 0$,
the filling volume is positive.   
\end{rmrk}

Any integral current space, $M=(X,d,T)$, is separable and so
one can map the space into a Banach space, $Z$, via the Kuratowski
Embedding theorem, $\iota: X\to Z$.   By Ambrosio-Kirchheim's
solution to the Plateau problem on Banach spaces
\cite{AK}[Prop 10.2], 
\be \label{fill-diam}
\fillvol(\partial M)\le \fillvol_Z(\varphi_\#(\partial T) )
\le \diam(X)\,\mass(\partial T)=
\diam(M)\,\mass(\partial M).
\ee

Wenger showed that the filling volume is continuous with respect
to weak convergence (and thus also intinsic flat convergence when 
applying Theorem~\ref{converge}).  Here we provide a precise estimate
which will be needed later in the paper:

\begin{thm} \label{fillvol-cont}
For any pair of integral current spaces, $M_i$, we have
\vspace{.2cm}
\be
\label{eq:FillVolByFlat}
\fillvol(\partial M_1)  \le 
\fillvol(\partial M_2) + d_{\mathcal{F}}(M_1,M_2).
\ee
and if $M_i$ have finite diameter then
\be
\fillvol(\partial M_1)  \le 
\fillvol(\partial M_2) + (1+3\diam(M_1)+3\diam(M_2)\,)\,
d_{\mathcal{F}}(\partial M_1,\partial M_2).
\ee
\end{thm}

\begin{proof}
Let $M_k=(X_{M_k}, d_{M_k}, T_{M_k})$ for $k=1,2$.

By the definition of intrinsic flat distance
there exists integral currents
$A_i,B_i$ in $Z_j'$
and isometric embeddings, $\varphi_{i,k}: X_{M_k} \to Z_i$,
such that 
\be
\varphi_{i,1\#} T_{M_1}-\varphi_{i,2\#} T_{M_2}=\partial B_i + A_i
\ee
where
\be
d_{\mathcal{F}}(M_1, M_2 ) = \lim_{i\to \infty}\mass(A_i) +\mass(B_i)
\ee
In particular
\be
\varphi_{i,1\#} \partial T_{M_1}-\varphi_{i,2\#} \partial T_{M_2}= \partial A_i.
\ee

Now by (\ref{eq-filling-volume}), there exists $N_i=(X_{N_i}, d_{N_i}, T_{N_i}) \in \mathcal{M}^{m+1}$ such that $\partial N_i=\partial M_2$ and 
\be
\fillvol(\partial M_2)=\lim_{i\to\infty} \mass(N_i).
\ee

Applying the gluing techniques which are developed
clearly in the Appendix,
we may glue the integral current space
$(\set(A_i)\subset Z_i, d_{Z_i}, A_i)$ to
$N_i=(X_{N_i}, d_{N_i}, T_{N_i})$ along
$\partial N_i=\partial M_2$ to create an
integral current space $M$ such that
$\partial M= \partial M_1$ and 
$\mass(M)\le \mass(A_i) + \mass(N_i)$.

Then
\be
\fillvol(\partial M_1)= \fillvol(\partial M)
\le  \mass(M) \le \mass(A_i) + \mass(N_i)
\ee
and taking $i \to \infty$ we have (\ref{eq:FillVolByFlat}).

For the second half of the theorem, we observe that
there exists a new pair of integral currents $B_j,A_j$ 
and isometric embeddings, $\varphi_{i,k}: \spt(\partial T_k) \to Z'_j$,
such that 
\be
\varphi_{j,1\#} \partial T_{M_1}-\varphi_{j,2\#} \partial T_{M_2}=
\partial B_j + A_j
\ee
where
\be
d_{\mathcal{F}}(\partial M_1, \partial M_2 ) = \lim_{j \to \infty} \mass(B_j) +\mass(A_j).
\ee
Let
\be
M_{B_j}=(\set(B_j), d_{Z'_j}, B_j) \textrm{ and } M_{A_j}=(\set(A_j), d_{Z'_j}, A_j).
\ee
As in the proof of Theorem 3.23 in \cite{SorWen2} (see also Remark~\ref{rmrk-inf-dist-attained}), we may assume that
\be
\diam(M_{B_j}), \diam(M_{A_j}) \le 3 \diam(M_1) + 3 \diam(M_2).
\ee

Observe that 
\be
\partial A_j= 
\partial(\varphi_{j,1\#} \partial T_{M_1}-\varphi_{j,2\#} \partial T_{M_2})= 0.
\ee
So we can study the filling volume of $A_j$.  
In \cite{AK} Theorem 10.2, we see that
\be
\fillvol(M_{A_j}) \le \diam(M_{A_j}) \,\mass(A_j).
\ee

Let $N_j$ be integral current spaces such that
$\partial N_{j}=\partial M_2$ and
\be
\fillvol(\partial M_2) \geq \mass(N_{j})  - 1/j.
\ee
Let $N'_j$ be integral current spaces such that
$\partial N'_{j}=A_j$ and
\be
\fillvol(A_j) \geq \mass(N'_{j}) - 1/j.
\ee

We glue $N_j$ to $M_{B_j}$ along $\partial N_j=\partial M_2$
and we also glue $N_j'$ to $M_{B_j}$ along $\partial N_j'=M_{A_j}$.
The glued space $M'_j$ will have $\partial M'_j= \partial M_1$
and 
\be
\mass(M'_j)\le \mass(N_j) +\mass(M_{B_j})+ \mass(N_j').
\ee
Thus
\be
 \fillvol(\partial M_1)= \fillvol(\partial M'_j) \le \mass(M'_j).
 \ee
 Combining these equations we have
 \begin{eqnarray}
 \fillvol(\partial M_1)  - \tfrac{2}{j}
 &\le&
  \fillvol(\partial M_2) + \mass(M_{B_j}) + \fillvol(M_{A_j}) \\
&\le &\fillvol(\partial M_2) + \mass(M_{B_j}) + \diam(M_{A_j}) \mass(M_{A_j})\\
&\le &\fillvol(\partial M_2) + \left(\diam(M_{A_j})+1\right)
\left(\mass(M_{B_j})+ \mass(M_{A_j})\right)
\end{eqnarray}
and taking $j \to \infty$ we have our second claim.
\end{proof}

\begin{rmrk}
Gromov's Filling Volume in \cite{Gromov-Filling} is defined as
in (\ref{eq-filling-volume}) where the infimum is taken over 
$N^{n+1}$ that are Riemannian
manifolds.  Thus it is conceivable that the filling volume in
Definition~\ref{defn-filling-volume} might
have a smaller value both because integral current spaces
have integer weight and because we have a wider class
of metrics to choose from, including metrics which are not
length metrics.   
\end{rmrk}

\begin{rmrk}
Note also that the mass used in  
Definition~\ref{defn-filling-volume} is Ambrosio-Kirchheim's
mass \cite{AK} Definition 2.6 stated as Definition~\ref{defn-mass}
here.   Even when the weight is $1$ and one has a Finsler manifold, 
the Ambrosio-Kirchheim mass has a different value than 
any of Gromov's masses  \cite{Gromov-Filling} and 
the masses used by Burago-Ivanov \cite{Burago-Ivanov-Vol-Tori}.
We need Ambrosio-Kirchheim's mass to have continuity of the
filling volumes under intrinsic flat convergence [Theorem~\ref{fillvol-cont}]
which is an essential tool in this paper.
\end{rmrk}   

\newpage
\section{Metric Properties of Integral Current Spaces} \label{subsect-metric-prop}
 
In this section prove a number of properties of integral current 
spaces as well as a new Gromov-Hausdorff Compactness Theorem.
After describing the natural notions of balls, isometric products, slices, spheres and filling volumes in the first three subsections,
we move on to key new notions.

We introduce the Sliced Filling Volume [Definition~\ref{sliced-filling-vol}]
and $\SF_k(p,r)$ [Definition~\ref{defn-SF_k}].  Then we prove a new
Gromov-Hausdorff Compactness Theorem 
[Theorem~\ref{SF_k-compactness}].   

We explore the filling volumes of $0$ dimensional
spaces, apply them to bound the volumes of balls, and then
introduce the Tetrahedral Property [Definition~\ref{defn-tetra}] and
the Integral Tetrahedral Property [Definition~\ref{defn-int-tetra}].   

We close this section with the notion of
interval filling volumes in Definition~\ref{IFV}
and Sliced Interval Filling Volumes in Definition~\ref{SIF}.  
 
Those studying the
proof of the Tetrahedral Compactness Theorem need to read all
Sections except 3.2 and 3.12 before continuing to Section 4.   
Those studying the Bolzano-Weierstrass
and Arzela-Ascoli Theorems need only read Sections 3.1 and 3.3-3.6
before continuing to Section 4.

\vspace{.4cm}
\subsection{Balls}

Many theorems in Riemannian geometry involve balls,
\be
B(p,r)= \{ x \in X : \, d_X(x,p)<r\} \qquad 
\bar{B}(p,r)= \{ x \in X : \, d_X(x,p)\le r\}.  
\ee  
In this subsection we quickly review key lemmas about balls proven
in the background of the second author's recent paper \cite{Sormani-AA}.

\begin{lem}\label{lem-ball-current}
A ball in an integral current space, $M=\left(X,d,T\right)$,
with the current restricted from the current structure of the Riemannian manifold is an integral current space itself, 
\be
S\left(p,r\right)=\left(\set(T\rstr B(p,r)),d,T\rstr B\left(p,r\right)\right)
\ee
for almost every $r > 0$.   Furthermore,
\be\label{ball-in-ball}
B(p,r) \subset \set(S(p,r))\subset \bar{B}(p,r)\subset X.
\ee
\end{lem}

One may imagine that it is possible that a ball is cusp shaped
when we are not in a length space and that some points in the
closure of the ball that lie in $X$ do not lie in the set of $S(p,r)$.
In a manifold, the set of $S(p,r)$ is a closed ball:

\begin{lem}
When $M$ is a Riemannian manifold with boundary
\be
S\left(p,r\right)=\left(\bar{B}\left(p,r\right),d,T\rstr B\left(p,r\right)\right)
\ee
is an integral current space for all $r > 0$.
\end{lem}

\begin{example}
See \cite{Sormani-AA} for an example of an integral
current space with a ball that is not an integral current
space because it's boundary has infinite mass.  
\end{example}

\begin{rmrk}\label{outside-balls}
Note that
the outside of the ball, $(M\setminus B(p,r), d, T-S(p,r))$, is
also an integral current space for almost every $r>0$.
\end{rmrk}

\vspace{.4cm}
\subsection{Isometric Products}

One of the most useful notions in Riemannian geometry is
that of an isometric product $M\times I$
of a Riemannian manifold $M$ with an interval, $I$.
endowed with the metric
\be\label{isom-product}
d_{M\times I}((p_1,t_1), (p_2,t_2))=\sqrt{ d_M(p_1,p_2)^2 + |t_1-t_2|^2 }.
\ee
We need to define the isometric product of
an integral current space with an interval:

\begin{defn} \label{times-I}
The product of an integral current space, $M^m=(X,d_X,T)$, 
with an interval $I_\epsilon$, denoted 
\be
M\times I=(X\times I_\epsilon, d_{X\times I_\epsilon},T\times I_\epsilon)
\ee
where $d_{X\times I_\epsilon}$ is defined as in (\ref{isom-product})
and
\be\label{eq-isom-product}
(T\times I_\epsilon)(f, \pi_1,..., \pi_{m+1})
=\sum_{i=1}^{m+1} (-1)^{i+1} 
\int_0^\epsilon T\left(f_t \frac{\partial \pi_{i}}{\partial t}, \pi_{1t},..., \hat{\pi}_{it},..., \pi_{(m+1)t}\right) \, dt
\ee
where $h_t: \bar{X}\to \R$ is defined $h_t(x)=h(x,t)$ for
any $h:\bar{X}\times I_\epsilon \to \R$ and where
\be
(\pi_{1t},..., \hat{\pi}_{it},..., \pi_{(m+1)t})= (\pi_{1t}, \pi_{2t}, ...,\pi_{(i-1)t},\pi_{(i+1)t},...,\pi_{(m+1)t}).
\ee
\end{defn}

We prove this defines an integral current space in 
Proposition~\ref{prop-times-I} below.

\begin{rmrk}This is closely related 
to the cone construction in Defn 10.1 of \cite{AK}, however
our ambient metric space changes after taking the product
and we do not contract to a point.
Ambrosio-Kirchheim observe that (\ref{eq-isom-product}) 
is well defined because 
for $\mathcal{L}^1$ almost every $t\in I_\epsilon$
the partial derivatives are defined for $||T||$ almost every
$x\in X$.   This is also true in our setting.   The proof that
their cone construction defines a current \cite{AK} Theorem 10.2, 
however, does not extend to
our setting because our construction does not close up at
a point as theirs does and our construction depends on $\epsilon$
but not on the size of a bounding ball.
\end{rmrk}

\begin{prop}\label{prop-times-I}
Given an integral current space $M=(X,d,T)$,
the isometric product $M\times I_\epsilon$
is an integral current space such that
\be\label{isom-mass}
\mass(M\times I_\epsilon)=\epsilon \mass(M)
\ee
and such that
\be\label{bndbnd}
\partial(T\times I_\epsilon)= - (\partial T) \times I_\epsilon 
+ T \times \partial I_\epsilon.
\ee
where
\be
T\times \partial I_\epsilon := \psi_{\epsilon \#} T - \psi_{0\#} T
\ee
where $\psi_t: \bar{X}\to \bar{X}\times I_\epsilon$
is the isometric embedding $\psi_t(x)=(x,t)$.  
\end{prop}

\begin{proof}
First we must show $T\times I_\epsilon$ satisfies the three
conditions of a current:

Multilinearity follows from the multilinearity of $T$ and
the use of the alternating sum in the definition of $T\times I$.

To see locality we suppose 
there is a $\pi_i$ which is constant on a neighborhood
of $\{f\neq 0\}$.  Then
$\partial \pi_i/\partial t =0$ on a neighborhood of $\{f \neq 0\}$
so the $i^{th}$ term in the sum is $0$.  Since 
for all $t\in I_\epsilon$,
$\pi_{it}$ is constant on a neighborhood of $\{f_t \neq 0\}$
the rest of the terms are $0$ as well by the locality of $T$.

To prove continuity and finite mass, we will use the
fact that $T$ is integer rectifiable.  In particular
there exists a parametrization as $\varphi_i: A_i\subset \R^m \to \bar{X}$
and weight functions $\theta_i\in L^1(A_i, \N)$
such that
\be
T=\sum_{k=1}^\infty \varphi_{k\#} \Lbrack \theta_k \Rbrack.
\ee
So $(T\times I_\epsilon)(f, \pi_1,..., \pi_{m+1})=$
\begin{eqnarray*}
&=&\sum_{i=1}^{m+1} (-1)^{i+1} 
\int_0^\epsilon 
\sum_{k=1}^\infty \varphi_{k\#} \Lbrack \theta_k \Rbrack
\left(f_t \frac{\partial \pi_{it}}{\partial t}, , \pi_{1t},..., \hat{\pi}_{it},..., \pi_{(m+1)t}\right) \, dt
 \\
 &=&\sum_{i=1}^{m+1} (-1)^{i+1} \sum_{k=1}^\infty 
\int_{t=0}^\epsilon \int_{A_k} \theta_k\,
f_t\circ \varphi_k \,\frac{\partial \pi_{it}}{\partial t}\circ \varphi_k 
\,d(\pi_{1t}\circ\varphi_k)\wedge\cdots\wedge d\hat{\pi}_{it}
\wedge \cdots \wedge d( \pi_{(m+1)t}\circ\varphi_k) \, dt\\
  &=&  \sum_{k=1}^\infty \int_{A_k}  
\int_{t=0}^\epsilon \theta_k(x) f( \varphi_k(x), t)  \, \left(
\sum_{i=1}^{m+1} (-1)^{i+1} 
\frac{\partial \pi_{it}}{\partial t}\circ \varphi_k 
d(\pi_{1t}\circ\varphi_k)\wedge\cdots\wedge d\hat{\pi}_{it}
\wedge \cdots \wedge d( \pi_{(m+1)t}\circ\varphi_k)\right) \, dt\\
  &=&  \sum_{k=1}^\infty \int_{A_k\times I_\epsilon}  
 \theta_k(x) f( \varphi_k(x), t)  
\,d(\pi_1\circ \varphi) \wedge \cdots \wedge d(\pi_{m+1}\circ \varphi).
 \end{eqnarray*}
Thus
\be
T\times I=\sum_{k=1}^\infty \varphi'_{k\#} \Lbrack \theta'_k \Rbrack.
\ee
where 
\be
\varphi'_k: A_k \times I_\epsilon \to \bar{X}\times I_\epsilon
\textrm{ satisfies } \varphi'_k(x,t)=(\varphi_k(x), t)
\ee
and $\theta'_k\in L^1(A_k\times I_\epsilon, \N)$ 
satisfies $\theta'_k(x,t)=\theta_k(x)$.   Observe that
the images of these charts are disjoint and that 
\begin{eqnarray}
\mass(T \times I_\epsilon)&=&\sum_{k=1}^\infty
\mass(\varphi'_{k\#}\Lbrack \theta'_k \Rbrack )\\
&=&\sum_{k=1}^\infty \int_{A_k\times I_\epsilon} 
|\theta'_k| \mathcal{L}^{m+1}\\
&=&\sum_{k=1}^\infty \epsilon \int_{A_k} |\theta_k| \mathcal{L}^{m}\\
&=&\sum_{k=1}^\infty\epsilon
\mass(\varphi_{k\#}\Lbrack \theta_k \Rbrack )= \epsilon\mass(T).
\end{eqnarray}
The continuity of $T\times I_\epsilon$ now follows because
all integer rectifiable currents defined by parametrizations
are currents.

Observe also that if $A\subset \bar{X}$ and $(a_1, a_2)\subset I$
then
\begin{eqnarray}
||T \times I_\epsilon|| (A\times (a_1,a_2))
&=& \mass( (T \times I_\epsilon) \rstr (A \times (a_1,a_2)) )\\
&=&\sum_{k=1}^\infty
\mass((\varphi'_{k\#}\Lbrack \theta'_k \Rbrack )
\rstr (A \times (a_1,a_2)) )\\
&=&\sum_{k=1}^\infty \int_{(A \cap A_k)\times (a_1,a_2)} 
|\theta'_k| \mathcal{L}^{m+1}\\
&=&\sum_{k=1}^\infty (a_2-a_1) \int_{A\cap A_k} |\theta_k| \mathcal{L}^{m}\\
&=&\sum_{k=1}^\infty (a_2-a_1)
\mass(\varphi_{k\#}\Lbrack \theta_k \Rbrack  \rstr A)\\
&=& (a_2-a_1)\mass(T\rstr A)=(a_2-a_1) ||T||(A).
\end{eqnarray}
Thus $||T\times I_\epsilon||=||T||\times \mathcal{L}^1$.

To prove that $T\times I_\epsilon$ is an integral current, we
need only verify that the current $\partial (T \times I_\epsilon)$
has finite mass. 

Let $\tau_1, \dots, \tau_m \in \Lip(X \times I_\epsilon)$, such that $\partial \tau_i/\partial t$ is Lipschitz as well for $i = 1, \dots, m$. 
Applying the Chain Rule \cite{AK}Thm 3.5
and Lemma~\ref{deriv-in} (proven below), we have
$$
\partial (T \times I_\epsilon) (f, \tau_1,...,\tau_m)
+((\partial T) \times I_\epsilon) (f, \tau_1,...,\tau_m)=\qquad\qquad\qquad
\qquad\qquad\qquad
$$
\begin{eqnarray*}
&=& (T \times I_\epsilon) (1, f, \tau_1,...,\tau_m)+
\sum_{i=1}^{m} (-1)^{i+1}  \int_0^\epsilon 
\partial T\left(f_t \frac{\partial \tau_{i}}{\partial t}, \tau_{1t},...,\hat\tau_{it},...,\tau_{mt} \right) \, dt  \\
&=& \int_0^\epsilon 
T\left(\frac{\partial f}{\partial t}, \tau_{1t},...,\tau_{mt}
\right) \, dt -
\sum_{i=1}^{m} (-1)^{i+1}  \int_0^\epsilon 
T\left(\frac{\partial \tau_{i}}{\partial t}, f_t, \tau_{1t},...,\hat\tau_{it},...,\tau_{mt}
\right) \, dt  \\
&& \qquad +
\sum_{i=1}^{m} (-1)^{i+1}  \int_0^\epsilon 
T\left(1, f_t \frac{\partial \tau_{i}}{\partial t}, \tau_{1t},...,\hat\tau_{it},...,\tau_{mt} \right) \, dt  \\
&=& \int_0^\epsilon 
T\left(\frac{\partial f}{\partial t}, \tau_{1t},...,\tau_{mt}
\right) \, dt +
\sum_{i=1}^{m} (-1)^{i+1}  \int_0^\epsilon 
T\left(f_t, \frac{\partial \tau_{i}}{\partial t}, \tau_{1t},...,\hat\tau_{it},...,\tau_{mt}
\right) \, dt  \\
&=& \int_0^\epsilon 
T\left(\frac{\partial f}{\partial t}, \tau_{1t},...,\tau_{mt}
\right) \, dt +
\sum_{i=1}^{m}  \int_0^\epsilon 
T\left(f_t, \tau_{1t},...,\tau_{(i-1)t},\frac{\partial \tau_{i}}{\partial t}, 
\tau_{(i+1)t},...,\tau_{mt}
\right) \, dt  \\
&=& \int_0^\epsilon 
\frac{\partial}{\partial t} T\left(f_t, \tau_{1t},...,\tau_{mt}
\right) \, dt \\
&=& T (f_\epsilon, \tau_{1\epsilon},...,\tau_{m\epsilon})
-T (f_0, \tau_{10},...,\tau_{m0})\\
&=&\psi_{\epsilon \#} T (f, \tau_1,...,\tau_m)
- \psi_{0\#} T (f, \tau_1,...,\tau_m)\\
&=&T\times\partial I(f, \tau_1,...,\tau_m).
\end{eqnarray*}
By mollification in the $t$-variable and by using the continuity properties of currents, we conclude that for arbitrary $\tau_1, \dots, \tau_m \in \Lip(X \times I_\epsilon)$, it still holds that
\be
\partial (T \times I_\epsilon) (f, \tau_1,...,\tau_m)
+((\partial T) \times I_\epsilon) (f, \tau_1,...,\tau_m)=T\times\partial I(f, \tau_1,...,\tau_m).
\ee
Thus we have (\ref{bndbnd}).

Observe that $T\times \partial I$ is an integral current
because it is the sum of push forwards of integral
currents and that
\be \label{mass-product}
\mass(T\times \partial I_\epsilon) =2\mass(T).
\ee
Since we know products are rectifiable, $(\partial T)\times I_\epsilon$
is rectifiable and has finite mass $\le \epsilon \mass(\partial T)$. 
Thus applying (\ref{bndbnd}) we see that
\be
\mass(\partial(T\times I_\epsilon)) \le 
 \mass((\partial T) \times I_\epsilon) 
+ \mass(T \times \partial I_\epsilon)
\le \epsilon \mass(\partial T) + 2\mass(T).
\ee
Thus the current structure of $M\times I_\epsilon$ is
an integral current.

Lastly we verify that 
\be
\set (T\times I_\epsilon)=\set(T) \times I_\epsilon.
\ee
Given $(p,t)\in \bar{X}\times I_\epsilon$, then
the following statements are equivalent:
\begin{eqnarray}
(p,t) &\in&  \set (T\times I_\epsilon).\\
0 & < & \liminf_{r\to 0} \frac {||T\times I_\epsilon||(B_{(p,t)}(r)) }{r^{m+1}}.
\\
0 & < & \liminf_{r\to 0} \frac {||T\times I_\epsilon||(B_{p}(r)\times (t-r,t+r)) }{r^{m+1}}.\\
0 & < & \liminf_{r\to 0} \frac {2r ||T||(B_{p}(r)) }{r^{m+1}}.\\
0 & < & \liminf_{r\to 0} \frac {||T||(B_{p}(r)) }{r^{m}}.\\
p&\in & \set(T).
\end{eqnarray}
The proposition follows.
\end{proof}

\begin{lem}\label{deriv-in}
If $\pi_{it}$ and $\partial_t \pi_{it}$ are Lipschitz in $Z\times I$, 
and $T\in \intcurr_m(Z)$ then for almost every $t\in I$,
\be
\frac{\partial}{\partial t} T(\pi_{0t},...,\pi_{mt})
= \sum_{i=0}^{m}  
T\left(\pi_{0t},...,\pi_{(i-1)t},\frac{\partial \pi_{i}}{\partial t}, 
\pi_{(i+1)t},...,\pi_{mt} \right)
\ee
\end{lem}

\begin{proof}
This follows from the multilinearity of $T$, the usual expansion of
the difference quotient as a sum of difference quotients in which
one term changes at a time, the fact the $T$ is continuous
with respect to pointwise convergence and that the difference
quotients have pointwise limits for almost every $t\in I$ because
the $\pi_{i}$ are Lipschitz in $t$.   
\end{proof}

The following proposition will be applied later when studying 
limits under intrinsic flat and Gromov-Hausdorff convergence.

\begin{prop} \label{interval-flat}
Suppose $M_i^m=(X_i, d_i, T_i)$ are integral current spaces and $\epsilon >0$, 
then
\be
d_{\mathcal{F}}(M_1^m\times I_\epsilon, M_2^m \times I_\epsilon)
\le (2+\epsilon) d_{\mathcal{F}}(M_1^m, M_2^m).
\ee
and, when $M_i$ are precompact,
\be
d_{GH}(M_1^m\times I_\epsilon, M_2^m \times I_\epsilon)
\le  d_{GH}(M_1^m, M_2^m),
\ee
\end{prop}

\begin{proof}
Let $\delta > 0$. 
There exists a 
metric space $Z$ and isometric embeddings $\varphi_i: X_i \to Z$,
and integral currents $A, B$ on $Z$ such that
\be
\varphi_{1\#} T_1- \varphi_{2\#} T_2 = A +\partial B
\ee
and 
\be
d_{\mathcal{F}}(M_1^m, M_2^m) \leq \mass(A) + \mass(B) + \delta.
\ee
Setting $Z'= Z\times I_\epsilon$ endowed
with the product metric, we have isometric embeddings 
$\varphi'_i: X_i\times I_\epsilon \to Z'$ and we have integral currents
$A'=A\times I_\epsilon$ and $B'=B \times I_\epsilon$ such that
\begin{eqnarray}
\varphi'_{1\#} (T_1\times I_\epsilon)- \varphi'_{2\#} (T_2 \times I_\epsilon)
& =&
(\varphi'_{1\#} T_1) \times I_\epsilon- (\varphi_{2\#} T_2 )\times I_\epsilon\\
& =&
(\varphi'_{1\#} T_1- \varphi_{2\#} T_2 )\times I_\epsilon\\
&=&(A+\partial B) \times I_\epsilon \\
&=& A \times I_\epsilon - \partial (B\times I_\epsilon) - B \times (\partial I_\epsilon).
\end{eqnarray}
Thus by Proposition~\ref{prop-times-I} and (\ref{mass-product}) we have
\begin{eqnarray}
d_{\mathcal{F}}(M_1^m\times I_\epsilon, M_2^m \times I_\epsilon)
&\le& 
\mass(A \times I_\epsilon) + \mass(B\times I_\epsilon) + \mass( B \times (\partial I_\epsilon)) \\
&\le& 
\epsilon \mass(A) + \epsilon \mass(B) + 2\mass(B)\\ 
&\le & (2+\epsilon)\mass(A) + (2+\epsilon) \mass(B) \\
&=& (2 + \epsilon)( d_{\mathcal{F}}(M_1^m, M_2^m) + \delta).
\end{eqnarray}
Finally, we let $\delta \downarrow 0$.

To see the Gromov-Hausdorff estimate, one needs only observe that
whenever $Y_1 \subset T_r(Y_2) \subset Z$, then
\be
Y_1\times I_\epsilon \subset T_r(Y_2 \times I_\epsilon) \subset Z\times I_\epsilon.
\ee
\end{proof}

\vspace{.4cm}
\subsection{Slices and Spheres}

While balls are a very natural object in metric spaces, a more important
notion in integral current spaces is that of a slice.  The following
proposition follows immediately from the Ambrosio-Kirchheim Slicing
Theorem (c.f. Theorem~\ref{theorem-slicing} and Remark~\ref{rmrk-slicing}):

\begin{prop}\label{prop-slicing}
Given an $m$ dimensional integral current space $M=(X,d,T)$ and Lipschitz
functions $F: X \to \R^k$ where $k<m$, then 
for almost every $t\in \R^k$, we
can define an $m-k$ dimensional integral current space called the slice
of $(X,d,T)$:
\be \label{defn-slice-1}
\Slice(M,F,t)=\Slice(F,t)=(\set<T,F,t>, d, <T,F,t>)
\ee
where $<T,F,t>=<T, F_1,...,F_k, t_1,...,t_k>$ 
is an integral current on $\bar{X}$ defined
using the Ambrosio-Kirchheim Slicing Theorem
and $\set<T,F,t>\subset F^{-1}(t)$.   We can
integrate the masses of slices to obtain lower bounds of the
mass of the original space:
\be
\int_{t\in\R^k} \mass(\Slice(M,F,t)) \, \mathcal{L}^k \le \prod_{j=1}^k \Lip(F_j)\, \mass(T).
\ee
and $\partial \Slice(M,F,t)= (-1)^k \Slice(\partial M, F, t)$.
\end{prop}

\begin{proof}
This proposition follows immediately from the Ambrosio-Kirchheim
Slicing Theorem 5.6 using the fact that $F$ has a unique extension
to $\bar{X}$ and Defn 2.5.   The last part follows from Lemma 5.8.
\end{proof}

\begin{rmrk}
Observe that in Example~\ref{concentric-spheres}
where $M$ is a 1 dimensional current space formed by
concentric circles and a center point $p_0=0$,  if
$F(x)=d(x,p_0)$ then almost every slice is the $\bf{0}$ integral
current space.   
\end{rmrk}

\begin{lem}\label{lem-sphere}
Given an $m$ dimensional integral current space $(X,d,T)$ and a point $p$
then for almost every $r\in \R$, we
can define an $m-1$ dimensional
integral current space called the sphere
about $p$ of radius $r$:
\be
\Sphere(p,r)=\Slice(\rho_p,r) 
\ee
On a Riemannian manifold with boundary,
\be
\Sphere(p,r)(f, \pi_1,...\pi_{m-1})
=\int_{\rho_p^{-1}(r)} f\, d\pi_1\wedge \cdots \wedge d\pi_{m-1}
\ee
is an integral current space for all $r\in \R$.
\end{lem}

\begin{proof}
This follows from Proposition~\ref{prop-slicing} and the 
Ambrosio Kirchheim Slicing Theorem
(c.f. Theorem~\ref{theorem-slicing} and the fact
that $\Lip(\rho_p)=1$.   The Riemannian part
follows from Stoke's Theorem and the fact that spheres of all radii in 
Riemannian manifolds have finite volume as can be seen either
by applying the Ricatti equation or Jacobi fields.
\end{proof}

Observe the distinction between the sphere and the boundary
of a ball in Lemma~\ref{lem-sphere} when $M$ has boundary.

Next we examine the setting when we do not hit the boundary:

\begin{lem}\label{lem-bndry-away}
If
 $\set (\partial T) \cap \bar{B}(p,R) \subset \bar{X}$
is empty then for almost every $r\le R$
\be
\Sphere(p,r)=\partial \, S(p,r).
\ee
Furthermore, 
\be  \label{eqn-most-balls}
\int_0^R \mass\left( \partial S\left(p,r\right)  \right)  \,d \mathcal{L}\left(r\right)  \le   \mass \left(S\left(p,R\right)\right).
\ee
In particular, on an open Riemannian manifold, for any $p\in M$,
there is a sufficiently small $R>0$ such that this lemma holds.
On a Riemannian manifold without boundary, these hold for
all $R>0$.
\end{lem}

\begin{proof}
This follows from Proposition~\ref{prop-slicing}
and Theorem~\ref{theorem-slicing}.
\end{proof}

\begin{lem}\label{lem-rho}
Given an $m$ dimensional integral current space $(X,d,T)$ and a 
$\rho: X \to \R$ a Lipschitz function with $\Lip(\rho)\le 1$ 
then for almost every $r\in \R$, we
can define an $m-1$ dimensional
integral current space, $\Slice(\rho,r)$ where
\be  
\int_{-\infty}^{\infty} \mass\left(\Slice(\rho,r)  \right)  \,d \mathcal{L}\left(r\right)  \le   \mass \left(T\right).
\ee
On a Riemannian manifold with boundary
\be
\Slice(\rho,r)(f, \pi_1,...\pi_{m-1})
=\int_{\rho^{-1}(r)} f\, d\pi_1\wedge \cdots \wedge d\pi_{m-1}
\ee
is defined for all $r\in \R$.
\end{lem}

\begin{proof}
This follows from Proposition~\ref{prop-slicing}
and Theorem~\ref{theorem-slicing}.
\end{proof}

\begin{lem}\label{lem-rho-more}
Given an $m$ dimensional integral current space $(X,d,T)$ and a 
$\rho: X \to \R^k$ have $\Lip(\rho_i)\le 1$
then for almost every $r\in \R^k$, we
can define an $m-k$ dimensional
integral current space, $\Slice(\rho,r)$ where
\be 
\int_{\R^k} \mass\left(\Slice(\rho,r)  \right)  \,d \mathcal{L}\left(r\right)  \le   \mass \left(T\right).
\ee
\end{lem}

\begin{proof}
This follows from Proposition~\ref{prop-slicing}
and Theorem~\ref{theorem-slicing}.
\end{proof}

\begin{rmrk}
On a Riemannian manifold with boundary
\be
\Slice(\rho,r)(f, \pi_1,...\pi_{m-1})
=\int_{\rho^{-1}(r)} f\, d\pi_1\wedge \cdots \wedge d\pi_{m-1}
\ee
is defined for all $r\in \R$ such that
$\rho_p^{-1}(r)$ is $m-1$ dimensional.   By the above lemma
this will be true for almost every $r$.   Note, however, that if
$\rho_i$ are distance functions from poorly chosen points, the
slice may be the $0$ space for almost every $r$ because
$\rho_p^{-1}(r)=\emptyset$.  
This occurs for example on the standard 
three dimensional sphere if we take $\rho_1, \rho_2$ to 
be distance functions from opposite poles.
\end{rmrk}

\vspace{.4cm}
\subsection{Filling Volumes of Spheres and Slices} 

The following lemmas were applied without proof in \cite{SorWen1}.
We may now easily prove them.   First
recall Definition~\ref{defn-filling-volume} for the notion of filling
volume used in this paper.

\begin{lem}\label{filling-balls}
Given an integral current space, $M^m=(X,d,T)$,
for all $p\in \bar{X}$ and almost every $r>0$,
\be
\mass(S(p,r)) \ge \fillvol(\partial S(p,r)).
\ee
Thus $p\in \bar{X}$ lies in $X=\set(T)$ if
\be\label{ess-lim-inf-1}
 \ess \liminf_{r\to 0} \fillvol(\partial S(p,r))/r^m >0.
\ee
Here we have the essential lim inf 
which is a lim inf as $r\to 0$ where
$r$ are selected from a set of full measure.
\end{lem}

\begin{proof}
This follows immediately from the definition of filling
volume in Defn~\ref{defn-filling-volume} and (\ref{fill-est}), 
the definition of $S(p,r)$ in  which is only defined
for almost every $r>0$, and the definition of $\set(T)$.
\end{proof}

Note that the converse of Lemma~\ref{filling-balls} is
not true
as can be seen by observing that in Example~\ref{concentric-spheres} we have the point $0\in X=\set(T)$, but
$\partial S(0,r)={\bf{0}}$ for almost every $r>0$.   So
(\ref{ess-lim-inf-1}) fails for $p=0$ in this example.  The
same is true for (\ref{ess-lim-inf-2}) in the next lemma.

\begin{lem}\label{filling-sphere} 
Given an integral current space, $M=(X,d,T)$.  If
$B_p(R) \cap \partial M=\emptyset$ then
for almost every $r\in (0,R)$ we have
\be
\mass(S(p,r)) \ge \fillvol(\Sphere(p,R)). 
\ee
Thus if $\partial M=\emptyset$, we know that $p\in \bar{X}$
lies in $X=\set(T)$ if  
\be\label{ess-lim-inf-2}
\ess\liminf_{r\to 0} \fillvol(\Sphere(p,r))/r^m >0.
\ee
\end{lem}

\begin{proof}
This follows immediately from Lemma~\ref{lem-bndry-away}
and Lemma~\ref{filling-balls}.
\end{proof}

Theorem 4.1 of \cite{SorWen1} can be
stated as follows:

\begin{thm} [Sormani-Wenger]
Suppose $M^m=(X,d,T)$ is a compact Riemannian manifold 
such that there exists $r_0>0, k>0$ such that
$\bar{B}(p, kr_0)\cap \partial M=\emptyset$ and
every 
$B(x,r)\subset\bar{B}(p,r_0)$ is contractible within 
$B(x, kr)\subset \bar{B}(p, r_0)$ then $\exists C_k$ such that
\be
\vol(\bar{B}(x,r)) = ||T||(\bar{B}(x,r)|| \ge \fillvol(\partial S(p,r))\ge C_k r^m.
\ee
\end{thm}

This theorem essentially follows from a result of Greene-Petersen \cite{Greene-Petersen} combined with Lemma~\ref{filling-sphere}.  The statement in \cite{SorWen2}
applies to a more general class of spaces and requires a much more 
subtle proof involving Lipschitz extensions.

\vspace{.4cm}
\subsection{Sliced Filling Volumes of Balls}

Spheres aren't the only slices whose filling volumes
may be used to estimate the volumes of balls.  We define the
following new notions:

\begin{defn}\label{sliced-filling-vol}
Given an integral current space, $M^m=(X,d,T)$ and  
$F_1, F_2,...F_k: M \to \R$ with $k\le m-1$ are 
Lipschitz functions with Lipschtiz constant $\Lip(F_j)=\lambda_j$ 
then
we define the sliced filling volume of 
$\partial S(p,r)\in \intcurr_{m-1}(\bar{X})$, to be
\be
\SF(p,r,F_1,...,F_k)
=\int_{t\in A_r} \fillvol(\partial\Slice(S(p,r),F,t)) \, \mathcal{L}^k. 
\ee
where 
\be
A_r=[\min F_1, \max F_1]\times [\min F_2, \max F_2]\times
\cdots \times [\min F_k, \max F_k]
\ee
where $\min F_j=\min\{F(x): x\in \bar{B}_p(r)\}$
and $\max F_j=\max\{F(x): x\in \bar{B}_p(r)\}$.
Given $q_1,...,q_k\in M$, we set,
\be
\SF(p,r,q_1,...,q_k)=
\SF(p,r,\rho_1,...,\rho_k) \textrm{ where }\rho_i(x)=d_X(q_i,x).
\ee
\end{defn}

\begin{defn} \label{defn-SF_k}
Given an integral current space $M^m$ and $p\in M^m$,
then for almost every $r$, we can define the
$k^{th}$ sliced filling,
\be
\SF_k(p,r) = \sup\{ \SF(p, r, q_1,...,q_k): q_i \in \partial B_p(r)\}
\ee
where $\partial B_p(r)$ is the boundary of the metric ball about $p$.
In particular,
\be
\SF_0(p,r) =  \SF(p, r)=\fillvol(\partial S(p,r)).
\ee
\end{defn}

\begin{lem}\label{filling-slices}
Given an integral current space, $M^m=(X,d,T)$ and  
$F_1, F_2,...F_k: M \to \R$ with $k\le m-1$ are 
Lipschitz functions with Lipschitz constants, $\Lip(F_j)=\lambda_j>0$, 
then
\be\label{filling-slices-1}
\mass(S(p,r))\ge \prod_{j=1}^k \lambda_j^{-1}  
\SF(p, r, F_1,...,F_k).
\ee
Thus $p\in \bar{X}$
lies in $X=\set(T)$ if  there exists $F_i: M\to \R$ as above
such that
\be
\liminf_{r\to 0} \frac{1}{r^m} \SF(p,r,F_1,..., F_k) >0.
\ee
Applying (\ref{filling-slices-1}) to $F_j=\rho_{q_{j,r}}$
where $q_{i,r}$ achieve the supremum in 
Definition~\ref{defn-SF_k}, we see that 
\be\label{filling-slices-2}
\mass(S(p,r))\ge \SF(p, r, q_{1,r},...,q_{k,r}) =\SF_k(p,r).
\ee
Thus
$p\in \bar{X}$
lies in $X=\set(T)$ if 
\be
\liminf_{r\to 0} \frac{1}{r^m} \SF_k(p,r) >0.
\ee
Conversely if $\partial S(p,r)\neq 0$ then for $k =0$ we have
\be 
\SF_k(p,r) \neq 0.
\ee
\end{lem}

\begin{proof}
By Proposition~\ref{prop-slicing} 
we know
\begin{eqnarray}
\mass(S(p,r)) &\ge& \prod_{j=1}^k \lambda_j^{-1} \mass(S(p,r) \rstr dF)\\
&\ge& \prod_{j=1}^k \lambda_j^{-1} \int_{t\in\R^k} \mass(\Slice(S(p,r),F,t)) \, \mathcal{L}^k \\
&=& \prod_{j=1}^k \lambda_j^{-1} \int_{t\in A} \mass(\Slice(S(p,r),F,t)) \, \mathcal{L}^k. 
\end{eqnarray}
Then (\ref{filling-slices-1}) follows because $k\le m-1$ implies each
slice is at least 1 dimensional, combined with (\ref{fill-est})
and the fact that $\partial<S,F,t>=-<\partial S, F,t>$.

The converse follows because $\SF_0(p,r)=\fillvol(\partial S(p,r))>0$
when $S(p,r)\neq 0$.
\end{proof}

\vspace{.4cm}
\subsection{Uniform $\SF_k$ and Gromov-Hausdorff Compactness}

We now prove a new Gromov-Hausdorff Compactness Theorem:

\begin{thm} \label{SF_k-compactness}
If $M_i^m=(X_i, d_i, T_i)$ are integral current spaces with a uniform
upper bound on $\vol(M_i) \le V_0$ and diameter $\diam(M_i)\le D_0$, and a uniform $r_0>0$, $C:(0, r_0]\to \R^+$, such that
we have a uniform lower bound on
the $k^{th}$ sliced filling
\be
\SF_k(p,r) \ge C(r)>0 \textrm{ for almost every } r\in (0,r_0]
\ee
for all $p\in M_i$, for all $i$,
then a subsequence $(X_i, d_i)$ converges in the Gromov-Hausdorff
sense to a limit space $(Y, d_Y)$.
\end{thm}

Later we will prove that the subsequence converges in the intrinsic
flat sense to the same limit space when $C(r)\ge C_{SF}r^m>0$
[Theorem~\ref{SF_k-compactness-2}].

\begin{proof}
For any $p$ in any $M_i$, there exist $q_1,..., q_k$ such that
\be
\SF(p, r, q_1,...,q_k) \ge C(r)/2 >0.
\ee
So by Lemma~\ref{filling-slices},
$\mass(S(p,r)) \ge C(r)/2$.  Thus the number of disjoint balls
of radius $r$ in $M_i$ is $\le 2V_0/C(r)$.   So we may apply Gromov's
Compactness Theorem.
\end{proof}

\vspace{.4cm}
\subsection{Filling Volumes of $0$ Dimensional Spaces}

Before proceeding we need the following lemma:

\begin{lem}\label{filling-one}
Let $M$ be an integral current space.
Suppose $S\in \intcurr_1(M)$ such that
$\partial S \neq \bf{0}$.  Then
$\set(\partial S) =\{p_1,...,p_N\}$ with $N\ge 2$ and
\be\label{filldown2}
\partial S(f)=\sum_{i=1}^N \sigma_i \theta_i f(p_i) 
\ee
where $\theta_i \in \Z^+$ and $|\sigma_i|=1$ and
\be \label{filldown1}
\fillvol(\partial S) \ge \max_{j=1..N} \left( |\theta_j| 
                                  \min_{i\neq j} d_X(p_i,p_j) \right)>0
\ee
In particular
\begin{eqnarray} \label{filldown}
\fillvol(\partial S) &\ge& \inf\left\{ d_X(p_i, p_j): \, i,j \in \{1,2...,N\} \, \right\}\\
&\ge& \inf\left\{d(x,y): \,\, x\neq y,\,\, x,y\in \set (\partial S)\right\}>0.
\end{eqnarray}
\end{lem}

\begin{proof}
Recall that by Remark~\ref{bndry-1-space}, $\partial S$ satisfies
(\ref{filldown2}) where $\sum_{i=1}^N \sigma_i \theta_i=0$.
So $N\ge 2$ when $\partial S \neq \bf{0}$.

Suppose $M'=(Y, d_Y, T)$ is any one dimensional integral
current space with a current preserving isometry 
$\varphi: \set(\partial M') \to \set(\partial S)\subset \bar{X}$
so that
\be 
\varphi_{\#}\partial T =\partial S \in \intcurr_0(M)
\ee
and $d_X(\varphi(y_1), \varphi(y_2))=d_Y(y_1,y_2)$ for all
$y_1, y_2 \in \set(T)\subset Y$.   In particular there
exist distinct points
\be
p_j'=\varphi^{-1}(p_j) \in \bar{Y} 
\ee
such that for any Lipschitz $f: \bar{Y} \to \R$ we have 
\be
T(1,f)=\partial(T)(f) = \sum_{i=1}^{N} \sigma_i \theta_i f(p'_i).
\ee
By (\ref{eqn-mass}) we have
\be
|T(1,f) | \le \Lip(f) \mass(T).
\ee
Let $f_j(y)=\min_{i\neq j} d_Y(y, p'_i)$.  Then we have, $\Lip(f_j)=1$, so
\begin{eqnarray}
\mass(T) &\ge& \left| \sum_{i=1}^{N} \sigma_i \theta_i f_j(p'_i) \right|\\
&\ge& \theta_j f_j(p'_j)=\theta_j \min_{i\neq j} d_Y(p'_i, p'_j)\\
&=& \theta_j \min_{i\neq j} d_X(p_i, p_j)
\end{eqnarray}
Taking an infimum over all $T$, we have,
\be
\fillvol(\partial S) \ge\theta_j \min_{i\neq j} d_X(p_i, p_j).
\ee
As this is true for all $j=1..N$, we have (\ref{filldown1}).
Since $\theta_j \in \Z^+$, we have the simpler
lower bound given in (\ref{filldown}).
\end{proof}

\vspace{.4cm}
\subsection{Masses of Balls from Distances}

Here we provide a lower bound on the mass of a ball
using a sliced filling volume and estimates on 
the filling volumes of $0$ dimensional currents.  First we
introduce the notation:
\be\label{P-now}
P(p,r, t_1,..., t_{m-1})= \rho_p^{-1}(r) \cap \rho_{p_1}^{-1}(t_1)\cap
\cdots \cap \rho_{p_{m-1}}^{-1}(t_{m-1}).
\ee

\begin{thm}\label{dist-set-2}
Given an integral current space, $M^m=(X,d,T)$ and
points $p_1,... p_{m-1} \in X$, then
then if $\bar{B}_p(R)\cap \set (\partial T)=\emptyset$ we have
for almost every $r\in (0,R)$,
\begin{eqnarray}
\mass(S(p,r)) &\ge&
\SF(p,r,p_1,...,p_{m-1}) \\
&\ge&  \int_{s_1-r}^{s_1+r} \cdots \int_{s_{m-1}-r}^{s_{m-1}+r }
h(p,r,t_1,..., t_{m-1})
\, dt_1dt_2...dt_{m-1}
\end{eqnarray}
where $t_i=d(p_i,p_0)$ and
\begin{eqnarray*}
\,\,h(p,r, t_1,..., t_{m-1})&=&
\inf\left\{d(x,y): \,\, x\neq y,\,\, x,y\in  P(p, r, t_1,...,t_{m-1})\right\}\textrm{ when }\\
&&
\,\,P(p,r, t_1,..., t_{m-1}) \textrm{ of (\ref{P-now}) is a nonempty discrete set of points and }\\
\,\,h(p,r,t_1,..., t_{m-1})&=&0 \,\,\textrm{ otherwise}.
\end{eqnarray*}
\vspace{.1cm}
Thus $p\in \bar{X} \setminus Cl(\set(\partial T))$ lies in $X=\set(T)$ if
\be
\liminf_{r\to 0}\,\, (1/r^m)
\int_{t_1=s_1-r}^{s_1+r} \cdots \int_{t_{m-1}=s_{m-1}-r}^{s_{m-1}+r }
h(p,r,t_1,..., t_{m-1})
\, dt_1dt_2...dt_{m-1}>0
\ee
\end{thm}

Theorem~\ref{dist-set-2} is in fact a special case of the following theorem:

\begin{thm}\label{dist-set}
Given an integral current space, $M^m=(X,d,T)$ and  
Lipschitz functions,
$F_1, F_2,...F_{m-1}: M \to \R$,  
with Lipschitz constants, $\Lip(F_j)=\lambda_j$, 
then for almost every $r>0$
\be
\mass(S(p,r)) \,\ge\, \SF(p, r, F_1,...,F_k)\,\ge
\,\prod_{j=1}^k \lambda_j^{-1} 
\,\,\int_{t\in A_r} 
h(p,r,F.t)
\, d\mathcal{L}^k 
\ee
where
\begin{eqnarray*}
h(p,r,F,t)&=&\inf\{d(x,y): \,\, x\neq y,\,\, x,y\in \set(\partial \Slice(S(p,r), F, t) )\}>0\\
&&\,\, \textrm{ when $\partial \Slice(S(p,r),F,t)\in \intcurr_0(\bar{X})\setminus\{ \bf{0}\}$ and } \\
h(p,r,F,t)&=&0 \textrm{ otherwise,}
\end{eqnarray*}
 and 
where 
\be
A_r=[\min F_1, \max F_1]\times [\min F_2, \max F_2]\times
\cdots \times [\min F_k, \max F_k]
\ee
with $\min F_j=\min\{F_j(x): x\in \bar{B}_p(r)\}$
and $\max F_j=\max\{F_j(x): x\in \bar{B}_p(r)\}$.
\end{thm}

Before presenting the proof we give two important examples:

\begin{ex}\label{E-dist-set}
On Euclidean space, $\E^m$, taking $F_i: \E^m\to \R$ to be a collection of
perpendicular coordinate functions for $i=1..m$, 
$F_i(x_1,...,x_m)=x_i$, we have $\lambda_i=1$ and
\be
h(p,r,F_1,...,F_{m-1}, t_1,...,t_{m-1})=2\sqrt{r^2 -(t_1^2+\cdots +t_{m-1}^2)}.
\ee
So
\be
\omega_mr^m =\mass(S(p,r))\ge \SF(p, r, F_1,...,F_k) =\omega_mr^m. 
\ee
\end{ex}

\begin{ex}\label{S-dist-set}
On the standard sphere, $S^2$, taking $p_1\in \partial B_p(\pi/2)$
and $r=\pi/2$ and $F_1(x)=d(p_1, x)$, then
\be
h(p,\pi/2, F_1, t) = \min\{ 2t, 2(\pi-t)\}
\ee
because the distances are shortest if one travels within the
great circle, $\partial B_p(\pi/2)$.
So
\be
2 \pi =\vol(S^2_+) =\mass(S(p, \pi/2)) \ge \SF(p,\pi/2, F_1)
\ee
with
\begin{eqnarray}
\SF(p,\pi/2, F_1)&=&\int_0^\pi h(p, \pi/2, F_1, t) \, dt \\
&=& 2\int_0^{\pi/2} 2 t \, dt = 2 (\pi/2)^2 =\pi^2/2.
\end{eqnarray}
\end{ex}

\begin{proof}
Theorem~\ref{dist-set-2} follows from Theorem~\ref{dist-set}
taking $F(x)=(F_1(x),...,F_{m-1}(x))$ where $F_i(x)=\rho_{p_i}(x)$.
When $\bar{B}(p,r)\cap \set{\partial T}=\emptyset$, then
for almost every $r \in \R, t_1\in \R,...,t_{m-1}\in\R$
$\partial \Slice(S(p,r),F,t))\in \intcurr_0(\bar{X})$ and
\be
\set(\partial \Slice(S(p,r),F,t))=\rho_p^{-1}(r)\cap F_1^{-1}(t_1)\cap
\cdots \cap F_{m-1}^{-1}(t_{m-1}).
\ee
so this set either has $0$ points or at least two points.  
\end{proof}

\begin{proof}
Theorem~\ref{dist-set} is proven by applying 
Lemma~\ref{filling-slices} to $F$ and then computing the
filling volume of the $0$ dimensional current,
 $\partial(\Slice(S(p,r),F,t))$, using 
Lemma~\ref{filling-one} stated and proven below.
Observe that if $t_i< s_i -r$ or $t_i>s_i+r$ then
$h(p,r,t_1,...,t_{m-1})=0$ because
$\rho_p^{-1}(r)\cap \rho_{p_i}^{-1}(t_i)=\emptyset$.
\end{proof}

\begin{rmrk}
Naturally we could combine Theorem~\ref{dist-set}
with any other lower bound on the filling volumes of
$0$ dimensional sets, like, for example, (\ref{filldown1}).
\end{rmrk}

\vspace{.4cm}
\subsection{Tetrahedral Property}

Theorem~\ref{dist-set-2} allows us to estimate the masses of
balls using a tetrahedral property (see Figure~\ref{fig-tetra-prop}).

\begin{defn} \label{defn-tetra}
Given $C>0$ and $\beta\in (0,1)$, a metric space $X$
is said to have the $m$ dimensional
$C,\beta$-tetrahedral property at a point $p$
for radius $r$ if
one can find points $p_1,...p_{m-1}\subset \partial B_p(r)\subset\bar{X}$,
such that 
\be
h(p,r, t_1,...,t_{m-1}) \ge Cr \qquad \forall (t_1,...,t_{m-1}) \in [(1-\beta)r, (1+\beta)r]^m
\ee
where 
\be
h(p,r, t_1,..., t_{m-1})=
\inf\left\{d(x,y): \,\, x\neq y,\,\, x,y\in  P(p, r, t_1,...,t_{m-1})\right\}
\ee
when
\be
P(p,r, t_1,..., t_{m-1})= \rho_p^{-1}(r) \cap \rho_{p_1}^{-1}(t_1)\cap
\cdots \cap \rho_{p_{m-1}}^{-1}(t_{m-1})
\ee
is nonempty and $0$ otherwise.
Observe that for almost every $(t_1,...,t_{m-1})$,
$P(p,r,t_1,..., t_{m-1})$ is the set of a $0$ current
and is thus a discrete set of
points. 
\end{defn}

\begin{ex} \label{ex-torus-Euclidean}
On Euclidean space, $\E^3$, taking $p_1, p_2 \in \partial B(p,r)$ to
such that $d(p_1,p_2)=r$, then there exists exactly two
points $x,y \in P(p,r,r, r)$ each forming a tetrahedron with
$p, p_1, p_2$.  See Figure~\ref{fig-tetra-prop}. As we vary $t_1, t_2\in (r/2, 3r/2)$, we still have
exactly two points in $P(p,r,t_1, t_2)$.   By scaling we see that
\be
h(p,r, t_1, t_2)= r h(p,1, t_1/r, t_2/r) \ge 
C_{\E^3} r
\ee
where
\be
C_{\E^3}=\inf\{h(p,1,s_1, s_2): \, s_i\in (1/2, 3/2)\}>0
\ee
could be computed explicitly.   So $\E^3$ satisfies the
$C_{\E^3}, (1/2)$ tetrahedral property at $p$ for all $r$.
\end{ex}

\begin{ex} \label{ex-torus-tetra}
On a torus, $M_\epsilon^3=S^1 \times S^1 \times S_\epsilon^1$ where
$S_\epsilon^1$ has been scaled to have diameter $\epsilon$
instead of $\pi$, we see that $M^3$ satisfies the
$C_{\E^3}, (1/2)$ tetrahedral property at $p$ for all $r< \epsilon/4$.
By taking $r<\epsilon/4$, we guarantee that the shortest
paths between $x$ and $y$ stay within the ball $B(p,r)$ allowing
us to use the Euclidean estimates.   If $r$ is too large,
$P(p, r, t_1, t_2)=\emptyset$.
\end{ex}

\begin{rmrk}\label{rmrk-tetra-2}
On a Riemannian manifold or an integral current space, we
know that $P(p, r, t_1,...,t_{m-1})$ is the set of a 0 current
which is a boundary.   So if it is not empty, it has at least two points,
one with positive weight and one with negative weight.  
\end{rmrk}

\begin{rmrk}
It is not just a simple application of the triangle inequality to
proceed from knowing
$h(p,r,r,...,r)\ge Cr$ to having $h(p,r, t_1,..., t_{m-1})\ge 
C_2r$.  There is the possibility that
$P(p,r,t_1,..., t_{m-1})$ is empty or has a closest pair
of points both near a single point of $C(p,r,r,...,r)$
even in a Riemannian manifold.   However one expects
the same type of curvature conditions that would lead
to control of $h(p,r,...,r)$ could be used to study
$h(p,t_1,...,t_{m-1})$.
\end{rmrk}

\begin{rmrk} \label{rmrk-tetra-R}
On a manifold with sectional curvature bounded below, one
should have the $C, 1/2$ tetrahedral property at any point
$p$ as long as $r< \injrad(p)/4$ where $C$
depends on the lower sectional curvature bound.   This
should be provable using the Toponogov Comparison Theorem.
     One would like to replace the
condition on injectivity radius with radius depending upon
a lower bound on volume.  Work in this direction is under 
preparation by the second author's doctoral students.    Note that
there is no uniform tetrahedral property
on manifolds with positive scalar curvature even when the
volume of the balls are uniformly bounded below by that
of Euclidean balls [Remark~\ref{rmrk-scalar-cancellation}].   
With lower bounds on Ricci
curvature one might expect to have the $C, 1/2$ tetrahedral
property or an integral version of this property.  Again a
uniform lower bound on volume will be necessary as seen
in the torus example above.
\end{rmrk}

\vspace{.4cm}
\subsection{Integral Tetrahedral Property}

For our applications we need only the following property
which is clearly holds at any point with the tetrahedral property:

\begin{defn} \label{defn-int-tetra}
Given $C>0$ and $\beta\in (0,1)$, a metric space $X$
is said to have the $m$ dimensional integral
$C,\beta$-tetrahedral property at a point $p$
for radius $r$ if
one can find points $p_1,...p_{m-1}\subset \partial B_p(r)\subset\bar{X}$,
such that 
\be
\int_{t_1=(1-\beta)r}^{(1+\beta)r} \cdots \int_{t_{m-1}=(1-\beta)r}^{(1+\beta)r }
h(p,r,t_1,...t_{m-1}))
\, dt_1dt_2...dt_{m-1}
\ge C (2\beta)^{m-1}r^m.
\ee
\end{defn}

\begin{prop}
If $X$ is a metric space that satisfies the $C \beta$ tetrahedral
property at $p$ for radius $r$ then it
has the $C \beta$ integral tetrahedral property.
\end{prop}

\begin{proof}
\begin{eqnarray*}
\int_{t_1=(1-\beta)r}^{(1+\beta)r} \cdots \int_{t_{m-1}=(1-\beta)r}^{(1+\beta)r }
&h(p,r,t_1,...t_{m-1})&
\, dt_1dt_2...dt_{m-1} \,\,\,\ge\\
&\ge& \int_{t_1=(1-\beta)r}^{(1+\beta)r} \cdots \int_{t_{m-1}=(1-\beta)r}^{(1+\beta)r }
C R
\, dt_1dt_2...dt_{m-1}\\
&\ge& CR \,\left( (1+\beta)r -(1-\beta)r \right)^{m-1} 
\end{eqnarray*}
\end{proof}

\vspace{.4cm}
\subsection{Tetrahedral Property and Masses of Balls}

\begin{thm}\label{tetra-ball}
Suppose $(X,d,T)$ is an integral current space and $p\in X$ has
$\bar{B}_p(R) \cap \set{\partial T}=0$.  Then for almost
every $r\in (0,R)$, if the
$m$ dimensional (integral)
$C,\beta$-tetrahedral property at a point $p$
for radius $r$ holds on $\bar{X}$ 
then
\be
\mass(S(p,r))\ge \SF_{m-1}(p,r)\ge C (2\beta)^{m-1}r^m.
\ee
\end{thm}

\begin{proof}
By Theorem~\ref{dist-set-2} with $s_i=r$ 
we have 
\begin{eqnarray*}
\mass(S(p,r)) &\ge &\SF(p,r,p_1,...,p_{m-1})\\
&\ge &\int_{t_1=s_1-r}^{s_1+r} \cdots \int_{t_{m-1}=s_{m-1}-r}^{s_{m-1}+r }
h(P_{(r,t_1,...t_{m-1})})
\, dt_1dt_2...dt_{m-1}\\
&\ge &\SF(p,r,p_1,...,p_{m-1})\\
&\ge &\int_{t_1=0}^{2r} \cdots \int_{t_{m-1}=0}^{2r }
h(P_{(r,t_1,...t_{m-1})})
\, dt_1dt_2...dt_{m-1}\\
&>&
\int_{t_1=(1-\beta)r}^{(1+\beta)r} \cdots \int_{t_{m-1}=(1-\beta)r}^{(1+\beta)r }
h(p,r,t_1,...t_{m-1}))
\, dt_1dt_2...dt_{m-1}\\
&>& 
C (2\beta)^{m-1}r^m
\end{eqnarray*}
\end{proof}

\begin{thm}\label{tetra-manifold}
Suppose $p_0$ lies in a Riemannian manifold with boundary, 
$M$, and
$B_{p_0}(R)\cap \partial M=\emptyset$.
For almost every $r\in (0,R)$, if
the
$m$ dimensional (integral)
$C,\beta$-tetrahedral property at a point $p$
for radius $r$ holds
then
\be
\vol(B(p,r))\ge C (2\beta)^{m-1}r^m
\ee
\end{thm}

\begin{proof}
This is an immediate consequence of Theorem~\ref{tetra-ball}.
\end{proof}

\begin{rmrk}
In Example~\ref{ex-torus-tetra}, as $\epsilon \to 0$,
the $\vol(B(p,r)) \le \vol(M^3_\epsilon) \to 0$, so we could not have
a uniform tetrahedral property on these spaces.
\end{rmrk}

\begin{thm}\label{tetra-compactness}
Given $r_0>0, \beta\in (0,1), C>0, V_0>0$.  If
a sequence of compact Riemannian manifolds, $M^m$, 
has $\vol(M^m) \le V_0$,
$\diam(M^m) \le D_0$, 
and the $C, \beta$ (integral) tetrahedral
property for all balls of radius $\le r_0$, then a subsequence
converges in the Gromov-Hausdorff sense.   In particular
they have a uniform upper bound on diameter
depending only on these constants.
\end{thm}

The proof of this 
theorem strongly requires that the manifold have no boundary.

\begin{proof}
This follows immediately from Theorem~\ref{tetra-manifold}
and Gromov's Compactness Theorem, using the fact that
we can bound the number of disjoint balls of radius $\epsilon>0$
in $M^m$.   In a manifold, this provides an upper bound
on the diameter of $M^m$.   
\end{proof}

Later we will apply the following theorem to prove that
the Gromov-Hausdorff limit is in fact an Intrinsic Flat limit
and thus is countably $\mathcal{H}^m$ rectifiable
[Theorem~\ref{tetra-compactness-2}].   

\begin{thm}\label{tetra-set}
Given an integral current space $(X,d,T)$ and a point
$p_0\in \bar{X} \setminus Cl(\set(\partial T)$
then $p_0\in X=\set(T)$ if there exists a pair of constants
$\beta\in (0,1)$ and  $C>0$ such that $\bar{X}$ 
has the tetrahedral property at $p_0$ for all
sufficiently small $r>0$.
\end{thm}

\begin{proof}
By Theorem~\ref{tetra-ball} we have
\be\label{tetra-set-r}
||T||(B_p(r))\ge C (2\beta)^{m-1}r^m 
\ee
for almost every $r$ sufficiently small.
For any $R$ sufficiently small, there exists
$r=r_j<R$ satisfying (\ref{tetra-set-r})
with $r_j \to R$, so
\begin{eqnarray}
||T||(B(p,R)) &\ge& \limsup_{j\to \infty} ||T||(B(p,r_j)) \\
&\ge &\limsup_{j\to\infty}C (2\beta)^{m-1}r_j^m\\
&=&C (2\beta)^{m-1}R^m.
\end{eqnarray}
Thus $p_0\in X=\set(T)$ by the definition of $\set(T)$.
\end{proof}

\vspace{.4cm}
\subsection{Fillings, Slices and Intervals}

In the above sections, a key step consisted of estimating
$\mass(M) \ge \fillvol(\partial M)$.   This is only a worthwhile
estimate when $\partial M \neq 0$ or has a filling volume close
to the mass.   

A better estimate can be obtained using the
following trick.  Given a Riemannian manifold $M$, 
\be
\vol(M) =\vol(M\times I) \ge \fillvol(\partial (M\times I))
\ee
where the metric on $M\times I$ is defined in (\ref{isom-product}).
This has the advantage that $M\times I$ is always a
manifold with boundary.  It may also be worthwhile to
use an interval, $I_\epsilon$, of length $\epsilon$, then 
\be
\vol(M) =\frac{\vol(M\times I_\epsilon)}{\epsilon} 
\ge \frac{\fillvol(\partial (M\times I_\epsilon))}{\epsilon}.
\ee
Intuitively it would seem reasonable to 
conjecture that 
\be
\vol(M) =\lim_{\epsilon\to 0}\frac{\vol(M\times I_\epsilon)}{\epsilon} 
=\lim_{\epsilon\to 0} \frac{\fillvol(\partial (M\times I_\epsilon))}{\epsilon}.
\ee

We introduce the following notion made rigorous 
on arbitrary integral current spaces
$M=(X, d_X, T)$ by applying
Definition~\ref{times-I} and Proposition~\ref{prop-times-I}.

\begin{defn}\label{IFV}
Given any $\epsilon>0$, we define the $\epsilon$ interval filling
volume,
\be
\mathbb{IFV}_\epsilon(M) = \fillvol(\partial(M\times I_\epsilon)).
\ee 
\end{defn}

\begin{lem} \label{lem-interval}
Given an integral current space $M=(X,d,T)$,
\be
\mass(M) =\epsilon^{-1}\mass(M\times I_\epsilon)
\ge \epsilon^{-1} \mathbb{IVF}_\epsilon(M).
\ee
\end{lem}

\begin{proof}
This follows immediately from Proposition~\ref{prop-times-I}.
\end{proof}

\begin{defn} \label{SIF}
Given an integral current space, $M^m=(X,d,T)$ and  
$F_1, F_2,...F_k: M \to \R$ with $k\le m$ are 
Lipschitz functions with Lipschitz constant $\Lip(F_j)=\lambda_j$ 
then for all $\epsilon>0$ 
and almost every $r>0$
we can define the $\epsilon$ sliced interval
filling volume of $\partial S(p,r) \in \intcurr_{m-1}(\bar{X})$
to be
\be
{\mathbb{SIF}}_\epsilon(p,r, F_1,...,F_k)=
\prod_{j=1}^k  \epsilon^{-1}
\int_{t\in A_r} \fillvol(\partial(\Slice(S(p,r),F,t)\times I_\epsilon)) \, \mathcal{L}^k 
\ee
where 
\be
A_r=[\min F_1, \max F_1]\times [\min F_2, \max F_2]\times
\cdots \times [\min F_k, \max F_k
\ee
where $\min F_j=\min\{F(x): x\in \bar{B}_p(r)\}$
and $\max F_j=\max\{F(x): x\in \bar{B}_p(r)\}$.
When the $F_i$ are distance functions $\rho_{p_i}$ we write,
\be
{\mathbb{SIF}}_\epsilon(p,r, p_i,...,p_k):=
{\mathbb{SIF}}_\epsilon(p,r, \rho_{p_i},...,\rho_{p_k}).
\ee
\end{defn}

\begin{prop}\label{filling-I-slices}
Given an integral current space, $M^m=(X,d,T)$ and  
$F_1, F_2,...F_k: M \to \R$ with $k\le m$ are 
Lipschitz functions with Lipschtiz constant $\Lip(F_j)=\lambda_j$ 
then for all $\epsilon>0$ and almost every $r>0$
we can bound the mass
of a ball in $M$ as follows:
\be \label{filling-I-slices-1}
\mass(S(p,r)) \ge \prod_{j=1}^k \lambda_j^{-1}\epsilon^{-1} 
\mathbb{SIF}_\epsilon(p,r, F_1,...,F_k).
\ee
Thus for any $p_1,...p_k \in X$, and almost every $r>0$
we have
\be \label{filling-I-slices-2}
\mass(S(p,r)) \ge \epsilon^{-1} 
\mathbb{SIF}_\epsilon(p,r, p_1,...,p_k).
\ee
\end{prop}

\begin{proof}
By Proposition~\ref{prop-slicing} and
Lemma~\ref{lem-interval} 
we have
\begin{eqnarray}
\mass(S(p,r)) &\ge& \prod_{j=1}^k \lambda_j^{-1} \mass(S(p,r) \rstr dF)\\
&\ge& \prod_{j=1}^k \lambda_j^{-1} \int_{t\in\R^k} \mass(\Slice(S(p,r),F,t)) \, \mathcal{L}^k \\
&=& \prod_{j=1}^k \lambda_j^{-1} \int_{t\in A} \mass(\Slice(S(p,r),F,t)) \, \mathcal{L}^k\\
&\ge& \prod_{j=1}^k \lambda_j^{-1}\epsilon^{-1} \int_{t\in A} 
\fillvol( \partial (\Slice(S(p,r),F,t)\times I_\epsilon) ) \, \mathcal{L}^k. 
\end{eqnarray}
\end{proof}

\begin{cor} \label{int-filling-balls}
A point $p\in \bar{X}$
lies in $X=\set(T)$ if there exists $\epsilon>0$
and points $p_1,...,p_k$ such that
\be
\ess\liminf_{r\to 0} \frac{1}{\epsilon r^m}
\mathbb{SIF}_\epsilon(p,r, p_1,...,p_k)>0.
\ee
\end{cor}

\begin{cor}
A point $p\in \bar{X}$
lies in $X=\set(T)$ if there exists $C>0$, 
and points $p_1,...,p_k$ such that
\be
\ess\liminf_{r\to 0} \frac{1}{C r^{m+1}}
\mathbb{SIF}_{Cr}(p,r, p_1,...,p_k)>0.
\ee
\end{cor}

\begin{cor}
Given an integral current space $M$ we have
\be
\mass(M) \ge \prod_{j=1}^k \lambda_j^{-1} \epsilon^{-1}
\int_{t\in \R^k} \fillvol(\partial(\Slice(M,F,t)\times I_\epsilon)) \, \mathcal{L}^k 
\ee
\end{cor}

\newpage

\section{Convergence and Continuity}
In this section we examine the limits of points in
sequences of integral current spaces that converge in the
intrinsic flat sense and prove various continuity theorems and
close with a pair of Bolzano-Weierstrass Theorems.

Recall that Theorem~\ref{converge} (which was proven in work of the second author
with Wenger in \cite{SorWen2}) states that a sequence of integral current spaces
which converges in the intrinsic flat sense, $M_i \Fto M_\infty$,
can be embedded into a common complete metric space, $Z$,
via distance preserving maps,
$\varphi_i:M_i\to Z$, such that 
$\varphi_{i\#} T_i \Fto \varphi_{\infty\#}T_\infty$.  This allowed the second author
to define converging, Cauchy and disappearing sequences of points in
\cite{Sormani-AA}:

\begin{defn} \label{point-conv}
If $M_i=(X_i, d_i,T_i) \Fto M_\infty=(X_\infty, d_\infty,T_\infty)$, 
then we say $x_i\in X_i$ are a converging sequence that converge to
$x_\infty\in X_\infty$ if there exists a complete metric space,
$Z$, and isometric embeddings, 
$\varphi_i:M_i\to Z$, such that 
$\varphi_{i\#} T_i \Fto \varphi_{\infty\#}T_\infty$ and 
$\varphi_i(x_i) \to \varphi_\infty(x_\infty)$.   We say collection of
points, $\{p_{1,i}, p_{2,i},...p_{k,i}\}$,
converges to a corresponding collection of points, 
$\{p_{1,\infty}, p_{2,\infty},...p_{k,\infty}\}$, if 
$\varphi_{i}(p_{j,i}) \to \varphi_\infty(p_{j, \infty})$ for $j=1,2...k$.
\end{defn}

Unlike in Gromov-Hausdorff convergence, we have the possibility of disappearing sequences of points:

\begin{defn} \label{point-Cauchy}
If $M_i=(X_i, d_i,T_i) \Fto M_\infty=(X_\infty, d_\infty,T_\infty)$, then we say $x_i\in X_i$ 
are Cauchy if there exists a complete metric space
$Z$ and isometric embeddings 
$\varphi_i:M_i\to Z$ such that 
$\varphi_{i\#} T_i \Fto \varphi_{\infty\#}T_\infty$ and 
$\varphi_i(x_i) \to z_\infty \in Z$.   

We say the
sequence is disappearing if $z_\infty \notin \varphi_\infty(X_\infty)$.
\end{defn}

Examples with disappearing splines from \cite{SorWen2}
demonstrate that there exist Cauchy sequences of points
which disappear.  In \cite{Sormani-AA} the 
second author proved theorems demonstrating
when $z_\infty$ lies in the metric
completion of the limit space, $\varphi_\infty(\bar{X}_\infty)$.
This material is reviewed in the first two subsections of this paper:
including Theorems~\ref{to-a-limit} and ~\ref{flat-to-GH}
 and some related open questions.

Here we study when $z_\infty$ lies in the limit space itself:
\be
z_\infty \in X_\infty=\set(\varphi_{\infty\#}(T_\infty))
\ee
which happens iff 
\be
\liminf_{r\to 0}\mass\left(\varphi_{\infty\#}(T_\infty)\rstr B(z_\infty,r)\right)/r^m>0.
\ee

In \cite{SorWen1} the second author and Wenger intuitively applied the idea
that the filling volumes are continuous to prove sequences of points in certain 
sequences of spaces
do not disappear.   Here we will use sliced filling volumes and also provide the complete
details not provided in \cite{SorWen1} to justify the convergence
of filling volumes.   We first prove that slices of converging spaces converge in Proposition~\ref{limit-slices-more}.   We apply this proposition to prove 
that the sliced filling volumes are continuous [Theorem~\ref{SF-cont}].    
These results are technically difficult and require a sequence of propositions and lemmas.
Using semilar methods, we prove the continuity of the sliced interval filling volumes [Theorem~\ref{SIF-cont}] and the interval filling volumes [Theorem~\ref{IFV-cont}].

In the penultimate subsection, we apply these continuity theorems to
prove Theorem~\ref{SF_k-in-set} which 
 describes when a Cauchy sequence 
of points converges.  This theorem is a crucial step in the 
proof of the Tetrahedral Compactness Theorem.

We close the section with two Bolzano-Weiestrass Theorems.
Theorem~\ref{B-W-fillvol} concerns sequences of points $p_i\in M_i$
with lower bounds on the filling volumes of spheres around them
and produces a subsequence which converges to a point in the
intrinsic flat limit of the $M_i$.   Theorem~\ref{B-W} assumes that the 
points have
a lower bound on the sliced filling volumes of balls about them and
obtains a converging subsequence as well.

\vspace{.4cm}
\subsection{Review of Limit Points and Diameter Lower Semicontinuity}

Recall Definition~\ref{point-conv}.   Recall the following theorems
proven by the second author in \cite{Sormani-AA}:

\begin{thm}\label{to-a-limit}
If a sequence of integral current spaces, $M_i=\left(X_i,d_i,T_i\right)\in \mathcal{M}_0^m$, 
converges to 
a integral current space, $M=\left(X,d,T\right)\in \mathcal{M}_0^m$, in the intrinsic flat sense, then every point $z$ in the limit space
$M$ is the limit of points $x_i\in M_i$.  
In fact there exists a sequence of maps $F_i: X \to X_i$
such that $x_i=F_i(x)$ converges to x and
\be
\lim_{i\to \infty} d_i(F_i(x), F_i(y))= d(x,y).
\ee
\end{thm}

This sequence of maps $F_i$ are not uniquely defined and
are not even unique up to isometry.

\begin{defn}
Like any metric space, one can define the diameter
of an integral current space, $M=(X,d,T)$, to be
\be
\diam(M)=\sup\left\{ d_X(x,y): \,\, x, y\in X\} \in [0,\infty]\right\}.
\ee
However, we explicitly define the diameter of the
$0$ integral current space to be $0$.
A space is bounded if the diameter is finite.
\end{defn}

\begin{thm} \label{diam-semi}
Suppose $M_i \Fto M$ are integral current spaces then
\be
\diam(M) \le \liminf_{i\to \infty} \diam(M_i) \subset [0,\infty]
\ee
\end{thm}

\vspace{.4cm}
\subsection{Review of Flat convergence to Gromov-Hausdorff Convergence}\label{sect-flat-to-GH}

Recall the following theorem proven by the second author in \cite{Sormani-AA}:

\begin{thm} \label{flat-to-GH}
If a sequence of precompact integral current spaces, $M_i=\left(X_i,d_i,T_i\right)\in \mathcal{M}_0^m$, 
converges to 
a precompact integral current space, $M=\left(X,d,T\right)\in \mathcal{M}_0^m$, in the intrinsic flat sense, 
then there exists $S_i \in \intcurr_m\left(\bar{X}_i\right)$ such that
$N_i=\left(\set \left(S_i\right), d_i\right)$ converges to $\left(\bar{X},d\right)$ in the Gromov-Hausdorff
sense and
\be \label{flat-to-GH-1}
\liminf_{i\to\infty}\mass(S_i) \ge \mass(M).
\ee
When the $M_i$ are Riemannian manifolds, the $N_i$ can be taken to be
settled completions of open submanifolds of $M_i$.
\end{thm}

\begin{rmrk}\label{flat-to-GH-r}
If in addition it is assumed that $\lim_{i\to \infty}\mass(M_i)=\mass(M)$, then
by (\ref{flat-to-GH-1}) we have $\mass(\set(T_i-S_i), d_i, T_i-S_i)=0$.
In the Riemannian setting, we have $\vol(M_i\setminus N_i)\to 0$.
\end{rmrk}

\begin{rmrk}
In Ilmanen's example \cite{SorWen2} of a sphere with
increasingly many spikes, then $\set(S_i)$ are spheres with the spikes removed.
\end{rmrk}

\begin{rmrk}
The precompactness of the limit integral current spaces is
necessary in this theorem 
because a noncompact limit space can never be the Gromov-Hausdorff
limit of precompact spaces.
\end{rmrk}

\begin{rmrk} \label{rmrk-conv-point}
Gromov's Compactness Theorem combined with Theorem~\ref{flat-to-GH} implies that
that any sequence of $x_i\in N_i\subset M_i$
has a subsequence converging to a point $x$ in the metric completion of $M$.
Other sequences of points may not have converging subsequences, as can be seen when the tips of thin splikes disappear.  Below we will use filling volumes to determine
which sequences have converging subsequences using filling volumes [Theorem~\ref{B-W-fillvol}].   Another such Bolzano-Weierstrass Theorem with different hypothesis is proven in \cite{Sormani-AA}.
\end{rmrk}

\vspace{.2cm}

\begin{rmrk}
It is not immediately clear whether the integral current spaces, $N_i$,
constructed in the proof of Theorem~\ref{flat-to-GH} actually converge
in the intrinsic flat sense to $M$.  One expects an extra assumption
on total mass would be needed to interchange between flat and weak
convergence, but even so it is not completely clear.  One would need
to uniformly control the masses of $\partial N_i$ using a common upper
bound on $\mass(N)$ which can be done using theorems in 
Section 5 of \cite{AK}, but is highly technical.  It is worth investigating.
\end{rmrk}

\vspace{.4cm}
\subsection{Limits of Slices, Spheres and Balls}

In this section we prove the following two theorems via a sequence of lemmas
and propositions which will be applied elsewhere in this paper.   

Recall that about any point, $p$, for almost every radius, $r$,
the ball about $p$ of radius $r$ may be viewed as an integral current
space, $S(p,r)$, as in Lemma~\ref{lem-ball-current}.  In prior work of 
the second author \cite{Sormani-AA} it was shown that if 
$M_i \Fto M_\infty$ and $p_i \to p_\infty$ then for almost every $r\in \R$
there is a subsequence such that
\be
d_{\mathcal{F}}(S(p_i,r), S(p_\infty,r))\to 0 \textrm{ for almost every } r\in \R.
\ee
Here we prove the following more precise estimate:

\begin{thm}\label{thm-limit-balls}
Suppose we have a sequence of $m$ dimensional integral current
spaces, $M_i=(X_i,d_i, T_i)$ and $M_\infty=(X_\infty, d_\infty, T_\infty)$
and isometric embeddings
$\varphi_i: X_i \to Z$ such that
\be
d_F^Z({\phi_i}_\# T_i, {\phi_\infty}_\# T_\infty )< d_{\mathcal{F}}(M_i, M_\infty) +\epsilon_i
\ee
and points $p_i \in M_i$ such that 
\be
d_Z(\varphi_i(p_i), \varphi_\infty(p_{\infty})) \leq \delta_{i} .
\ee
Then,
for almost every $r\in \R$, the integral current spaces
$S(p_i,r)$, 
satisfy
\be \label{limit-balls-here}
d_{\mathcal{F}}(S(p_i,r),S(p_\infty, r)) \le \varepsilon_i(r)+d_{\mathcal{F}}(M_i, M_\infty) + \epsilon_i + ||T_\infty|| \left(\rho_{x_\infty}^{-1}(r-\delta_i, r+\delta_i) \right)
\ee
and
\be
\int_{-\infty}^\infty \varepsilon_i(r)\, dr\le d_{\mathcal{F}}(M_i, M_\infty)+\epsilon_i.
\ee
If $d_{\mathcal{F}}(M_i, M_\infty)\to 0$ and $p_i \to p_\infty$ 
then there is a subsequence (that we do not relabel), such that for almost every $r\in \R$,
\be
\lim_{i\to\infty} d_{\mathcal{F}}(S(p_i,r),S(p_\infty, r)) =0.
\ee
\end{thm}

In fact we will prove a more general statement in Proposition~\ref{limit-balls}.
Note that a subsequence is required to obtain the final limit as can be seen
in Example~\ref{many-intervals}.   

Recall that in Proposition~\ref{prop-slicing} we defined
the slices of an integral current space.  
In this section we will also prove the following
theorem concerning limits of slices:

\begin{thm} \label{limit-slices-more} 
Suppose we have a sequence of $m$ dimensional integral current
spaces, $M_i=(X_i,d_i, T_i)$ and $M_\infty=(X_\infty, d_\infty, T_\infty)$
and isometric embeddings
$\varphi_i: X_i \to Z$ such that
\be
d_F^Z({\phi_i}_\# T_i, {\phi_\infty}_\# T_\infty )< d_{\mathcal{F}}(M_i, M_\infty) +\epsilon_i
\ee
and points $p_{j,i} \in M_i$ such that 
\be
d_Z(\varphi_i(p_{j,i}), \varphi_\infty(p_{j,\infty})) \leq \delta_{i} \textrm{ for }j=1, \dots, k\le m  .
\ee
Then
\be
\begin{split}
\label{eq:L1Flat}
\int_{R^k} & d_{\mathcal{F}}\left(\Slice(M_i, \rho_{p_{i,1}}, ..., \rho_{p_{i,k}}, r_1,...,r_k), \Slice(M_\infty, \rho_{p_{\infty,1}},..., \rho_{p_{\infty,k}},r_1,...,r_k)\right)\,dr_1...dr_k\\
& \quad \qquad \qquad\leq \,\, d_{\mathcal{F}}(M_i, M_\infty)+\epsilon_i+ 2 \delta_{i} \left(\mass(T_\infty) + \mass(\partial T_\infty)\right).
\end{split}
\ee
If in addition $M_i \Fto M_\infty$ and $p_{i,j}\to p_{\infty,j}$ 
then there is a subsequence (which we do not relabel) 
such that for almost every $r\in \R^k$ we have
\be
\lim_{i\to\infty} d_{\mathcal{F}}\left(\Slice(M_i, \rho_{p_{i,1}}, ..., \rho_{p_{i,k}}, r_1,...,r_k), \Slice(M_\infty, \rho_{p_{\infty,1}},..., \rho_{p_{\infty,k}},r_1,...,r_k)\right) =0.
\ee
\end{thm}

Before proving either of the key propositions leading to these theorems, we
will prove a proposition [Proposition~\ref{slice-shift}] 
which captures the main idea leading to these results,
followed by a technical lemma [Lemma~\ref{lem-ann}].   Later we will
prove Proposition~\ref{le:IntFlatDistEst} by iterating the idea in this
proposition.

\begin{prop}\label{slice-shift}
Given an integral current space, $M=(X,d,T)$
and Lipschitz functions, $\rho: X \to \R$ 
and $f: X \to \R$, such that
\be\label{shift-slice-1}
|f(x) - \rho(x)| < \delta \qquad \forall x\in X,
\ee
then for almost every $r\in \R$
\be \label{shift-slice-2}
d_{\mathcal{F}}(\Slice(M, \rho,r), \Slice(M, f, r) ) \le 
||T||(\rho^{-1}(r-\delta, r+\delta) )
+ ||\partial T||(\rho^{-1}(r-\delta, r+\delta) ).
\ee 
\end{prop}

\begin{proof}
First observe that by the definition of intrinsic flat distance,
\be
d_{\mathcal{F}}(\Slice(M,p,r), \Slice(M, f, r) ) \le 
d_F^{\bar{X}}(<T, \rho, r> ,<T, f,r> ) \le
\mass(B) +\mass(A)
\ee
where
\begin{eqnarray}
B&=&T \rstr \rho^{-1}(-\infty,r]) \,\,\,
-   \,\,\, T \rstr f^{-1}(-\infty,r])  \\
A&=& (\partial T) \rstr f^{-1}(-\infty,r] \,\,\,-\,\,\,(\partial T) \rstr \rho^{-1}(-\infty,r].
\end{eqnarray}
Next note that for any pair of sets $U, V \subset X$,
\be
\mass(T\rstr U - T \rstr V) = \mass(T \rstr (\chi_U-\chi_V))
= \mass(T \rstr(U\setminus V)) +\mass(T\rstr(V\setminus U))
\ee
and the same holds for $\partial T$.
Since 
\be
\rho^{-1}(-\infty,r]\setminus f^{-1}(-\infty,r] \subset \rho^{-1}(r-\delta,r+\delta)
\ee
and
\be
 f^{-1}(-\infty,r]\setminus \rho^{-1}(-\infty,r] \subset \rho^{-1}(r-\delta,r+\delta)
\ee
we have
\begin{eqnarray}
\mass(B)&\le &\mass(T \rstr (\rho^{-1}(r-\delta,r+\delta) )\\
\mass(A)&\le &\mass(\partial T \rstr (\rho^{-1}(r-\delta,r+\delta) ).
\end{eqnarray}
\end{proof}

The following technical lemma is used in the proof of Proposition~\ref{limit-balls} and again in the proof of Proposition~\ref{le:IntFlatDistEst} below.

\begin{lem} \label{lem-ann} 
Let $\mu$ be a finite Borel measure on $\R$.
Then for every $\delta > 0$,
\be
\label{eq:ChangeOrderIntegrals}
\frac{1}{2\delta} \int_{-\infty}^{\infty} \mu(t - \delta, t+ \delta) dt = \mu(\R).
\ee
Moreover, the set of $a \in \R$ such that $\mu(\{ a \} ) > 0$ is at most countable.

In particular, given an integral current space, $(X,d,T)$, and
any Borel function, $f: X\to \R$, we have 
for all $r\in \R$ outside an at most countable set,
\be
\lim_{\delta\to 0}||T|| \left(f^{-1}(r-\delta, r+\delta)\right)=0
\ee
and
\be
\int_{r\in \R} ||T|| \left(f^{-1}(r-\delta, r+\delta)\right)\, dr
= 2\delta \mass(T).
\ee
\end{lem}

\begin{proof}
The equality (\ref{eq:ChangeOrderIntegrals}) follows by changing the order of integration (or rather Tonelli's Theorem, which is the analogue to Fubini's Theorem for nonnegative functions, cf. \cite[Chapter 12, Theorem 20]{Royden-analysis}) as follows
\be
\begin{split}
\int_{-\infty}^\infty \mu(t - \delta, t + \delta) dt 
&= \int_{-\infty}^\infty \left( \int_\R \chi_{(t - \delta, t+ \delta)} d\mu \right) dt \\
&= \int_{R^2} \chi_{\{(x,y) \in \R^2 \, | \, -\delta < y - x < \delta\}} d\lm^1 d\mu\\
&= \int_{\R} 2 \delta d\mu = 2\delta \mu(\R).
\end{split}
\ee
That the set $A_\mu := \{a \in \R \,|\, \mu(\{a\}) > 0 \}$  for any finite Borel measure $\mu$ is at most countable is a well-known fact. It follows as the the number of $b \in \R$ such that $\mu(\{b\})>1/n$ is bounded by $n \mu(\R)$, and therefore the set $A_\mu$ is at most countable, as it is the countable union of at most finite sets.

To conclude the second half of the lemma, we just apply the first part to $\mu = f_\# \|T\|$. 
Indeed, by the $\sigma$-additivity of the measure,
\be
\lim_{\delta \to 0} \|T\|(f^{-1}(r-\delta,r+\delta)) = \lim_{\delta \to 0} f_\# \|T\| (r - \delta, r+ \delta) =  f_\# \|T\| (\{ r \}),
\ee
which is strictly positive for at most countably many $r \in \R$. 
\end{proof}

Observe that Theorem~\ref{thm-limit-balls} is an immediate consequence of
the next proposition by taking $z_i=\varphi_i(p_i)$: 

\begin{prop}\label{limit-balls}
Suppose we have a sequence of integral current
spaces, $M_i=(X_i,d_i, T_i)$ and isometric embeddings
$\varphi_i: X_i \to Z$, such that
\be
\lim_{i\to \infty} d^Z_{F}(\varphi_{i\#} T_i, \varphi_{\infty\#}T_\infty)=0
\ee
and points $z_i \in Z$ such that $\delta_i=d_Z(z_i, z_\infty)$. 
Then,
for almost every $r\in \R$, the balls,
$S_i(r)=\varphi_{i\#} T_i \rstr B(z_i,r)$, 
satisfy
\be \label{limit-balls-here}
d_{F}^Z(S_i(r),S_\infty(r)) \le \varepsilon_i(r)+d^Z_F(\varphi_{i\#}T_i, \varphi_{\infty\#} T_\infty)
+ ||\varphi_{\infty\#}T_\infty|| \left(f^{-1}(r-\delta_i, r+\delta_i) \right)
\ee
where $f(x)=\rho_{z_\infty}(\varphi_\infty(x))$ and
\be
\int_{-\infty}^\infty \varepsilon_i(r)\, dr\le d^Z_F(\varphi_{i\#}T_i, \varphi_{\infty\#} T_\infty).
\ee
If $\delta_i \to 0$, then there is a subsequence (that we do not relabel), such that for almost every $r\in \R$,
\be
\lim_{i\to\infty} d_{F}^Z(S_i(r),S_\infty(r)) =0
\ee
\end{prop}

\begin{proof}
There exists integral currents $A_i, B_i$ in $Z$, such that
\be
\varphi_{i\#}T_i-\varphi_{\infty\#}T_\infty= A_i + \partial B_i
\ee
and 
\be
d^Z_F(\varphi_{i\#}T_i,\varphi_{\infty\#}T_\infty)= \mass(A_i) +\mass(B_i).
\ee
For almost every $r$, the restrictions of these
spaces to balls, $B(z_i,r)$, are integral current spaces such that
\be
S_i(r)-S_\infty(r)=(\varphi_{i\#}T_i- \varphi_{\infty\#}T_\infty)\rstr \bar{B}(z_i,r) + S'_i(r) 
\ee
where
\be
S'_i(r)=  (\varphi_{\infty\#}T_\infty)\rstr \bar{B}(z_i,r)-
(\varphi_{\infty\#}T_\infty)\rstr \bar{B}(z_\infty,r).
\ee
Thus
\begin{eqnarray}
S_i(r)-S_\infty(r)&=&A_i\rstr \bar{B}(z_i,r) + (\partial B_i) \rstr \bar{B}(z_i,r)+S'_i(r)\\
&=& A_i \rstr \bar{B}(z_i,r) + <B_i, \rho_i,r> + \partial (B_i \rstr \bar{B}(z_i,r))+S'_i(r)
\end{eqnarray}
Since these are integral currents for almost every $r$ we have
\be
d_{F}^Z(S_i(r),S_\infty(r))\le 
\mass(A_i \rstr \bar{B}(z_i,r) + <B_i, \rho_i,r>) + \mass(B_i \rstr \bar{B}(z_i,r)) +\mass(S'_i(r)).
\ee
By the Ambrosio-Kirchheim Splitting Theorem and $\Lip(\rho_i)\le 1$
we have a Lebesgue measurable function $\epsilon_i:\R\to [0,\infty)$
such that
\be
\mass(A_i \rstr \bar{B}(z_i,r) + <B_i, \rho_i,r>)
\le \mass(A_i) + \varepsilon_i(r),
\ee
where 
\be
\int_{-\infty}^\infty \varepsilon_i(r) \, dr \le \mass(B_i).
\ee
Naturally
\be
\mass(B_i \rstr \bar{B}(z_i,r))\le \mass(B_i)
\ee
and
\begin{eqnarray}
\mass(S'_i(r))&=&\mass\left((\varphi_{\infty\#}T_\infty)\rstr \bar{B}(z_i,r)-
(\varphi_{\infty\#}T_\infty)\rstr \bar{B}(z_\infty,r)\right)\\
&\le& ||\varphi_{\infty\#}T_\infty|| Ann_{z_\infty}(r-\delta_i, r+\delta_i).
\end{eqnarray}
Thus we have (\ref{limit-balls-here}).
The rest follows from Lemma~\ref{lem-ann} and the fact that for a subsequence and almost every $r$ we have
$\lim_{i\to\infty}\varepsilon_i(r) =0$.
\end{proof}

In the next proposition, we will iterate the proof of Proposition~\ref{slice-shift}
to bound the flat distance between lower-dimensional slices of two different currents with two different Lipschitz functions.  

\begin{prop}
\label{le:IntFlatDistEst}
Let $T_1$ and $T_2$ be two $m$ dimensional integral currents on a complete metric space $Z$. 
Let $k \in \{ 1, \dots, m \}$ and let $\pi: Z \to \R^k$ and $\tilde{\pi}:Z \to \R^k$ be two Lipschitz functions, such that 
\be
|\pi_j(z) - \tilde{\pi}_j(z)|<\delta \qquad \forall z\in Z, \,\, \forall j  \in \{1, \dots, k\}.
\ee
and such that
\be
\Lip(\pi) \le L \textrm{ and } \Lip(\tilde{\pi})\le L.
\ee
Then
\be
\int_{\R^k} d_F^Z(<T_1, \pi, t>,<T_2, \tilde{\pi}, t>) dt 
	\leq L^k d_F^Z(T_1, T_2) + 2 k \delta L^{k-1} ( \mass(T_2) + \mass(\partial T_2) ).
\ee
\end{prop}

Note that this proposition implies Theorem~\ref{limit-slices-more}
by taking $T_1=\varphi_{i\#}T_i$ and $T_2=\varphi_{\infty\#}T_\infty$,
$\pi=(\rho_{z_{1,i}},...,\rho_{z_{k,i}})$ where $z_{j,i}=\varphi_i(p_{j,i})$ and 
$\tilde{\pi}=(\rho_{z_{1,\infty}},...,\rho_{z_{k,\infty}})$ where 
$z_{j,\infty}=\varphi_\infty(p_{j,\infty})$.
Then $\delta=\delta_i$ and $L=1$.  It may also be applied to study 
Cauchy sequences of points in a similar way.  This proposition
is applied later in the paper to study sliced filling volumes.   

\begin{proof}
First, we write
\be
<T_1, \pi, t>\, -\, <T_2, \tilde{\pi}, t> \,\,= \,\,<T_1 - T_2, \pi, t>
 +\, <T_2,\pi, t>\, -\,<T_2, \tilde{\pi},t>.
\ee
Let $\epsilon > 0$ be arbitrary. 
There exist integral currents $A \in \intcurr_m(Z)$ and $B \in \intcurr_{m+1}(Z)$ such that
\be
T_1 - T_2 = A + \partial B
\ee
and
\be
\mass(A) + \mass(B) \leq d_F^Z(T_1, T_2) + \epsilon.   
\ee
Then
\be
\begin{split}
<T_1 - T_2, \pi, t> 
	&= <A, \pi, t> + <\partial B, \pi, t>\\
	&= <A, \pi, t> + \, (-1)^k \partial < B , \pi, t >.
\end{split}
\ee
Note that by the Ambrosio-Kirchheim Slicing Theorem
\begin{align}
\int_{\R^k} \mass(< A, \pi, t>) dt &\leq L^k \mass(A),\\
\int_{\R^k} \mass( <B, \pi, t>) dt &\leq L^k \mass(B).
\end{align}

We define the projections $P_j:\R^k \to \R^j$ and $Q_j:\R^k \to \R^{k-j}$ by
\begin{align}
P_j(x_1, \dots, x_j, x_{j+1}, \dots, x_k) &:= (x_1, \dots, x_j)\\
Q_j(x_1, \dots, x_j, x_{j+1}, \dots, x_k) &:= (x_{j+1}, \dots, x_k)
\end{align}
so that $P_k$ and $Q_0$ are identity maps.   Using the following slight
abuse of notation:
\begin{eqnarray}
T
	&=&<T,P_0 \circ \tilde{\pi}, P_0(t)>\\
T 
	&=&\,\,<T, Q_{k} \circ \pi, Q_{k}(t)>
\end{eqnarray}
we have for $\lm^k$-a.e. $t \in \R^k$,
\be
\begin{split}
<T_2, \pi, t> - <T_2, \tilde{\pi}, t> 
	&=\,\,<<T_2,P_0 \circ \tilde{\pi}, P_0(t)>, Q_0 \circ \pi, Q_0(t)>  \\
	&\qquad -  <<T_2,P_{k} \circ \tilde{\pi}, P_{k}(t)>, Q_{k} \circ \pi, Q_{k}(t)>\\
	&= \sum_{j=0}^{k-1} \Big[<<T_2,P_j \circ \tilde{\pi}, P_j(t)>, Q_j \circ \pi, Q_j(t)>  \\
	&\qquad -  <<T_2,P_{j+1} \circ \tilde{\pi}, P_{j+1}(t)>, Q_{j+1} \circ \pi, Q_{j+1}(t)>\Big]. 
\end{split}
\ee
We calculate each term in the sum using the iterated definition of a slice:
\be
\begin{split}
<T_2,& (\tilde{\pi}_1, \dots, \tilde{\pi}_j, \pi_{j+1},\dots \pi_k),(t_1, \dots, t_k)> \\
& \quad \quad- <T_2,(\tilde{\pi}_1,\dots,\tilde{\pi}_{j+1}, \pi_{j+2}, \dots, \pi_k),(t_1, \dots, t_k) >\\
& =\,\,\, <<<T_2, P_j \circ \tilde{\pi},P_j(t)>,\pi_{j+1},t_{j+1}>,
		Q_{j+1}\circ \pi, Q_{j+1}(t)> \\
& \quad \quad- <<<T_2, P_j \circ \tilde{\pi},P_j(t)>,\tilde{\pi}_{j+1},t_{j+1}>,
		Q_{j+1}\circ \pi, Q_{j+1}(t)> \\
&= \,\,\,<\partial <T_2,P_j\circ \tilde{\pi},P_j(t)>\llcorner \pi_{j+1}^{-1}(t_{j+1},\infty),Q_{j+1}\circ \pi,Q_{j+1}(t)>\\
&\quad \quad - < \partial\left(<T_2,P_j \circ \tilde{\pi},P_j(t)>\llcorner \pi_{j+1}^{-1}(t_{j+1},\infty)\right)
,Q_{j+1}\circ \pi,Q_{j+1}(t)>\\
&\quad \quad- <\partial <T_2,P_j\circ \tilde{\pi},P_j(t)>\llcorner \tilde{\pi}_{j+1}^{-1}(t_{j+1},\infty),Q_{j+1}\circ \pi,Q_{j+1}(t)>\\
&\quad\quad +  < \partial\left(<T_2,P_j\circ \tilde{\pi},P_j(t)>\llcorner \tilde{\pi}_{j+1}^{-1}(t_{j+1},\infty)\right)
,Q_{j+1} \circ \pi,Q_{j+1}(t)>\\
&=\,\,\, <\partial <T_2,P_j\circ \tilde{\pi},P_j(t)>\llcorner \pi_{j+1}^{-1}(t_{j+1},\infty),Q_{j+1}\circ \pi,Q_{j+1}(t)>\\
&\quad\quad -  (-1)^{k-j}\partial < <T_2,P_j \circ \tilde{\pi},P_j(t)>\llcorner \pi_{j+1}^{-1}(t_{j+1},\infty)
,Q_{j+1}\circ \pi,Q_{j+1}(t)>\\
&\quad\quad - <\partial <T_2,P_j\circ \tilde{\pi},P_j(t)>\llcorner \tilde{\pi}_{j+1}^{-1}(t_{j+1},\infty),Q_{j+1}\circ \pi,Q_{j+1}(t)>\\
&\quad \quad+  (-1)^{k-j}\partial < <T_2,P_j\circ \tilde{\pi},P_j(t)>\llcorner \tilde{\pi}_{j+1}^{-1}(t_{j+1},\infty)
,Q_{j+1} \circ \pi,Q_{j+1}(t)>\\
&=\,\,\,A_j(t)\,\, +\,\, \partial B_j(t),
\end{split}
\ee
where
\begin{align}
A_j(t) &:= <\partial <T_2, P_j \circ \tilde{\pi},P_j(t)>\llcorner 
\chi_{j+1,t_{j+1}},Q_{j+1} \circ \pi,Q_{j+1}(t)> \\
B_j(t) &:= (-1)^{k-j} < <T_2, P_j\circ \tilde{\pi},P_j(t)>\llcorner 
\chi_{j+1,t_{j+1}}, Q_{j+1}\circ \pi,Q_{j+1}(t) >.
\end{align}
Here, $\chi_{j+1,t_{j+1}}$ is defined as the following difference of characteristic functions
\be
\chi_{j+1,t_{j+1}} := \chi_{\pi_{j+1}^{-1}(t_{j+1},\infty)} -
\chi_{\tilde{\pi}_{j+1}^{-1}(t_{j+1},\infty)}.
\ee
It follows that
\be
<T_2, \pi, t > - < T_2, \tilde{\pi}, t> = \sum_{j=0}^{k-1} (A_j(t) + \partial B_j(t)).
\ee

Since $\chi_{j+1,t_{j+1}}$ is supported on $\pi_{j+1}^{-1}[t_{j+1} - \delta, t_{j+1} + \delta]$, Lemma \ref{lem-ann} below implies that
\be
\begin{split}
\int_{\R^m} & \mass(A_j(t)) \, dt\\
& \leq  L^{k-j-1} \int_{\R^m} \mass(\partial <T_1, P_j \circ \tilde{\pi},P_j(t)>\llcorner 
\chi_{j+1,t_{j+1}}) \, dt_1 \dots dt_{j+1} \\
& \leq 2 \delta L^{k-j-1}\int_{\R^m} \mass(\partial <T_1, P_j \circ \tilde{\pi},P_j(t)>) \, dt_1 \dots dt_j \\
& \leq 2 \delta L^{k-1} \mass( \partial T_1 ).
\end{split}
\ee
In the same way,
\be
\begin{split}
\int_{\R^m} & \mass(< <T_1, P_j\circ \tilde{\pi},P_j(t)>\llcorner 
\chi_{j+1,t_{j+1}}, Q_{j+1}\circ \pi,Q_{j+1}(t) >) \, dt_1 \dots dt_k \\
& \leq L^{k-j-1} \int_{\R^m} \mass(<T_1, P_j \circ \tilde{\pi},P_j(t)>\llcorner 
\chi_{j+1,t_{j+1}}) \, dt_1 \dots dt_{j+1} \\
& \leq 2 \delta L^{k-j-1} \int_{\R^m} \mass(<T_1, P_j \circ \tilde{\pi},P_j(t)>) \, dt_1 \dots dt_j \\
& \leq 2 \delta L^{k-1} \mass( T_1 ).
\end{split}
\ee
We conclude by applying the triangle inequality and by taking the limit $\epsilon \downarrow 0$.
\end{proof}

\begin{rmrk}
One might be able to strengthen the results in this section if one assumes the
sequence of integral current spaces is a sequence of Riemannian manifolds with
boundary.   Or one may wish to try to produce an example of a sequence of
manifolds which have the same sorts of disturbing properties as 
Examples~\ref{concentric-spheres} and~\ref{many-intervals} at least in a limiting sense.
\end{rmrk}

\vspace{.4cm}
\subsection{Continuity of Sliced Filling Volumes}

In this section we prove continuity and semicontinuity of the 
various Sliced Filling Volumes [Theorem~\ref{SF-cont}].
Recall that Theorem~\ref{fillvol-cont} implies the continuity of 
filling volume in the following sense:
\be
M_i \Fto M_\infty\,\,\, \implies \,\,\,\fillvol(\partial M_i) \to \fillvol(\partial M_\infty)
\ee
where the filling volume is defined as in Definition~\ref{defn-filling-volume}.
In this section we combine Theorem~\ref{fillvol-cont}
in combination with the convergence of slices proven in Proposition~\ref{limit-slices-more}.
An immediate consequence of these results is that
the filling volumes of slices converge.   In particular the filling 
volumes of spheres converge to the filling volumes spheres,
as stated in prior work of the first author with Wenger \cite{SorWen2}.   

The situation is more complicated when one considers sliced filling volumes:

\begin{ex}
Recall the integral current spaces $M$ and $M_\ell$ defined in Example~\ref{many-intervals}. 
One may observe that there exists a sequence $p_\ell \in M_\ell$ converging to $p$ such that for $\lm$-a.e. $r \in (0,1/4)$,
\be
\liminf_{\ell \to \infty} \SF_0(p_\ell,r) = 0 < 2r = \SF_0(p,r) < 4r = \limsup_{\ell \to \infty} \SF_0(p_\ell, r).
\ee
In fact these inequalities hold for all sequences $p_\ell$ that converge to $p$ at a high enough rate.
\end{ex}

Nevertheless we are able to prove the following continuity theorem:

\begin{thm} \label{SF-cont}
Let $M_i = (X_i,d_i,T_i)$ be a sequence of $m$ dimensional integral current spaces such that $M_i \Fto M_\infty = (X_\infty, d_\infty, T_\infty)$ and let the collections of points $p_i , p_{i,1}, \dots, p_{i,k}$ converge to $p_\infty, p_{\infty,1}, \dots, p_{\infty,k}$ as $i \to \infty$ for $k \in \{ 1, \dots, m-1\}$. 
Then there exists a subsequence, which we do not relabel, such that for $\lm^1$-a.e. $r > 0$,
\be
\label{SF-contin}
\lim_{i \to \infty} \SF(p_{i}, r, p_{i,1}, \dots, p_{i,k}) = \SF(p_\infty, r, p_{\infty,1}, \dots, p_{\infty,k}).
\ee
In the $k=0$ case we have for every $r > 0$,
\be
\label{eq:FillVollL1Cont}
\lim_{i \to \infty} \frac{1}{r} \int_0^r |\SF_0(p_i, \tau) - \SF_0(p_\infty, \tau)| d\tau = 0.
\ee
Consequently, there is a subsequence $i_j$, such that for $\lm^1$-a.e. $r > 0$,
\be
\lim_{j \to \infty} \SF_0(p_{i_j}, r) = \SF_0(p_\infty,r)\le  \mass(S(p_\infty,r)).
\ee
For $k \in \{ 0, \dots, m-1 \}$,
we also obtain the following inequality for $\lm^1$-a.e. $r>0$,
\be
\liminf_{i \to \infty} \SF_k(p_i,r)  \leq \mass(S(p_\infty,r)).
\ee
Finally, for every $r>0$, we have
\be
\limsup_{i \to \infty} \frac{1}{r} \int_0^r \SF_k(p_i, \tau) d \tau
\leq \frac{1}{r} \int_0^r \mass( S(p_\infty, \tau) ) d\tau.
\ee
\end{thm}

Before we prove Theorem~\ref{SF-cont}, we state and prove two key ingredients
towards the proof [Proposition~\ref{SF-cont-a} and Lemma~\ref{le:EstBySlices}].  

\begin{prop} \label{SF-cont-a} 
Suppose we have a sequence of $m$ dimensional integral current
spaces, $M_i=(X_i,d_i, T_i)$ and isometric embeddings $\phi_i:X_i \to Z$ 
and points $z_{j,i} \in Z$, and $\delta_i > 0$ such that $d_Z(z_{j,i}, z_{j,\infty}) < \delta_i$
for $j=1..k$ for some $k\in \{0,..,m-1\}$  and $p_i \in X_i$, such that
$d_Z(\varphi_i(p_i), \varphi(p))< \delta_i$, then
for almost every $t\in \R^k$
\be
\label{eq:SlicesByFlat}
\begin{split}
&|\fillvol(\partial \Slice(M_i,\rho_i,t)) - \fillvol(\partial \Slice(M_\infty,\rho_\infty,t))| \\
&\qquad  \leq d_{\mathcal{F}}(\partial \Slice(M_i,\rho_i,t),\partial \Slice(M_\infty, \rho_\infty,t))\\
\end{split}
\ee
and thus
\be
\label{eq:L1normFillVol}
\begin{split}
\int_{\R^k} & |\fillvol(\partial \Slice(M_i,\rho_i,t)) - \fillvol(\partial \Slice(M_\infty,\rho,t))|dt \\
& \leq \int_{R^k} d_{\mathcal{F}}(\partial \Slice(M_i,f,t),\partial \Slice(M_\infty, \rho,t)) dt\\
& \leq d_F^Z({\phi_i}_\# M_i, {\phi_\infty}_\# M_\infty ) + 2 \delta  (\mass(T_\infty) + \mass(\partial T_\infty)).
\end{split}
\ee

If $\lim_{i\to\infty}\delta_{i}=0$
and
\be
\lim_{i \to \infty} d^Z_{F}(\varphi_{i\#} T_i, \varphi_{\infty\#}T_\infty)=0
\ee
then for almost every $t \in \R^k$
the masses satisfy
\be
\liminf_{i\to\infty} \mass(\Slice(M_i,\rho_i,t))
\ge\mass(\Slice(M_\infty,\rho_{\infty}, t)).
\ee
Finally, there is a subsequence such that (without relabeling) for almost every $t \in \R^k$,
\be
\label{eq:FlatConvSlice}
\lim_{i \to \infty} d_{\mathcal{F}}(\partial \Slice(M_i,\rho_i,t),\partial \Slice(M_\infty, \rho_\infty,t)) = 0.
\ee
\end{prop}

\begin{proof}
That one can estimate the difference in Filling Volume of the boundaries of two currents in terms of the flat distance between them as in (\ref{eq:SlicesByFlat}) was explained in Theorem \ref{fillvol-cont}. 
Inequality (\ref{eq:L1normFillVol}) is then a direct consequence of inequality (\ref{eq:L1Flat}) in Lemma \ref{limit-slices-more}.
We select a subsequence of $M_{i_j}$ of $M_i$ such that
\be
\lim_{j \to \infty} \mass(\Slice(M_{i_j}, \rho_{i_j}, t)) = \liminf_{i \to \infty} \mass(\Slice(M_i, \rho_i,t)).
\ee
The integral bound (\ref{eq:L1Flat}) implies that for a subsequence of the $M_{i_j}$ (that we do not relabel), for almost every $t \in \R^k$ equation (\ref{eq:FlatConvSlice}) holds. 
Since flat convergence implies weak convergence and the mass is lower-semicontinuous under weak convergence, 
\be
\liminf_{i\to\infty} \mass( \Slice(M_i,\rho_i,t) ) = \liminf_{j \to \infty} \mass(\Slice(M_{i_j}, \rho_{i_j}, t)) \geq 
\mass(\Slice(M_\infty, \rho_\infty, t)).
\ee
\end{proof}

\begin{lem}
\label{le:EstBySlices}
Let $M= (X, d, T)$ be an $m$ dimensional integral current space, and $\pi:X \to \R^k$ be a Lipschitz function with $\Lip(\pi_j) \leq 1$. 
Then
\be
\label{eq:EstFillVolBySlices} \int_{\R^k} \fillvol(\partial \Slice( M, \pi, t )) dt \leq \fillvol(\partial M).
\ee
In particular,
\be \label{SFkleSF0}
\SF_k(p,r) \leq \SF_0(p,r) \le \mass(S(p,r))
\ee
\end{lem}

\begin{proof}
Let $\epsilon > 0$. 
There is a $m$ dimensional integral current space $A$ such that $\partial A = \partial M$ and
\be
\fillvol(\partial M) + \epsilon \geq \mass(A).
\ee
By the Ambrosio-Kirchheim slicing theorem,
\be
\begin{split}
\fillvol(\partial M) + \epsilon &\geq  \mass(A) \\
&\geq  \int_{\R^k} \mass( \Slice(A, \pi, t) ) dt \\
&\ge  \int_{\R^k} \fillvol(\partial \Slice(A, \pi, t)) dt \\
&= \int_{\R^k} \fillvol( \Slice(\partial A, \pi, t))  dt \\
&= \int_{\R^k} \fillvol( \Slice(\partial M, \pi, t))  dt \\
&= \int_{\R^k} \fillvol(\partial \Slice(M, \pi, t)) dt.
\end{split}
\ee
Since $\epsilon > 0$ is arbitrary, the estimate (\ref{eq:EstFillVolBySlices}) holds.

Taking $M=S(p,r)$, we have by Definition~\ref{sliced-filling-vol},
\begin{eqnarray} 
\SF(p,r,F_1,...,F_k)
&=&\int_{t\in A_r} \fillvol(\partial\Slice(S(p,r),F,t)) \, \mathcal{L}^k\\
&\le &\int_{t\in A_r} \fillvol(\partial S(p,r)) \, \mathcal{L}^k\\
&=&\SF_0(p,r)
\end{eqnarray} 
Taking $F_i=\rho_{q_i}$ we have
\be
\SF(p,r,q_1,...,q_k) \le \SF_0(p,r)
\ee
and taking the supremum over $q_i \in \partial B_p(r)$ we obtain (\ref{SFkleSF0}).
\end{proof}

We may now prove Theorem~\ref{SF-cont}:

\begin{proof}
By the definition of convergence of points, there exists a complete separable metric space $Z$ and isometric embeddings $\phi_i: \overline{X_i} \to Z$, such that
\be
d_F^Z ({\phi_i}_\# T_i , {\phi_\infty}_\# T_\infty ) = 0,
\ee
and $\delta_i := d_Z(\phi_i(p_i) , \phi_\infty(p_\infty)) \to 0$, $\delta_{i,j} := d_Z(\phi_i(p_{i,j},\phi_\infty(p_{\infty,j})) \to 0$, as $i \to \infty$, $j = 1, \dots, k$.

By Proposition~\ref{limit-balls} there exists a subsequence, such that for $\lm^1$-a.e. $r>0$,
\be
d_F^Z ({\phi_i}_\# T_i \llcorner B_{r}(\phi_i(p_i)), {\phi_\infty}_\# T_\infty \llcorner B_{r}(\phi_{\infty}(p_\infty)) )  \to 0
\ee
as $i \to \infty$. 

Hence, for such a value of $r>0$ we can apply Proposition (\ref{SF-cont-a}) to the integral current spaces associated to ${\phi_i}_\# T_i \llcorner B_r(p_i)$. In particular, inequality (\ref{eq:L1normFillVol}) yields the continuity property expressed by (\ref{SF-contin}).

Similarly, if $M_i \Fto M$ and the points $p_i \in M_i$ converge to $p_\infty$, there is a complete separable metric space $Z$ and isometric embeddings $\phi_i:\overline{X_i} \to Z$ such that
\be
\lim_{i \to \infty} d_F^Z({\phi_i }_\# T_i , {\phi_\infty}_\# T_\infty) = 0,
\ee
and $\delta_i := d_Z(\phi_i(p_i) , \phi_\infty(p_\infty)) \to 0$ in $Z$ as $i \to \infty$.

By using respectively Theorem~\ref{fillvol-cont}, Proposition~\ref{limit-balls}, and Lemma~\ref{lem-ann}, we find
\be
\begin{split}
\frac{1}{r} \int_0^r | \SF_0( p_i, \tau ) - \SF_0(p_\infty,\tau)| d\tau \,\,
& \leq\,\, \frac{1}{r} \int_0^r d_\mathcal{F} ( \partial S(p_i, r), \partial S(p_\infty, r) ) d\tau\\
& \leq\,\, \left( \frac{1}{r} + 1 \right) d_F^Z( {\phi_i }_\# T_i, {\phi_\infty}_\# T_\infty) \\ 
& \,\,\qquad + \frac{1}{r} \int_0^r \| T_\infty\| ( \rho_{p_\infty}^{-1} (\tau - \delta_i, \tau + \delta_i) ) d\tau\\
& \leq \,\,\left( \frac{1}{r} + 1 \right) \,d_F^Z( {\phi_i }_\# T_i, {\phi_\infty}_\# T_\infty) \\ 
& \,\,\qquad + \frac{2}{r} \,\delta_i \,\mass (T_\infty).
\end{split}
\ee
When we take the limit $i\to\infty$, we obtain (\ref{eq:FillVollL1Cont}).

By Lemma \ref{le:EstBySlices}, for $\lm^1$-a.e. $r>0$, we have
\be
\liminf_{i \to \infty} \SF_k(p_i,r) \leq \lim_{j\to\infty} \SF_0(p_{i_j},r) = \SF_0(p_\infty,r) \leq \mass(S(p_\infty,r)).
\ee
Thus for every $r>0$, we have
\be
\begin{split}
\limsup_{i \to \infty} \frac{1}{r} \int_0^r \SF_k(p_i, \tau) d \tau
&\leq \limsup_{i \to \infty} \frac{1}{r} \int_0^r \SF_0(p_i, \tau) d\tau \\
&= \frac{1}{r} \int_0^r \SF_0(p_\infty, \tau) d\tau \\
&\leq \frac{1}{r} \int_0^r \mass( S(p_\infty, \tau) ) d\tau.
\end{split}
\ee
\end{proof}

\vspace{.4cm}
\subsection{Continuity of Interval Filling Volumes}

Recall the definition of the interval filling volume of
a manifold or integral current space in Definition~\ref{IFV},
\be
\IFV(M)=\fillvol(\partial (M\times I_\epsilon))
\le \epsilon \mass(M).
\ee
This notion was particularly useful for $M$ without boundary.
In this section we prove the interval filling volume is continuous
with respect intrinsic flat convergence
[Theorem~\ref{IFV-cont}].  Taking more precise estimates we prove the
sliced interval filling volumes are continuous as well
[Theorem~\ref{SIF-cont}].

\begin{thm} \label{IFV-cont}  
Suppose we have $m$ dimensional integral current
spaces, $M_i=(X_i,d_i, T_i)$, such that $M_i \Fto M_\infty$,
then for any fixed $\epsilon>0$, their interval filling volumes converge,
\be
\lim_{i\to \infty} \IFV(M_i) = \IFV(M_\infty)
\ee.
\end{thm}

\begin{proof}
By Proposition~\ref{interval-flat}, we see that 
$M_i \times I_\epsilon \Fto M_\infty \times I_\epsilon$.
Thus we have continuity applying Theorem~\ref{fillvol-cont}.
\end{proof}

We now prove the continuity of the 
sliced interval filling volume defined in Definition~\ref{SIF}.

\begin{thm} \label{SIF-cont} 
Suppose $M_i \Fto M$.
If $p_i\in M_i$ converge to
$p_\infty\in M_\infty$, and $q_{j,i}\in M_\infty$ converge to $q_{j,\infty}\in M_\infty$  for $j=1, \dots, k$ then for any fixed $\epsilon>0$, there is a subsequence such that for almost every $r \in \R$,
\be \label{eq:SIF-contin}
\lim_{i\to\infty} \SIF(p_i,r, q_{1,i}, ... q_{k,i}) 
= \SIF(p_\infty, r, q_{1,\infty},...q_{k,\infty}).
\ee
\end{thm}

\begin{proof}
This theorem is a consequence of Proposition~\ref{SIF-cont-a} stated and
proven immediately below, combined with Proposition~\ref{limit-balls}
and Definition~\ref{SIF}.
\end{proof}  

\begin{prop} \label{SIF-cont-a}
Suppose we have a sequence of $m$ dimensional integral current
spaces, $M_i=(X_i,d_i, T_i)$ and isometric embeddings
$\varphi_i: X_i \to Z$, constants $\delta_i> 0$ 
and points $z_{j,i} \in Z$ such that $d_Z(z_{j,i}, z_{j,\infty}) < \delta_i$
for $j=1..k$ for some $k\in \{0,..,m-1\}$  and $p_i \in X_i$, such that
$d_Z(\varphi_i(p_i), \varphi(p))< \delta$, then
for any fixed $\epsilon >0$ and almost every $t\in \R^k$
\be
\begin{split}
\label{eq:SIF-cont-0}
|\IFV&(\Slice(M_i, \rho_i,t)) - \IFV(\Slice(M_\infty, \rho_\infty, t))|  \\
& \leq (2 + \epsilon) \, d_\mathcal{F}(\Slice(M_i, \rho_i, t), \Slice(M_\infty, \rho_\infty, t)).
\end{split}
\ee
In particular,
\be
\begin{split}
\label{eq:SIF-cont-2}
\int & |\IFV(\Slice(M_i, \rho_i,t)) - \IFV(\Slice(M_\infty, \rho_\infty, t))| \, dt \\
& \leq (2 + \epsilon )\, d_\mathcal{F}(M_i, M_\infty) + 2 \delta_i ( \mass(T_\infty) + \mass(\partial T_\infty) ).
\end{split}
\ee
If $\lim_{i\to\infty}\delta_{i}=0$ and
\be
\lim_{i\to\infty} d^Z_{F}(\varphi_{i\#} T_i, \varphi_{\infty\#}T_\infty)=0
\ee
there is a subsequence such that for almost every $t \in \R^k$,
the interval filling volumes of slices converge,
\be \label{SIF-cont-1}
\lim_{i\to\infty} \IFV(\Slice(M_i,\rho_i,t))
=\IFV( \Slice(M_\infty,\rho_{\infty}, t)).
\ee
\end{prop}

\begin{proof}
Since
\be
\IFV(\Slice(M_i, \rho_i, t)) = \fillvol(\partial (\Slice(M_i, \rho_i, t) \times I_\epsilon)),
\ee
Theorem~\ref{fillvol-cont} and Proposition~\ref{interval-flat} imply that inequality~(\ref{eq:SIF-cont-0}) holds for almost every $t \in \R^k$. The estimate (\ref{eq:SIF-cont-2}) on the integrated quantity then follows from Proposition~\ref{limit-slices-more}.
\end{proof}

\vspace{.4cm}
\subsection{Limits of Points}

In this section we prove two statements about
Cauchy and converging sequences of points. 
Recall also Definition~\ref{point-conv}
and Definition~\ref{point-Cauchy}.
In \cite{SorWen1}, the second author and Stefan Wenger prove that
certain points in Gromov-Hausdorff limits of sequences were
also in the intrinsic flat limit by bounding Gromov's filling volumes of
spheres about points converging to these points.   In \cite{Sormani-AA} 
the second author removes the assumption of a Gromov-Hausdorff limit,
and studies whether or not Cauchy sequences of points converge
to points in the metric completion of the intrinsic flat limit.
The technique there involved uniformly
bounding the intrinsic flat distance of spheres away from $0$
and is not precise enough to distinguish between points in the
intrinsic flat limit and its metric completion.

Here we use sliced filling volumes and filling volumes to determine
when there is a limit point in the intrinsic flat limit space and not
just in its metric completion.   We prove Theorem~\ref{SF_k-in-set} 
which assumes one has a Cauchy sequence
and determines when the Cauchy sequence has a limit
point in the limit space not just in the metric completion of the
limit space.  Before we state and prove this theorem, we 
discuss the difficulties arising in identifying points in the limit space.

Recall that given an integral current space $(X,d,T)$, then
any isometric embedding
$\varphi: X\to Z$ with $Z$ complete 
maps $X$ isometrically onto $\set(\varphi_\#T)$
and extends to $\varphi: \bar{X} \to Z$ which maps $X$
isometrically onto 
\be
\set(\varphi_\#T)=\{z\in Z: \, \liminf_{r\to 0} \frac{||\varphi_\#T||(B(z,r)))}{r^m} > 0\,\}.
\ee
In Lemma~\ref{le:EstBySlices}] we proved
\be
\SF_k(p,r) \le \SF_0(p,r)=\fillvol(\partial S(p,r)) \le \mass(S(p,r))
\ee
In work of the second author with Wenger, continuity of the
filling volume is applied to prove that for certain sequences of
spaces a point in the Gromov-Hausdorff
limit lies in the intrinsic flat limit.  In generally it may be tricky 
to use the filling volume in this way:

\begin{rmrk}
Let $M = (X, d, T)$ be the $m$ dimensional integral current space of Example~\ref{concentric-spheres}.  It has a point $p \in X$ which is the center of
the concentric spheres.   This $p \in X=\set(T)$ because 
$\mass(B(p,r))\le Cr^m$.  However $\SF_0(p,r) = 0$ for $\lm$-a.e. small $r>0$. 
This shows that although $\SF_k(p,r)$ for $k = 0,\dots, m-1$ provide lower bounds for $\mass(S(p,r))$, in general these lower bounds could be far from sharp and thus
not be able to identify when a point lies in $X$.
\end{rmrk}

The next theorem concerns Cauchy sequences of points
in the sense of Definition~\ref{point-Cauchy}.

\begin{thm} \label{SF_k-in-set} 
Suppose $M_i$ are integral current spaces of dimension $m$
with $M_i$ converge to $M_\infty$ in the intrinsic flat sense.
Let $k\in \{0,1,2...,m-1\}$. 
Suppose that 
 $p_i\in M_i$ are Cauchy. If there is a function $c: \R_+ \to \R_+$ such that
\be
\label{eq:LimPointUnLowBound}
\frac{1}{r} \int_0^r \SF_k(p_i,r) \, d\tau \geq c(r),
\ee
then
$p_i$ converge to a point $p_\infty\in \bar{M}_\infty$.  
If in addition
\be
\label{eq:cDensLowEst2}
\liminf_{r \downarrow 0} \frac{c(r)}{r^m} > 0,
\ee 
then $p_i$ converge to a point $p_\infty\in M_\infty$.  
\end{thm}

\begin{proof}
By Lemma~\ref{le:EstBySlices}, it suffices to prove the case $k = 0$.
If $M_i \Fto M$ and the points $p_i \in M_i$ are Cauchy, there is a complete separable metric space $Z$ and isometric embeddings $\phi_i:X_i \to Z$ such that
\be
\lim_{i \to \infty} d_F^Z({\phi_i }_\# T_i , {\phi_\infty}_\# T_\infty) = 0,
\ee
and there exists a $z_\infty \in Z$ such that $\delta_i := d_Z(\phi_i(p_i) , z_\infty) \to 0$ in $Z$ as $i \to \infty$.

Again, by using Theorem~\ref{fillvol-cont}, Proposition~\ref{limit-balls}, and Lemma~\ref{lem-ann}, we find

\be
\begin{split}
\mass({\phi_\infty}_\# T_\infty \llcorner B_r(z_\infty) ) & \ge \frac{1}{r} \int_0^r \mass({\phi_\infty}_\# T_\infty \llcorner B_\tau(z_\infty) ) d\tau \\
&\ge \frac{1}{r} \int_0^r \fillvol_\infty( \partial ({\phi_\infty}_\# T_\infty \llcorner B_\tau(z_\infty) ) )d\tau \\
&\ge c(r) - \left(\frac{1}{r} + 1 \right) d_F^Z({\phi_i}_\# T_i, {\phi_\infty}_\#T_\infty) \\
&\quad - \frac{1}{r} \int_0^r \|T\|_\infty \left( \rho_{p_\infty}^{-1}(\tau - \delta_i, \tau+ \delta_i) \right) d\tau \\
&\ge c(r) - \left(\frac{1}{r} + 1 \right) d_F^Z({\phi_i}_\# T_i, {\phi_\infty}_\#T_\infty) - \frac{2\delta_i} {r} \mass(T_\infty).
\end{split}
\ee
We take the limit $i \to \infty$ and conclude that
\be
\mass({\phi_\infty}_\# T_\infty \llcorner B_r(z_\infty) ) \geq c(r).
\ee

Therefore, $M_\infty$ is not the $\mathbf{0}$ space, and $z_\infty \in \overline{{\set(\phi_\infty}_\# T_\infty)}$. 
Moreover, if inequality~(\ref{eq:cDensLowEst2}) is satisfied, $z_\infty \in \set({\phi_\infty}_\# T_\infty)$. 
\end{proof}

\begin{example}
It is quite possible for a Cauchy sequence of points to have more
than one limit as can be seen simply by taking the constant
sequence of integral current spaces, $S^1$, and noting that
due to the isometries, any point may be set up as the limit of
a Cauchy sequence of points.  One may also use isometries
of $S^1$ to relocate a Cauchy sequence so that the images are no longer 
Cauchy in $Z$.   This is also true of converging sequences in
the theory of Gromov-Hausdorff convergence.
\end{example}

\begin{rmrk}
Note that in order to apply the first part of Theorem~\ref{SF_k-in-set} it is sufficient to find a function $c(r) :\R_+ \to \R_+$ and a constant $r_0 > 0$, such that (\ref{eq:LimPointUnLowBound}) is satisfied for $0< r < r_0$. 
In particular, (\ref{eq:LimPointUnLowBound}) holds if $\SF_k(p_i,r) > \tilde{c}(r)$ for $\lm$-a.e. $0 < r < r_0$. 
Finally, (\ref{eq:cDensLowEst2}) holds if there exists a constant $C$ such that $\SF_k(p_i,r) \geq C r^m$ for $0 < r < r_0$.
\end{rmrk}

\vspace{.4cm}
\subsection{Bolzano-Weierstrass Theorems}

When one has a sequence of compact metric spaces converging
in the Gromov-Hausdorff sense to a compact metric space, and
one has a sequence of points in those metric spaces, then a subsequence
converges to a point in the Gromov-Hausdorff limit.  This is the
Gromov-Hausdorff Bolzano-Weierstrass Theorem and is an immediate
consequence of Gromov's Embedding Theorem which provides a
common metric space which is compact.   The immediate
restatement of the Gromov-Hausdorff Bolzano-Weierstrass Theorem
is not true when the spaces converge in the intrinsic flat sense
instead of the Gromov-Hausdorff sense.    This
can be seen in Ilmanen's Example with disappearing tips.   The key difficulty lies in the fact that, unlike Gromov's Embedding Theorem,
Theorem~\ref{converge}
does not provide a compact common metric space.

Nevertheless we are able to prove the following two
Bolzano-Weierstrass Theorems by assuming the limit space is
compact and preventing the points in the sequence from disappearing.   These theorems require bounding the Gromov's Filling Volumes
and Sliced Filling Volumes of spheres.   A simpler Bolzano-Weierstrass Theorem is proven by the second author 
in \cite{Sormani-AA}.  It requires uniformly
bounding the intrinsic flat distance of spheres away from $0$
but only produces a subsequence which converges in the metric
completion of the intrinsic flat limit space.

\begin{thm}\label{B-W-fillvol}
Suppose $M^m_i=(X_i, d_i, T_i)$ are $m$-dimensional integral current spaces converging in the intrinsic flat sense to a limit integral current space 
$M^m_\infty=(X_\infty, d_\infty, T_\infty)$.
 Suppose there exists a sequence
$p_i \in M_i$ and a function $c:\R_+ \to \R_+$ such that
\be 
\frac{1}{r} \int_0^r \fillvol(\partial S(p_i,\tau )) \, d\tau \ge  c(r) > 0.
\ee 
Then there exists a subsequence $p_{i_j}$
which converges to $p_\infty\in \bar{M}^m_\infty$. 
In particular, $M_\infty^m$ is nonzero.

If in addition 
\be
\label{eq:cDensLowEst}
\liminf_{r \downarrow 0} \frac{c(r)}{r^m} > 0,
\ee 
then the subsequence converges to a point $p_\infty \in M_\infty^m$.
\end{thm}

\begin{proof}
By Theorem~\ref{converge} there are a complete metric space $Z$ and isometric embeddings $\phi_i:X_i \to Z$ such that $d_F^Z({\phi_i}_\# T_i, {\phi_\infty}_\# T_\infty) \to 0$. 
Set $z_i := \phi_i(p_i)$.

Note that by Proposition~\ref{limit-balls}
\be
\int_0^{r_0} d_F^Z({\phi_i}_\# T_i \llcorner B_r(z_i),{\phi_\infty}_\# T_\infty \llcorner B_r(z_i) ) dr
 \leq 2 d_{\mathcal{F}}( M_i , M_\infty ).
\ee
Hence, by Theorem \ref{fillvol-cont} and our hypothesis, 
\be
\label{eq:BW-mass-bound-limit}
\begin{split}
\mass({\phi_\infty}_\# T_\infty \llcorner B_r(z_i) ) &\geq  \liminf_{i \to \infty } \frac{1}{r} \int_0^r \mass ( {\phi_\infty}_\# T_\infty \llcorner B_\tau(z_i) ) d\tau  \\
& \geq \liminf_{i \to \infty } \frac{1}{r} \int_0^r \fillvol_\infty( \partial ( {\phi_\infty}_\# T_\infty \llcorner B_\tau(z_i) )  ) d\tau \\
&\geq c(r).
\end{split}
\ee
In particular, $M_\infty$ is not the $\mathbf{0}$ space.

We claim that the $z_i$ have a Cauchy subsequence. 

We argue by contradiction:
If not, the metric space $(\{z_i\}, d_Z)$ is complete, and therefore not totally bounded, so that there is a $\delta > 0$ such that for a subsequence (without relabeling) $d_Z(z_i, z_j) > 4 \delta$ for $i \neq j$. 
As the balls $B_\delta(z_i)$ are mutually disjoint, this would mean that 
\be
\lim_{i \to \infty} \mass({\phi_\infty}_\# T_\infty \llcorner B_\delta(z_i)) = 0,
\ee
which contradicts (\ref{eq:BW-mass-bound-limit}).

Consequently, the $z_i$ have a Cauchy subsequence. 
We could conclude now by applying Theorem \ref{SF_k-in-set}, or by mimicking its proof.
With the established notation, however, we can easily finish the proof in an alternative fashion.

Since $Z$ is complete, this subsequence, also denoted $z_i$, converges to a limit $z_\infty \in Z$. 
Since for every $\tau < r$, for $i$ large enough $B_\tau(z_i) \subset B_r(z_\infty)$, 
by (\ref{eq:BW-mass-bound-limit}), for every $\tau < r$,
\be
\mass({\phi_\infty}_\# T_\infty \llcorner B_r(z_\infty)) \geq 
\limsup_{i\to\infty} \mass({\phi_\infty}_\# T_\infty \llcorner B_\tau(z_i)) \geq c(\tau).
\ee
Consequently, $z_\infty \in \overline{\set({\phi_\infty}_\# T_\infty) }$, and if inequality~(\ref{eq:cDensLowEst}) holds, even 
$z_\infty \in \set({\phi_\infty}_\# T_\infty)$.

\end{proof}

This theorem is a special case of the following more general
Bolzano-Weierstrass Theorem. 
The generalization follows from the special case and Lemma \ref{le:EstBySlices}.

\begin{thm}\label{B-W}
Suppose $M^m_i=(X_i, d_i, T_i)$ are $m$ dimensional integral current spaces 
converging in the intrinsic flat sense to
a limit integral current space 
$M^m_\infty=(X_\infty, d_\infty, T_\infty)$. 
Suppose there exists $k\in \{0,1,...,(m-1)\}$, $r_0>0$ and a sequence
$p_i \in M_i$ and a function $c: \R_+ \to \R_+$ such that
\be \label{eqn-B-W}
\frac{1}{r} \int_0^r \SF_k(p_i,\tau) d \tau \geq  c(r) > 0.
\ee 
Then there exists a subsequence $p_{i_j}$
which converges to $p_\infty\in \bar{M}^m_\infty$. 
In particular, $M_\infty^m$ is nonzero.

If in addition
\be
\label{eq:cDensLowEst2}
\liminf_{r \downarrow 0} \frac{c(r)}{r^m} > 0,
\ee 
then in fact $p_\infty \in M_\infty^m$.
\end{thm}

\section{Compactness Theorems}

In this section we complete the proofs
of our main two compactness theorems:
Theorem~\ref{tetra-compactness-2} and
Theorem~\ref{SF_k-compactness-2}.    These theorems
were announced by the second author in \cite{Sormani-tetra}.   
Both of these
theorems prove that certain sequences of spaces have
subsequences which converge in both the intrinsic flat and
the Gromov-Hausdorff sense to the same space.   Theorem~\ref{tetra-compactness-2}
is the Tetrahedral Compactness Theorem concerning sequences of
Riemannian manifolds satisfying the tetrahedral property.   It
was partially stated in the introduction.  It is a consequence of
Theorem~\ref{SF_k-compactness-2} which applies to integral current
spaces which have uniform lower bounds on the sliced filling
volumes of the form $\SF_k(p,r)\ge C_{SF}r^m$.

In prior work of the second author with Wenger \cite{SorWen1} 
another pair of compactness theorems was proven
providing subsequences of manifolds which converge both in the intrinsic flat
and Gromov Hausdorff sense to the same limit.
One theorem concerned noncollapsing
sequences of Riemannian manifolds with nonnegative Ricci curvature
(extending Gromov's Ricci Compactness Theorem \cite{Gromov-metric}).
The other concerned sequences of Riemannian manifolds with
a uniform linear contractibility function and a uniform upper bound on volume
(extending Greene-Petersen's Compactness Theorem \cite{Greene-Petersen}).    The techniques used in the proof of the Contractibility Function
Compactness Theorem in \cite{SorWen1} involve
the continuity of the filling volumes of balls.   Here we use the
continuity of sliced filling volumes in a similar way.   

The proofs of the theorems in this section are very short because the build upon
the prior theorems proven in this previous sections of this paper.   Those theorems
have applications in other situations and so it was important to prove them
separately rather than hiding those results within the proofs of these theorems.

\vspace{.4cm}
\subsection{Sliced Filling Compactness Theorem}

\begin{thm} \label{SF_k-compactness-2}
Given a sequence of $m$ dimensional
integral current spaces $M_i=(X_i, d_i, T_i)$
with $\mass(M_i) \le V_0$, $\mass(\partial M_i) \le A_0$,
$\diam(M_i)\le D_0$
and a uniform constant $C_{SF}>0$ such that
for some $k\in 0..(m-1)$ we have
\be
\SF_k(p,r) \ge C_{SF} r^m
\ee
then a subsequence of the
$M_i$ converge in the Gromov-Hausdorff sense and the
Intrinsic Flat sense to a nonzero integral current space $M_\infty$.
\end{thm}

\begin{proof}
By Theorem~\ref{SF_k-compactness}, we know a
subsequence $(X_i, d_i)$ has a Gromov-Hausdorff
limit $(Y, d_Y)$.  Thus by Gromov, there exists a 
common compact metric space $Z$ and isometric embeddings
$\varphi_i: X_i \to Z$, $\varphi: Y \to Z$, such
that $d_H^Z(\varphi(X_i), \varphi(Y)) \to 0$.   By Ambrosio-Kirchheim
Compactness Theorem, a subsequence of $\varphi_{1\#}T_i$
converges to $T_\infty \in \intcurr_m(Z)$.  Let
$M_\infty=(\set(T_\infty), d_Z, T_\infty)$.   

We need only show
$\varphi(Y)=\set(T_\infty)$.   
Let
$z_\infty \in Y$, and let $p_i\in X_i$ such that
$z_i=\varphi_i(p_i) \to z$.  By Theorem~\ref{SF_k-in-set},
we see that $z_\infty=\varphi(p_\infty)$.
\end{proof}

\vspace{.4cm}
\subsection{Tetrahedral Compactness Theorem}

\begin{thm}\label{tetra-compactness-2}
Given $r_0>0, \beta\in (0,1), C>0, V_0>0, A_0>0$.  If
a sequence of compact Riemannian manifolds, $M^m$, 
has $\vol(M^m) \le V_0$, 
$\diam(M^m) \le D_0$
and the $C, \beta$ (integral) tetrahedral
property for all balls of radius $\le r_0$, then a subsequence
converges in the Gromov-Hausdorff and Intrinsic Flat sense
to a nonzero integral current space.
\end{thm}

Here our manifolds do not have boundary.

\begin{proof}
The $C, \beta$ (integral) tetrahedral property implies
that there exists $C_{SF}>0$ such that
\be
\SF_{m-1}(p,r) \ge C_{SF} r^m.
\ee
Theorem~\ref{tetra-compactness} implies there
exists a uniform upper bound on diameter.
So we apply Theorem~\ref{SF_k-compactness-2}.
\end{proof}

\begin{rmrk}  \label{rmrk-scalar-cancellation}   
As a consequence of this theorem, we see that
there is no uniform tetrahedral property
on manifolds with positive scalar curvature even when the
volume of the balls are uniformly bounded below by that
of Euclidean balls.   In fact there exist a sequence of
such manifolds, $M_j^3$, whose intrinsic flat limit is $0$
described in \cite{SorWen1}.
\end{rmrk}

\section{Appendix: Gluing Integral Current Spaces}

Here we define how to glue two integral current spaces
along an isometric boundary to produce a new integral
current space.    This is applied to prove Theorem~\ref{fillvol-cont}.

\begin{thm}\label{thm-gluing}
Given two integral current spaces, $M_i=(X_i, d_i, T_i)$,
with $S_i, S'_i \in \intcurr_{m-1}(\bar{X}_i)$ such that
\be
\partial T_i = S_i + S'_i \textrm{ with } 
\spt(S_i) \subset \spt(\partial T_i)
\ee
and a current reversing distance preserving bijection 
\be
F: \spt(S_1) \to \spt(S_2) \textrm{ such that }  F_\# S_1\,=\,-\,S_2,
\ee
we define the glued integral current space
\be
M=M_1 \disjointunion_F M_2 =(X,d,T)
\ee
such that there are distance preserving maps
\be\label{d-f-1}
f_i: \bar{X}_i \to Y 
\ee
where $Y$ is the glued metric space:
\be
Y = \bar{X}_1 \disjointunion (\bar{X}_2 \setminus \spt(S_2)),
\ee
such that 
\be
f_2 \circ F =f_1 \textrm{ when restricted to } \spt(S_1),
\ee
and such that 
\be \label {adding-Ti}
T= f_{1\#}T_1 + f_{2\#}T_2 \textrm{ and } X=\set T \subset Y.
\ee
Since $f_i$ are distance preserving,
\be
\mass(M)\le \mass(M_1) + \mass(M_2).
\ee
\end{thm}

Note that it is possible that the glued integral current space, $X$,
is a proper subset of the glued metric space, $Y$.
In fact the glued integral current space could be the $0$
space (see Example~\ref{glued-0}).

Before we prove this theorem we apply it to
prove the following useful corollary which glues
integral current spaces together in order to provide an 
estimate on the filling volume.  This corollary is applied
to prove Theorem~\ref{fillvol-cont}.  See Definition~\ref{defn-filling-volume} for the definition of filling volume being applied here.

\begin{cor}\label{cor-glue-fillvol}
Given two integral current spaces, $M_i=(X_i, d_i, T_i)$,
with $S_{i}, S'_i \in \intcurr_{m-1}(\bar{X}_i)$ such that
\be
\partial T_i = S_i + S'_i
\ee
and a current reversing distance preserving map 
\be
F: \spt(S_1) \to \spt(S_2) \textrm{ such that }  F_\# S_1\,=\,-\,S_2.
\ee
If $S_2'=0$ then
\be
\fillvol(N) \le \mass(M_1) + \mass(M_2).
\ee
where $N=(\set(S_1'), d_1, S_1')$.
\end{cor}

\begin{proof}
By Theorem~\ref{thm-gluing} we have
\begin{eqnarray}
\partial T &=&  f_{1\#}\partial T_1 + f_{2\#}\partial T_2\\
&=& f_{1\#} S_1 + f_{1\#}S'_1 + f_{2\#}S_2 + f_{2\#} S'_2\\
&=& f_{2\#} F_\# S_1  + f_{2\#}S_2 + f_{1\#}S'_1+ f_{2\#} S'_2\\
&=&  f_{1\#}S'_1+ f_{2\#} S'_2.
\end{eqnarray}
Thus if $S_2'=0$ then $\partial T= f_{1\#}S'_1$ and so
by (\ref{d-f-1})
we have a current preserving isometry
\be
f_1: N=(\set(S_1'), d_1, S_1') \to 
\partial M=(\set (\partial T) , d, \partial T).
\ee
Thus
\be
\fillvol(N) = \fillvol(\partial M)
 \le \mass(T)\le \mass(M_1) + \mass(M_2).
\ee
\end{proof}

We now prove Theorem~\ref{thm-gluing}:

\begin{proof}

First we prove the glued metric space, $Y$, is well defined.
This is included for completeness of exposition and because
there are different methods used for gluing metric spaces. 

Let $d: Y\times Y \to [0, \infty)$ be symmetric such that
\be
d(x,y)= \left\{
        \begin{array}{ll}
            d_1(x,y) & \quad x,y\in \bar{X}_1 \\
            d_2(x,y) & \quad x,y\in \bar{X}_2 \setminus \spt(S_2)\\
            \inf\{d_1(x,w)+d_2(F(w),y):\, w\in \spt(S_1) \}& \quad 
            x\in \bar{X}_1,\,\, y\in \bar{X}_2\setminus \spt(S_2)
        \end{array}
    \right.
    \ee
Observe that $d(x,y)\ge 0$.   

Observe that $d(x,y)=0$
implies that $x=y$ if both $x,y\in \bar{X}_1$ or
both $x,y\in \bar{X}_2 \setminus \spt(S_2)$ since $d_1$
and $d_2$ are metrics.   The third case, where
$x\in \bar{X}_1$ and $y\in \bar{X}_2 \setminus \spt(S_2)$
cannot produce $d(x,y)=0$.   If it did occur then
we would have
$w_j \in \spt(S_1)$ such that
\be
0=d(x,y) =  \lim_{j\to \infty} \, d_1(x,w_j)+d_2(F(w_j),y)
\ee
so 
\be
\lim_{j\to \infty} \, d_2(F(w_j),y)=0
\ee
and $y$ is in the closure of the image of $F$, which means
$y \in \spt(S_2)$ which is a contradiction.

To see the triangle inequality observe that if $x_i \in \bar{X}_1$
and $y_i \in \bar{X}_2\setminus \spt(S_2)$ then 
\begin{eqnarray*}
d(x_1, x_3)+d(x_3,x_2)&=&d_1(x_1, x_3)+d_1(x_3,x_2)\ge
d_1(x_1,x_2)=d(x_1,x_2)\\
d(x_1,y_3)+ d(y_3, x_2)&=&\inf\{d_1(x_1,w)+d_2(F(w),y_3)+d_1(x_2,w')+d_2(F(w'),y_3):\, w,w'\in \spt(S_1)\}\\
&\ge & \inf\{d_1(x_1,w)+d_2(F(w),F(w'))+d_1(x_2,w'):\, w,w'\in \spt(S_1)\}\\
&= &
 \inf\{d_1(x_1,w)+d_1(w,w')+d_1(x_2,w'):\, w,w'\in \spt(S_1)\}\\
&\ge & d_1(x_1, x_2)=d(x_1,x_2)\\
d(y_1, y_3)+d(y_3,y_2)&=&d_2(y_1, y_3)+d_2(y_3,y_2)\ge
d_2(y_1,y_2)=d(y_1,y_2)\\
d(y_1,x_3)+ d(x_3, y_2)&=&\inf\{d_1(x_3,w)+d_2(F(w),y_1)+d_1(x_3,w')+d_2(F(w'),y_2):\, w,w'\in \spt(S_1)\}\\
&\ge&\inf\{d_1(w',w)+d_2(F(w),y_1)+d_2(F(w'),y_2):\, w,w'\in \spt(S_1)\}\\
&=&\inf\{d_2(F(w'),F(w))+d_2(F(w),y_1)+d_2(F(w'),y_2):\, w,w'\in \spt(S_1)\}\\
&\ge& d_2(y_1,y_2)= d(y_1,y_2)\\
d(x_1, x_3)+d(x_3,y_2) 
&=&\inf\{d_1(x_1,x_3)+d_1(x_3,w)+d_2(F(w),y_2):\, w\in \spt(S_1)\}\\
 &\ge&\inf\{d_1(x_1,w)+d_2(F(w),y_2):\, w\in \spt(S_1)\}\\
&=& d(x_1, y_2) \\
d(x_1, y_3)+d(y_3,y_2) 
&=&\inf\{d_1(x_1,w)+d_2(F(w),y_3) + d_2(y_3, y_2):\, w\in \spt(S_1)\}\\
&\ge&\inf\{d_1(x_1,w)+d_2(F(w), y_2):\, w\in \spt(S_1)\}\\
&=& d(x_1,y_2).
\end{eqnarray*}
Thus $d$ is a metric on $\bar{X}$.

By the definition of $d$ immediately we have
natural identification maps
\begin{eqnarray}
f_1: \bar{X}_1 \to \bar{X} &\textrm{ such that }& f_1(x)=x \\
f_2: \bar{X}_2\setminus \spt(S_2) \to X
&\textrm{ such that }& f_2(x)=x 
\end{eqnarray}
that are immediately distance preserving.   
We define
\be
f_2: \spt(S_2) \to Y \textrm{ such that } f_2(y)=f_1(F^{-1}(y)).
\ee
Then for $y_1,y_2\in \spt(S_2)$
and $y_3\in \bar{X}_2\setminus \spt(S_2)$ we have
\begin{eqnarray*}
d(f_2(y_1), f_2(y_2))&=&
d\left(f_1(F^{-1}(y_1)), f_1(F^{-1}(y_2))\right)\\
&=& d_1\left(F^{-1}(y_1), F^{-1}(y_2)\right)=d_2(y_1, y_2)\\
d(f_2(y_1), f_2(y_3))&=&
d\left(f_1(F^{-1}(y_1)), f_2(y_3)\right)\\
&=&  \inf\left\{
d_1\left(F^{-1}(y_1),w)+d_2(F(w),f_2(y_3)\right)
:\, w\in \spt(S_1)\right\}\\
&=&  \inf\left\{d_2(y_1,F(w))+d_2(F(w),f_2(y_3))
:\, w\in \spt(S_1)\right\}\\
&=& d_2(y_1,y_3).
\end{eqnarray*}
Thus $f_2: \bar{X}_2 \to Y$ is also distance preserving.

Since $f_i$ are distance preserving, they are Lipschitz, and
so $f_{i\#}T_i$
is well defined and 
$
\mass(f_{i\#}T_i)=\mass(T_i).
$
Defining $M=(X,d,T)$ as in (\ref{adding-Ti}), we have
\be
\mass(M)=\mass(T) \le 
\mass( f_{1\#}T_1) + \mass( f_{2\#}T_2 )
= \mass(T_1) + \mass(T_2).
\ee
\end{proof}

\begin{example}\label{glued-0}
Let $M_1=(X_1, d_1, T_1)$ where $T_1$ is a two dimensional
integral current in $X_1 \subset [0,1]^2$, endowed with the standard Euclidean metric $d_1(x,y)=|x-y|$.   
Let $M_2=(X_2,d_2,T_2)$ where $T_2=-T_1$ and $d_2=d_1$,
so $X_1=X_2$.  Then one can glue $M_1$ to
$M_2$ along their boundary.

Suppose that $\partial T_1$ is dense in $X_1$
so that $\set(T) \subset \spt(\partial T)$.  
If we glue $M_1$ to $M_2$ along their boundary
they are completely glued together with opposite orientation
and we obtain the $\bf{0}$ space.   More precisely,
we have
\be
Y=\spt(T)=\spt(\partial T) = \bar{X}_1
\ee
and $f_i: [0,1]\to [0,1]$ are identity maps.   So
\be
f_{1\#}T_1= f_{2\#}T_1= f_{2\#}(-T_2)=-f_{2\#}T_2.
\ee
Thus $T=f_{1\#}T_1+f_{2\#}T_2=0$.
\end{example}

\bibliographystyle{plain}
\bibliography{prop}
\end{document}